\documentclass[11pt]{amsart}
\usepackage{amssymb}
\usepackage[all]{xy}

\usepackage{url}
\usepackage{verbatim}
\usepackage{comment}
\usepackage{geometry}
\usepackage{mathtools}
\usepackage{rotating}
\usepackage{tikz,graphicx}
\tikzstyle{vertex}=[circle, draw, inner sep=0pt, minimum size=4pt]
\newcommand{\vertex}{\node[vertex]}

\usepackage{amsmath}
\usepackage{tcolorbox}
\usepackage{amssymb}
\usepackage{amsthm}
\usepackage{lastpage}
\usepackage{fancyhdr}
\usepackage{accents} 

\usepackage{amsfonts}
\usepackage{dsfont}
\usepackage{graphicx}
\usepackage{mathrsfs}
\usepackage{mathtools}
\usepackage{multicol} 
\usepackage{relsize}

\newtheorem{theorem}{Theorem}[section]
\newtheorem{lemma}[theorem]{Lemma}
\newtheorem{proposition}[theorem]{Proposition}
\newtheorem{corollary}[theorem]{Corollary}
\theoremstyle{definition}
\newtheorem{definition}[theorem]{Definition}
\newtheorem{example}[theorem]{Example}

\newtheorem{remark}[theorem]{Remark}

\newcommand\Z{\mathbb{Z}}
\newcommand\R{\mathbb{R}}

\DeclareMathOperator{\vol}{vol}

\renewcommand{\b}{\beta}
\renewcommand{\phi}{\varphi}

\newcommand{\ra}{\rightarrow}

\newcommand{\F}{\mathcal{F}} 
\renewcommand{\a}{\mathbf{a}}

\newcommand{\outd}{\mathrm{outd}}
\newcommand{\ind}{\mathrm{ind}}
\newcommand{\w}{\mathbf{w}}

\title{Volume inequalities for flow polytopes of full directed acyclic graphs}

\author{Benjamin Braun}
\address{Department of Mathematics\\
         University of Kentucky\\
\url{https://sites.google.com/view/braunmath/}}
\email{benjamin.braun@uky.edu}

\author{James Ford McElroy}
\address{Department of Mathematics\\
         University of Kentucky\\
\url{https://math.as.uky.edu/users/jfmc230}}
\email{jfmcelroy@uky.edu}

\date{28 May 2024}

\thanks{Both authors were partially supported by the US National Science Foundation award DMS-1953785.}

\begin{document}

\begin{abstract}
  Given a finite directed acyclic graph, the space of non-negative unit flows is a lattice polytope called the flow polytope of the graph.
  We consider the volumes of flow polytopes for directed acyclic graphs on $n+1$ vertices with a fixed degree sequence, with a focus on graphs having in- and out-degree two on every internal vertex.
  When the out-degree of the source is three and the number of vertices is fixed, we prove that there is an interchange operation on the edge set of these graphs that induces a partial order on the graphs isomorphic to a Boolean algebra.
  Further, we prove that as we move up through this partial order, the volumes of the corresponding flow polytopes weakly decrease.
  Finally, we show that each such graph is strongly planar and we provide an alternative interpretation of our results in the context of linear extensions for posets that are bipartite non-crossing trees.
\end{abstract}

\maketitle

\section{Introduction} \label{sec:intro}

Given a finite directed acyclic graph (DAG) $G$, the space of non-negative flows on $G$ of strength one is a lattice polytope denoted by $\F_1(G)$.
While flow polytopes for DAGs can be defined for arbitrary netflow vectors, strength one flows that conserve flow at inner vertices have remained an important and central object of study in geometric and algebraic combinatorics~\cite{baldoni-vergne,caracolvolume,combflowmodel,DKK,generalpitmanstanley,cyclicorderflow,integerpointsflow,lidskii,gtflow,flowdiagonalharmonics,meszaros-morales,meszaros-morales-striker,morrisidentityflow,vonbell2024triangulations,framedtriangulations}.
A central problem is to determine the (normalized) volume of $\F_1(G)$, and to relate volumes of flow polytopes to structural qualities of the DAGs in question.

In this work, we consider the following general problem.
Consider the family of DAGs on a fixed number of vertices with fixed in- and out-degree sequences.
Which DAGs in this family maximize or minimize flow polytope volume?
What relationships, if any, are there between the DAGs in this family and their flow polytope volumes?
Note that if a DAG has a vertex with either in-degree or out-degree equal to $1$, then contracting the single corresponding edge yields a DAG with an integrally equivalent flow polytope.
Thus, the first non-trivial case to consider is when the in- and out-degree sequence for inner (i.e., non-sink and non-source) vertices is $(2,\ldots,2)$.
When the degree of the source and sink are both $2$, this family consists of only one DAG on $n+1$ vertices, with corresponding flow polytope given by a unit cube.

In this paper, we consider the DAGs having source vertex and sink vertex both of degree $3$, where all other in- and out-degrees are $2$. 
We show that DAGs satisfying this condition are in bijection with binary strings.
Further, crossing nested edges in such a DAG corresponds to changing $0$s and $1$s in the associated binary string.
The main result of this work is to show that the operation of crossing nested edges weakly decreases the volume of the associated flow polytopes; thus, there is a Boolean algebra structure on the set of DAGs on $n+1$ vertices with this degree condition where the volumes of flow polytopes are reverse-ordered from the partial order on the Boolean algebra.

Further, the DAGs we consider are planar.
Due to M\'{e}sz\'{a}ros, Morales, and Striker~\cite{meszaros-morales-striker}, any result about volumes of flow polytopes arising from planar DAGs (along with some other conditions) can be translated to a result about linear extensions of posets arising from the duals of the DAGs.
Thus, our results regarding volume inequalities for flow polytopes translate into inequalities for linear extensions of posets.

The remainder of this paper is structured as follows.
In Section~\ref{sec:background}, we provide background regarding flow polytopes and their volumes.
We prove in Section~\ref{sec:boolean} results regarding Boolean algebra poset structures on our families of DAGs of interest.
In Section~\ref{sec:volume}, we introduce the tool of flow decomposition trees in order to prove Theorem~\ref{thm:MainThm}, which states that the values of the volumes of flow polytopes for DAGs under the Boolean partial order obtained in Corollary~\ref{cor:3coverage} respect the dual partial order.
Finally, in Section~\ref{sec:LinearExtensions} we recast our results about flow polytopes in the context of linear extensions for posets.

\section{Background} \label{sec:background}

Let $G=(V,E)$ be a directed acyclic graph (DAG) on vertices $V=[n+1]$, such that every edge $(i,j)$ in $G$ satisfies that $i<j$, and such that $G$ has a single source $1$ and a single sink $n+1$.
We use the notation $\a=(a_1,\ldots,a_n,-\sum_i a_i)\in \Z_{\geq0}^{n+1}$ for a netflow vector, we denote the (multi)set $S_G:=\{e_i-e_j \ |\ (i,j)\in E\}$, and we denote by $M_G$ the matrix with columns given by the vectors of $S_G$. 

\begin{definition}\label{def:flowpolytope}
  Given a DAG $G$ and netflow vector $\a$, we define the \emph{flow polytope} of $G$ to be
  \[
    \F_G(\a) := \{f_G\in\R^{E}_{\geq0} \ |\ M_G f_G = \a\} \, .
  \]
  \end{definition}
We think of $f_G$ as a flow on $G$, that is, an assignment of nonnegative flow values to each edge satisfying the netflow requirements at each vertex given by the netflow vector $\a$. 
An important choice of netflow vector is $\a=(1,0,\ldots,0,-1)$, and we use the special symbol
\[ 
\F_1(G) := \F_G(1,0\ldots,0,-1)
\]
to denote the corresponding flow polytope.
Flow polytopes are lattice polytopes, in that their vertices are always elements of the integer lattice $\Z^{E}$.

Computing the volume of an arbitrary lattice polytope is a challenging problem.
There is a closed formula for flow polytope volumes, using Kostant partition functions, due to Lidskii~\cite{lidskii} and Baldoni and Vergne~\cite{baldoni-vergne}. 
The Kostant partition function $K_G$ evaluated at the vector $\a\in\Z^{n+1}$ is defined as 
\[
	K_G(\a) = \#\left\{ (f(e))_{e\in E(G)} | \sum_{e\in E(G)} f(e)\alpha(e)=\a, \ \text{and} f(e)\in\Z_{\geq0} \right\},
\]
where $\{\{\alpha(e)\}\}_{e\in E(G)}$ is the multiset of positive type $A_n$ roots corresponding with the multiset $E(G)$.
That is, the Kostant partition function counts the number of ways to write the vector $\a$ as a $\mathbb{N}$-linear combination of vectors $\alpha(e)$ for $e\in E(G)$.
The Lidskii formula states that the (normalized) volume of $\F_G(\a)$ is 
\begin{align*}
    \sum_{\stackrel{(m_1,\ldots,m_n)\geq (\outd_1,\ldots,\outd_n)}{\sum m_i= |E|}}   \binom{|E|-|V|+1}{m_1-1,\ldots,m_n-1} a_1^{m_1-1}\cdots a_n^{m_n-1} K_G(m_1-\outd_1,\ldots,m_n-\outd_n,0),
\end{align*}
where we sum over weak compositions using dominance order. 
Note that the Kostant partition function is also difficult to compute in general, and thus the efficacy of this formula depends on the DAG under consideration.
In the case of the netflow vector $\a=(1,0,\ldots,0,-1)$, the Lidskii formula reduces to 
\begin{equation} \label{eq:LidskiiSimple}
    \vol \F_1(G) = K_G(|E|-|V|+2-\outd_1,1-\outd_2,\ldots,1-\outd_n,0)
\end{equation}
which is equivalent to the following, via flow reversal~\cite{meszaros-morales},
\begin{align*}
    \vol \F_1(G) = K_G(0,\ind_2-1,\ldots,\ind_n-1,|V|-|E|-2+\ind_{n+1}).
\end{align*}
Geometrically, this means that the volume of $\F_1(G)$ is equal to the number of lattice points in the related flow polytope $\F_G(0,\ind_2-1,\ldots,\ind_n-1,|V|-|E|-2+\ind_{n+1})$ or equivalently, the integer flows on $G$ with netflow vector $(0,\ind_2-1,\ldots,\ind_n-1,|V|-|E|-2+\ind_{n+1})$.
It is worth noting that an alternative method exists to compute flow polytope volumes for DAGs containing a directed Hamiltonian path, one which relies on combinatorics of $G$-cyclic orders~\cite[Theorem 4.8]{cyclicorderflow}.

A family of DAGs of particular interest is that of full DAGs and a generalization which we call $r$-full DAGs.
We are interested in $r$-full DAGs because von Bell, Braun, Bruegge, Hanely, Peterson, Serhiyenko and Yip proved~\cite[Proposition 6.15]{vonbell2024triangulations} that the flow polytope $\F_1(G)$ for an $r$-full DAG $G$ is Gorenstein.

\begin{definition} \label{def:r-full}
    A DAG $G$ is \emph{$r$-full} if every inner vertex has in- and out-degree $r$. 
    In the case where $r=2$, we say that $G$ is \emph{full}. 
\end{definition}

\begin{example} \label{ex:r-full}
	Pictured below is a $2$-full (i.e., full) DAG and a $3$-full DAG, respectively. 
	
      
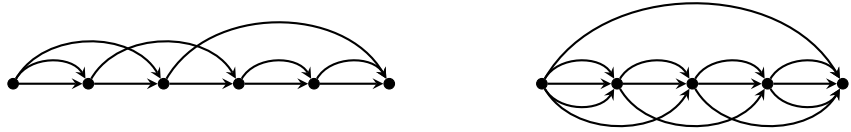
\begin{center}
     \begin{tikzpicture}
   
     \begin{scope}[xshift=0, yshift=0, scale=1]
   
	\vertex[fill](a1) at (1,0) {};
	\vertex[fill](a2) at (2,0) {};
	\vertex[fill](a3) at (3,0) {};
	\vertex[fill](a4) at (4,0) {};
	\vertex[fill](a5) at (5,0) {};
	\vertex[fill](a6) at (6,0) {};
	\draw[-stealth, thick] (a1) to[out=0,in=180] (a2);
	 \draw[-stealth, thick] (a1) to[out=60,in=120] (a2);
	 \draw[-stealth, thick] (a1) to[out=60,in=120] (a3);
	 \draw[-stealth, thick] (a2) to[out=0,in=180] (a3);
	\draw[-stealth, thick] (a2) to[out=60,in=120] (a4);
	 \draw[-stealth, thick] (a3) to[out=0,in=180] (a4);
	\draw[-stealth, thick] (a3) to[out=60,in=120] (a6);
	 \draw[-stealth, thick] (a4) to[out=0,in=180] (a5);
	\draw[-stealth, thick] (a4) to[out=60,in=120] (a5);
	 \draw[-stealth, thick] (a5) to[out=0,in=180] (a6);
	\draw[-stealth, thick] (a5) to[out=60,in=120] (a6);
   
     \end{scope}     
     
     \begin{scope}[xshift=200, yshift=0, scale =1]

	\vertex[fill](a1) at (1,0) {};
	\vertex[fill](a2) at (2,0) {};
	\vertex[fill](a3) at (3,0) {};
	\vertex[fill](a4) at (4,0) {};
        \vertex[fill](a5) at (5,0) {};
	\draw[-stealth, thick] (a1) to[out=0,in=180] (a2);
        \draw[-stealth, thick] (a1) to[out=60,in=120] (a2);
        \draw[-stealth, thick] (a1) to[out=-60,in=-120] (a2);
        \draw[-stealth, thick] (a1) to[out=-60,in=-120] (a3);
        \draw[-stealth, thick] (a1) to[out=60,in=120] (a5);
        \draw[-stealth, thick] (a2) to[out=0,in=180] (a3);
        \draw[-stealth, thick] (a2) to[out=60,in=120] (a3);
        \draw[-stealth, thick] (a2) to[out=-60,in=-120] (a4);
        \draw[-stealth, thick] (a3) to[out=0,in=180] (a4);
        \draw[-stealth, thick] (a3) to[out=60,in=120] (a4);
        \draw[-stealth, thick] (a3) to[out=-60,in=-120] (a5);
        \draw[-stealth, thick] (a4) to[out=0,in=180] (a5);
        \draw[-stealth, thick] (a4) to[out=60,in=120] (a5);
        \draw[-stealth, thick] (a4) to[out=-60,in=-120] (a5);
    
	\end{scope}
        
     \end{tikzpicture}
     \end{center}
         
\end{example}

\begin{definition} \label{def:Fnk}
    Let $\mathscr{F}_{n,k}$ be the set of full DAGs on vertices $[n+1]$ with $\outd_1=k$. That is, DAGs with out-degree sequence $(k,2,2,\ldots,2,0)$ and in-degree sequence $(0,2,2,\ldots,2,k)$.
  \end{definition}

  Note that every full DAG with a unique source and sink is an element of $\mathscr{F}_{n,k}$ for some $k$.

\begin{proposition} \label{prop:edgecount}
    A DAG $G\in\mathscr{F}_{n,k}$ has $n+1$ vertices and $2n+k-2$ edges.
\end{proposition}

\begin{proof}
    The vertex count is true by assumption. 
The edge count comes from the fact that the number of edges in a DAG is equal to the sum of the out-degrees (or equivalently, the sum of the in-degrees). 
So
\begin{align*}
    |E(G)| &= \sum_{i\in V(G)} \outd_i = k + 2(n-1) = 2n +k -2 
\end{align*}
\end{proof}

We next define the length $n+1$ vector
  \[ 
  \w_n := (0,1,\ldots,1,-(n-1))\, ,
\]
which plays the following special role in computing volumes of flow polytopes for full DAGs.

\begin{proposition} \label{prop:VolumeSpecialization}
    Suppose $G\in\mathscr{F}_{n,3}$. 
    Then
    \[ 
    \vol \F_1(G) = K_G(\w_n). 
    \]
    In other words, computing the volume of the flow polytope $\F_1(G)$ is equivalent to counting lattice points in the flow polytope $\F_G(\w_n)$. 
\end{proposition}

\begin{proof}
    This follows directly from applying Definition~\ref{def:Fnk} for $k=3$ and the edge count from Proposition~\ref{prop:edgecount} to Equation~\eqref{eq:LidskiiSimple}.
\end{proof}

Following Karrer and Newman~\cite{randomDAG2009}, we use the following method to construct DAGs of a fixed degree sequence.
For each vertex of a DAG $G$ with a fixed degree sequence, we imagine ``stubs'' attached to the vertex, with as many outgoing stubs as the out-degree and as many incoming stubs as the in-degree.
To create a DAG we match outgoing stubs with incoming stubs, where we only create an edge $(i,j)$ if $i<j$.
This construction method will be used in Section~\ref{sec:boolean}.

 \section{Poset Structure} \label{sec:boolean}

 In this section, we show first that full DAGs on $n+1$ vertices with $\outd_1=3$ are in bijection with binary strings of length $n-2$. 
 Then we show that the Boolean partial order induced by this bijection agrees with an edge operation on $\mathscr{F}_{n,3}$.

 \begin{lemma} \label{lemma:spinetc}
     Every DAG $G\in\mathscr{F}_{n,3}$ has the following edges:
     \[ 
     \{(i,i+1) \ |\ i\in[n] \} \cup \{(1,2),(n,n+1)\} \subset E(G). 
     \]
 \end{lemma}

 \begin{proof}
     Let $G\in\mathscr{F}_{n,3}$. 
     Note that $\ind_2=2$. 
     Since $(i,j)\in E(G)$ implies $i<j$, all edges going in to $2$ must start at $1$. 
     Thus there must be two copies of the edge $(1,2)$. 
     A similar argument can be made for two copies of the edge $(n,n+1)$ using the $\outd_n=2$. 
   
     Now consider a $2\leq i \leq n-1$. 
     Since $G$ is full, $\ind_i=2$. 
     Once again, both edges must begin at a vertex smaller than $i$. 
     But this is also true for every vertex \textit{before} $i$. 
     If we look at the available pool of out-degrees at non-consecutive vertices with which to connect the two in-stubs at $i$, we have
     \begin{align*}
          \sum_{j=1}^{i-1} \outd_j - \sum_{j=1}^i \ind_j &= (3 + 2(i-2)) - 2(i-1) = 1.
     \end{align*}
     So at most one of the in-stubs at $i$ can be connected to an out-stub \textit{before} $i-1$, forcing at least one edge $(i,i-1)$. 
 \end{proof}

 \begin{definition} \label{def:spine}
     Let $G$ be a DAG with vertices $[n+1]$. 
     If $(i,i+1)\in E(G)$ for all $i\in[n]$, we say $G$ has a \emph{spine}. 
     Further, if $G$ has a spine, we call the edge set $\{(i,i+1) \ |\  i\in [n]\}$ the spine. 
     Throughout this work, we label the spine edges as $e_i=(i,i+1)$ to distinguish them from other (multi)edges between consecutively-labeled vertices. 
 \end{definition}

 \begin{corollary} \label{cor:spine}
     Every DAG in $\mathscr{F}_{n,3}$ has a spine.
 \end{corollary}

 \begin{theorem} \label{thm:bijection}
     The set $\mathscr{F}_{n,3}$ is in bijection with binary strings of length $n-2$, i.e., the elements of $(\Z/2\Z)^{n-2}$. 
 \end{theorem}

 \begin{proof}
     We construct a bijection directly. 

 First we construct a map $\lambda:(\Z/2\Z)^{n-2} \ra \mathscr{F}_{n,3}$. 
 Let $\b=b_1 b_2 \cdots b_{n-2}\in (\Z/2\Z)^{n-2}$. 
 Let $J_\b$ be the list of indices in which $\b$ contains a 1. Then $J_\b=\{i\in[n-2] \ |\ b_i=1\}=\{j_1<j_2<\cdots <j_r\}$, for some $0\leq r\leq n-2$. 
 Now define $\lambda(\b)=G_\b$, where $G_\b$ is a DAG with vertices $V(G_\b)=[n+1]$ and edges 
 \begin{align*}
     E(G_\b) &= \{e_i \ |\ i\in[n] \} \cup \{(1,2),(n,n+1)\} \\
     &\cup \{(l+1,l+2) \ |\ b_l=0, \text{for } 1\leq l\leq n-2\} \\
     &\cup \{(1,j_1+2),(j_1+1,j_2+2),\ldots,(j_{r-1}+1,j_r+2),(j_r+1,n+1)\}.
 \end{align*}
 These three subsets of $E(G_\b)$ correspond with the following structures, respectively:
 \begin{enumerate}
     \item[(i)] the edges described in Lemma~\ref{lemma:spinetc};
     \item[(ii)] a second edge from vertex $l+1$ to $l+2$ for every $0$ at index $l$ in $\b$; and
     \item[(iii)] an edge connecting remaining consecutive pairs of out-stubs and in-stubs.
 \end{enumerate}

 Next we show $G_\b$ is full with $\outd_1=3$. 
 By construction, $G_\b$ is a DAG with $V(G)=[n+1]$. 
 Our construction also ensures that 
 \[
 \outd_1=2+0+1=3
 \]  
 and
 \begin{align*}
     \ind_2 &= 2+0+0 = 2.
 \end{align*}
 Further, for $3\leq i\leq n$,
 \begin{align*}
     \ind_i &= \left\{\begin{array}{ll}
         1+1+0, & b_{i-2} = 0   \\
         1+0+1, & b_{i-2} = 1 
     \end{array}\right. 
     = 2
 \end{align*}
 and, for $2\leq i\leq n$,
 \begin{align*}
     \outd_i &= \left\{\begin{array}{ll}
         1+1+0, & b_{i-1} = 0  \\
         1+0+1, & b_{i-1} = 1 
     \end{array}\right. 
     = 2.
 \end{align*}

 Thus $G_\b$ has the required degree sequence to be in $\mathscr{F}_{n,3}$. 

 Now we show $\lambda$ is injective. 
 Suppose $\b\neq\b'\in (\Z/2\Z)^{n-2}$. 
 Then there exists an index $l\in[n-2]$ for which $b_l\neq b'_l$. 
 Without loss of generality we may assume $b_l=0$ and $b'_l=1$. 
 Then according to our construction, we have 
 \begin{align*}
     \{(l+1,l+2),(l+1,l+2)\} &\subset E(G_\b) \text{ and }\\
     \{(l+1,l+2),(l+1,k)\} &\subset E(G_{\b'})\, ,
 \end{align*}
 where 
 \begin{align*}
     l+2 < k = \left\{\begin{array}{ll}
	 n+1, & b'_l \text{ is the last } 1 \text{ in } \b' \\
	 l+m+2,	 & b'_{l+m}=1 \text{ is the next } 1 \text{ in } \b' \text{ for } m\geq 1
     \end{array}\right. .
 \end{align*}

 Since both $G_\b$ and $G_{\b'}$ are full and $l+1$ is an interior vertex, we know there are no other edges out of vertex $l+1$ in $G_{\b'}$ or $G_\b$. 
 Thus $G_\b \neq G_{\b'}$.

 Finally we show that $\lambda$ is surjective. 
 Since $\lambda$ is an injection of finite sets, it suffices to show that $|\mathscr{F}_{n,k}|\leq |(\Z/2\Z)^{n-2}| = 2^{n-2}$. 
 By Proposition~\ref{prop:edgecount}, we have to decide how to allocate $2n+1$ directed edges onto $n+1$ vertices. 
 By Lemma~\ref{lemma:spinetc}, $n+2$ edges are forced by the spine and initial/final vertices. 
 Thus we have 
 \[ 
 2n+1 - (n+2) = n-1 
 \]
 edges remaining. 
 By interpreting the remaining edges as connections between unused stubs on vertices~\cite{randomDAG2009}, we can allocate an edge by pairing an out-going stub with an in-going stub, as long as the in-going stub is at a higher vertex. 

 Observe that once we have placed the edges from Lemma~\ref{lemma:spinetc}, and looking at vertex $i=3$, there are $2$ out-going stubs preceding $3$ and and $1$ in-going stub at $3$.
 Thus, there are two choices of incoming edges for vertex $3$.
 Inductively, for $i=4,\ldots,n$, assume that we have made our choice of incoming edges for vertices up to $i-1$.
 That means the in-going stub at $i$ has $2$ choices of preceding out-going stubs to connect to. 
 Once we reach vertex $n+1$, there is a unique choice of out-going stub for the in-going stub at $n+1$ to connect to. 
 In total, this gives $2^{n-2}$ possible ways to pair in- and out-going stubs to allocate edges. 
 \end{proof}

 \begin{example} \label{ex:n=4bijection}
     Bijection construction for $(\Z/2\Z)^2$ to $\mathscr{F}_{4,3}$ shown in Figure~\ref{fig:n=4Bijection}. 
     In the first column we begin with the edges described in Lemma~\ref{lemma:spinetc}. 
     In the second column, we add a consecutive edge $(l+1,l+2)$ for each $0$ at index $l$ in $\b$. 
     In the final column, we add the remaining edges corresponding to the $1$s in $\b$ and complete the missing in- and out-degrees at each vertex.
 \end{example}

 \begin{figure}
 \begin{center}
 \begin{tikzpicture}

 \begin{scope}[scale=0.8, xshift=-170, yshift=180]

	\vertex[fill,label=below:\footnotesize{$1$}](a1) at (1,0) {};
	\vertex[fill,label=below:\footnotesize{$2$}](a2) at (2,0) {};
	\vertex[fill,label=below:\footnotesize{$3$}](a3) at (3,0) {};
	\vertex[fill,label=below:\footnotesize{$4$}](a4) at (4,0) {};
	\vertex[fill,label=below:\footnotesize{$5$}](a5) at (5,0) {};
	\draw[-stealth, thick] (a1) to[out=0,in=180] (a2);
     \draw[-stealth, thick] (a1) to[out=60,in=120] (a2);
     \draw[-stealth, thick] (a2) to[out=0,in=180] (a3);
     \draw[-stealth, thick] (a3) to[out=0,in=180] (a4);
     \draw[-stealth, thick] (a4) to[out=0,in=180] (a5);
	\draw[-stealth, thick] (a4) to[out=60,in=120] (a5);
     \node[] at (2.5,-0.5) {\textcolor{orange}{$0$}};
     \node[] at (3.5,-0.5) {\textcolor{orange}{$0$}};

 \end{scope}

 \begin{scope}[scale=0.8, xshift=-170, yshift=120]

	\vertex[fill,label=below:\footnotesize{$1$}](a1) at (1,0) {};
	\vertex[fill,label=below:\footnotesize{$2$}](a2) at (2,0) {};
	\vertex[fill,label=below:\footnotesize{$3$}](a3) at (3,0) {};
	\vertex[fill,label=below:\footnotesize{$4$}](a4) at (4,0) {};
	\vertex[fill,label=below:\footnotesize{$5$}](a5) at (5,0) {};
	\draw[-stealth, thick] (a1) to[out=0,in=180] (a2);
     \draw[-stealth, thick] (a1) to[out=60,in=120] (a2);
     \draw[-stealth, thick] (a2) to[out=0,in=180] (a3);
     \draw[-stealth, thick] (a3) to[out=0,in=180] (a4);
     \draw[-stealth, thick] (a4) to[out=0,in=180] (a5);
	\draw[-stealth, thick] (a4) to[out=60,in=120] (a5);
     \node[] at (2.5,-0.5) {\textcolor{orange}{$0$}};
     \node[] at (3.5,-0.5) {\textcolor{orange}{$1$}};

 \end{scope}

 \begin{scope}[scale=0.8, xshift=-170, yshift=60]

	\vertex[fill,label=below:\footnotesize{$1$}](a1) at (1,0) {};
	\vertex[fill,label=below:\footnotesize{$2$}](a2) at (2,0) {};
	\vertex[fill,label=below:\footnotesize{$3$}](a3) at (3,0) {};
	\vertex[fill,label=below:\footnotesize{$4$}](a4) at (4,0) {};
	\vertex[fill,label=below:\footnotesize{$5$}](a5) at (5,0) {};
	\draw[-stealth, thick] (a1) to[out=0,in=180] (a2);
     \draw[-stealth, thick] (a1) to[out=60,in=120] (a2);
     \draw[-stealth, thick] (a2) to[out=0,in=180] (a3);
     \draw[-stealth, thick] (a3) to[out=0,in=180] (a4);
     \draw[-stealth, thick] (a4) to[out=0,in=180] (a5);
	\draw[-stealth, thick] (a4) to[out=60,in=120] (a5);
     \node[] at (2.5,-0.5) {\textcolor{orange}{$1$}};
     \node[] at (3.5,-0.5) {\textcolor{orange}{$0$}};

 \end{scope}

 \begin{scope}[scale=0.8, xshift=-170, yshift=0]

	\vertex[fill,label=below:\footnotesize{$1$}](a1) at (1,0) {};
	\vertex[fill,label=below:\footnotesize{$2$}](a2) at (2,0) {};
	\vertex[fill,label=below:\footnotesize{$3$}](a3) at (3,0) {};
	\vertex[fill,label=below:\footnotesize{$4$}](a4) at (4,0) {};
	\vertex[fill,label=below:\footnotesize{$5$}](a5) at (5,0) {};
	\draw[-stealth, thick] (a1) to[out=0,in=180] (a2);
     \draw[-stealth, thick] (a1) to[out=60,in=120] (a2);
     \draw[-stealth, thick] (a2) to[out=0,in=180] (a3);
     \draw[-stealth, thick] (a3) to[out=0,in=180] (a4);
     \draw[-stealth, thick] (a4) to[out=0,in=180] (a5);
	\draw[-stealth, thick] (a4) to[out=60,in=120] (a5);
     \node[] at (2.5,-0.5) {\textcolor{orange}{$1$}};
     \node[] at (3.5,-0.5) {\textcolor{orange}{$1$}};

 \end{scope}


 \begin{scope}[scale=0.8, xshift=0, yshift=180]

	\vertex[fill,label=below:\footnotesize{$1$}](a1) at (1,0) {};
	\vertex[fill,label=below:\footnotesize{$2$}](a2) at (2,0) {};
	\vertex[fill,label=below:\footnotesize{$3$}](a3) at (3,0) {};
	\vertex[fill,label=below:\footnotesize{$4$}](a4) at (4,0) {};
	\vertex[fill,label=below:\footnotesize{$5$}](a5) at (5,0) {};
	\draw[-stealth, thick] (a1) to[out=0,in=180] (a2);
     \draw[-stealth, thick] (a1) to[out=60,in=120] (a2);
     \draw[-stealth, thick] (a2) to[out=0,in=180] (a3);
	\draw[-stealth, thick] (a2) to[out=60,in=120] (a3);
     \draw[-stealth, thick] (a3) to[out=0,in=180] (a4);
	\draw[-stealth, thick] (a3) to[out=60,in=120] (a4);
     \draw[-stealth, thick] (a4) to[out=0,in=180] (a5);
	\draw[-stealth, thick] (a4) to[out=60,in=120] (a5);
     \node[] at (2.5,-0.5) {\textcolor{orange}{$0$}};
     \node[] at (3.5,-0.5) {\textcolor{orange}{$0$}};

 \end{scope}

 \begin{scope}[scale=0.8, xshift=0, yshift=120]

	\vertex[fill,label=below:\footnotesize{$1$}](a1) at (1,0) {};
	\vertex[fill,label=below:\footnotesize{$2$}](a2) at (2,0) {};
	\vertex[fill,label=below:\footnotesize{$3$}](a3) at (3,0) {};
	\vertex[fill,label=below:\footnotesize{$4$}](a4) at (4,0) {};
	\vertex[fill,label=below:\footnotesize{$5$}](a5) at (5,0) {};
	\draw[-stealth, thick] (a1) to[out=0,in=180] (a2);
     \draw[-stealth, thick] (a1) to[out=60,in=120] (a2);
     \draw[-stealth, thick] (a2) to[out=0,in=180] (a3);
	\draw[-stealth, thick] (a2) to[out=60,in=120] (a3);
     \draw[-stealth, thick] (a3) to[out=0,in=180] (a4);
     \draw[-stealth, thick] (a4) to[out=0,in=180] (a5);
	\draw[-stealth, thick] (a4) to[out=60,in=120] (a5);
     \node[] at (2.5,-0.5) {\textcolor{orange}{$0$}};
     \node[] at (3.5,-0.5) {\textcolor{orange}{$1$}};

 \end{scope}

 \begin{scope}[scale=0.8, xshift=0, yshift=60]

	\vertex[fill,label=below:\footnotesize{$1$}](a1) at (1,0) {};
	\vertex[fill,label=below:\footnotesize{$2$}](a2) at (2,0) {};
	\vertex[fill,label=below:\footnotesize{$3$}](a3) at (3,0) {};
	\vertex[fill,label=below:\footnotesize{$4$}](a4) at (4,0) {};
	\vertex[fill,label=below:\footnotesize{$5$}](a5) at (5,0) {};
	\draw[-stealth, thick] (a1) to[out=0,in=180] (a2);
     \draw[-stealth, thick] (a1) to[out=60,in=120] (a2);
     \draw[-stealth, thick] (a2) to[out=0,in=180] (a3);
     \draw[-stealth, thick] (a3) to[out=0,in=180] (a4);
	\draw[-stealth, thick] (a3) to[out=60,in=120] (a4);
     \draw[-stealth, thick] (a4) to[out=0,in=180] (a5);
	\draw[-stealth, thick] (a4) to[out=60,in=120] (a5);
     \node[] at (2.5,-0.5) {\textcolor{orange}{$1$}};
     \node[] at (3.5,-0.5) {\textcolor{orange}{$0$}};

 \end{scope}

 \begin{scope}[scale=0.8, xshift=0, yshift=0]

	\vertex[fill,label=below:\footnotesize{$1$}](a1) at (1,0) {};
	\vertex[fill,label=below:\footnotesize{$2$}](a2) at (2,0) {};
	\vertex[fill,label=below:\footnotesize{$3$}](a3) at (3,0) {};
	\vertex[fill,label=below:\footnotesize{$4$}](a4) at (4,0) {};
	\vertex[fill,label=below:\footnotesize{$5$}](a5) at (5,0) {};
	\draw[-stealth, thick] (a1) to[out=0,in=180] (a2);
     \draw[-stealth, thick] (a1) to[out=60,in=120] (a2);
     \draw[-stealth, thick] (a2) to[out=0,in=180] (a3);
     \draw[-stealth, thick] (a3) to[out=0,in=180] (a4);
     \draw[-stealth, thick] (a4) to[out=0,in=180] (a5);
	\draw[-stealth, thick] (a4) to[out=60,in=120] (a5);
     \node[] at (2.5,-0.5) {\textcolor{orange}{$1$}};
     \node[] at (3.5,-0.5) {\textcolor{orange}{$1$}};

 \end{scope}


 \begin{scope}[scale=0.8, xshift=170, yshift=180]

	\vertex[fill,label=below:\footnotesize{$1$}](a1) at (1,0) {};
	\vertex[fill,label=below:\footnotesize{$2$}](a2) at (2,0) {};
	\vertex[fill,label=below:\footnotesize{$3$}](a3) at (3,0) {};
	\vertex[fill,label=below:\footnotesize{$4$}](a4) at (4,0) {};
	\vertex[fill,label=below:\footnotesize{$5$}](a5) at (5,0) {};
	\draw[-stealth, thick] (a1) to[out=0,in=180] (a2);
     \draw[-stealth, thick] (a1) to[out=60,in=120] (a2);
     \draw[-stealth, thick] (a1) to[out=60,in=120] (a5);
     \draw[-stealth, thick] (a2) to[out=0,in=180] (a3);
	\draw[-stealth, thick] (a2) to[out=60,in=120] (a3);
     \draw[-stealth, thick] (a3) to[out=0,in=180] (a4);
	\draw[-stealth, thick] (a3) to[out=60,in=120] (a4);
     \draw[-stealth, thick] (a4) to[out=0,in=180] (a5);
	\draw[-stealth, thick] (a4) to[out=60,in=120] (a5);
     \node[] at (2.5,-0.5) {\textcolor{orange}{$0$}};
     \node[] at (3.5,-0.5) {\textcolor{orange}{$0$}};

 \end{scope}

 \begin{scope}[scale=0.8, xshift=170, yshift=120]

	\vertex[fill,label=below:\footnotesize{$1$}](a1) at (1,0) {};
	\vertex[fill,label=below:\footnotesize{$2$}](a2) at (2,0) {};
	\vertex[fill,label=below:\footnotesize{$3$}](a3) at (3,0) {};
	\vertex[fill,label=below:\footnotesize{$4$}](a4) at (4,0) {};
	\vertex[fill,label=below:\footnotesize{$5$}](a5) at (5,0) {};
	\draw[-stealth, thick] (a1) to[out=0,in=180] (a2);
     \draw[-stealth, thick] (a1) to[out=60,in=120] (a2);
     \draw[-stealth, thick] (a1) to[out=60,in=120] (a4);
     \draw[-stealth, thick] (a2) to[out=0,in=180] (a3);
	\draw[-stealth, thick] (a2) to[out=60,in=120] (a3);
     \draw[-stealth, thick] (a3) to[out=0,in=180] (a4);
	\draw[-stealth, thick] (a3) to[out=60,in=120] (a5);
     \draw[-stealth, thick] (a4) to[out=0,in=180] (a5);
	\draw[-stealth, thick] (a4) to[out=60,in=120] (a5);
     \node[] at (2.5,-0.5) {\textcolor{orange}{$0$}};
     \node[] at (3.5,-0.5) {\textcolor{orange}{$1$}};

 \end{scope}

 \begin{scope}[scale=0.8, xshift=170, yshift=60]

	\vertex[fill,label=below:\footnotesize{$1$}](a1) at (1,0) {};
	\vertex[fill,label=below:\footnotesize{$2$}](a2) at (2,0) {};
	\vertex[fill,label=below:\footnotesize{$3$}](a3) at (3,0) {};
	\vertex[fill,label=below:\footnotesize{$4$}](a4) at (4,0) {};
	\vertex[fill,label=below:\footnotesize{$5$}](a5) at (5,0) {};
	\draw[-stealth, thick] (a1) to[out=0,in=180] (a2);
     \draw[-stealth, thick] (a1) to[out=60,in=120] (a2);
     \draw[-stealth, thick] (a1) to[out=60,in=120] (a3);
     \draw[-stealth, thick] (a2) to[out=0,in=180] (a3);
	\draw[-stealth, thick] (a2) to[out=60,in=120] (a5);
     \draw[-stealth, thick] (a3) to[out=0,in=180] (a4);
	\draw[-stealth, thick] (a3) to[out=60,in=120] (a4);
     \draw[-stealth, thick] (a4) to[out=0,in=180] (a5);
	\draw[-stealth, thick] (a4) to[out=60,in=120] (a5);
     \node[] at (2.5,-0.5) {\textcolor{orange}{$1$}};
     \node[] at (3.5,-0.5) {\textcolor{orange}{$0$}};

 \end{scope}

 \begin{scope}[scale=0.8, xshift=170, yshift=0]

	\vertex[fill,label=below:\footnotesize{$1$}](a1) at (1,0) {};
	\vertex[fill,label=below:\footnotesize{$2$}](a2) at (2,0) {};
	\vertex[fill,label=below:\footnotesize{$3$}](a3) at (3,0) {};
	\vertex[fill,label=below:\footnotesize{$4$}](a4) at (4,0) {};
	\vertex[fill,label=below:\footnotesize{$5$}](a5) at (5,0) {};
	\draw[-stealth, thick] (a1) to[out=0,in=180] (a2);
     \draw[-stealth, thick] (a1) to[out=60,in=120] (a2);
     \draw[-stealth, thick] (a1) to[out=60,in=120] (a3);
     \draw[-stealth, thick] (a2) to[out=0,in=180] (a3);
	\draw[-stealth, thick] (a2) to[out=60,in=120] (a4);
     \draw[-stealth, thick] (a3) to[out=0,in=180] (a4);
	\draw[-stealth, thick] (a3) to[out=60,in=120] (a5);
     \draw[-stealth, thick] (a4) to[out=0,in=180] (a5);
	\draw[-stealth, thick] (a4) to[out=60,in=120] (a5);
     \node[] at (2.5,-0.5) {\textcolor{orange}{$1$}};
     \node[] at (3.5,-0.5) {\textcolor{orange}{$1$}};

 \end{scope}

 \end{tikzpicture}
 \end{center}
 \caption{Bijection construction for $(\Z/2\Z)^2$ to $\mathscr{F}_{4,3}$.}
 \label{fig:n=4Bijection}
 \end{figure}

 \begin{corollary} \label{cor:DAGCount}
     There are $2^{n-2}$ full DAGs on $n+1$ vertices and $\outd_1=3$. 
 \end{corollary}

 Now that we have a bijection between full DAGs and binary strings, we show that the Boolean poset structure of the binary strings is represented in the DAGs by an edge operation.   

 \begin{definition} \label{def:nested}
     Give a DAG $G$ with a linear order on its vertices, and a pair of edges $(a,d)$, $(b,c)$ satisfying $a<b<c<d$, we say such a pair of edges is \emph{nested}. 
 \end{definition}

 \begin{definition} \label{def:interchange}
     Give a DAG $G$ with a nested pair of edges $(a,d)$, $(b,c)$, define the \emph{interchange operation} to be the mapping $G\mapsto G'$, where $G'$ is defined by vertices $V(G')=V(G)$ and edges $E(G')=E(G)\setminus \{(a,d),(b,c)\}\cup \{(a,c),(b,d)\}$. 
 \end{definition}

 \begin{definition} \label{def:crossed}
     Give a DAG $G$ with a linear order on its vertices, and a pair of edges $(a,c)$, $(b,d)$ satisfying $a<b<c<d$, we say such a pair of edges is \emph{crossed} or \emph{twisted} or \emph{interchanged}. 
   \end{definition}

 \begin{definition} \label{def:reverse-interchange}
     Give a DAG $G'$ with a crossed pair of edges $(a,c)$, $(b,d)$, define the \emph{reverse interchange operation} to be the mapping $G'\mapsto G$, where $G$ is defined by vertices $V(G)=V(G')$ and edges $E(G)=E(G')\setminus \{(a,c),(b,d)\}\cup \{(a,d),(b,c)\}$. 
 \end{definition}

 \begin{remark}
     What we refer to in this paper as the interchange operation is also known as a double edge swap in the contexts of graph sampling~\cite{SIAMMCMC}. 
     The specialization to directed graphs continues below.
   \end{remark}

   \begin{example}\label{ex:interchange}
     Considering Figure~\ref{fig:n=4Bijection}, the right-most DAG labeled by the string $[1,0]$ has edge $(3,4)$ nesting in edge $(2,5)$.
     Edges $(1,3)$ and $(2,5)$ are crossed.
     Note that the DAGs labeled by strings $[1,1]$ and $[1,0]$ are related by the interchanging operations.
     \end{example}

 \begin{lemma} \label{lemma:closure}
     $\mathscr{F}_{n,k}$ is closed under the interchange operation.
 \end{lemma}

 \begin{proof}
     Let $G\in\mathscr{F}_{n,k}$. 
 Suppose $G$ has two edges $(a,d)$, $(b,c)\in E(G)$ satisfying $a<b<c<d$. 
 Let $G'$ be the DAG obtained by performing an interchange on $G$. 
 Then $G'$ still has $n+1$ vertices, and the edges incident to vertices other than $a,b,c,d$ are unchanged, preserving the in- and out-degrees at these vertices. 
 Thus we only need verify that the in- and out-degrees for $a,b,c,d$ are also unchanged, which follows directly from the definition of the interchange operation.
 \end{proof}  

 Next, we prove that the interchange operation spans the entirety of $\mathscr{F}_{n,3}$. We first introduce a useful fact about DAGs in $\mathscr{F}_{n,3}$. 

 \begin{lemma} \label{lemma:UniqueOverpass}
     Let $G\in\mathscr{F}_{n,3}$ and let $i$ be an interior vertex of $G$. 
     Then there is exactly one edge in $E(G)$ that passes vertex $i$. 
     That is, there is a unique edge $(a,d)$ so satisfying $a<i<d$. 
 \end{lemma}

 \textbf{Proof:} From Theorem~\ref{thm:bijection}, we know that there exists a length $n-2$ binary string $\b$ so that $G=G_\b$ with  
 \[ 
 J_\b=\{i\in[n-2] \ |\ b_i=1\}=\{j_1<\cdots<j_r\}. 
 \]
 For edge case convenience, define $j_0:=0$, $j_{r+1}:=n-1$ and 
 \[ 
 \widehat J_\b:=\{j_0=0<j_1<\cdots<j_r<j_{r+1}=n-1\}. 
 \]
 Note that this ensures $\widehat J_\b$ is always nonempty, even if $J_\b$ is empty. 
 Now let $m=\max\{l\in \widehat J_\b \ |\ j_l<i-1\}$. 
 That is, $j_m$ is the largest index in $\b$ where there is a $1$ left of $i$.

 Then $j_m<i-1$ by construction, so $j_m+1<i$. 
 Our maximum construction also ensures $j_{m+1}>i-1$, so $j_{m+1}+2>i$. 
 Thus the edge $(j_m +1,j_{m+1}+2)$ passes over vertex $i$, and we know $(j_m +1,j_{m+1}+2)\in E(G_\b)$ by the construction in Theorem~\ref{thm:bijection}. 

 To prove uniqueness, we once again turn to the construction of $E(G_\b)$ in the proof of Theorem~\ref{thm:bijection}. 
 Recall that we have three types of edges:
 \begin{enumerate}
     \item[(i)] Spine edges: $\{(i,i+1) \ |\ i\in[n] \} \cup \{(1,2),(n,n+1)\}$
     \item[(ii)] $0$-type edges: $\{(l+1,l+2) \ |\ b_l=0, \text{for } 1\leq l\leq n-2\}$
     \item[(iii)] $1$-type edges: $\{(1,j_1+2),(j_1+1,j_2+2),\ldots,(j_{r-1}+1,j_r+2),(j_r+1,n+1)\}$
 \end{enumerate}
 Since both (i) and (ii) connect consecutive vertices, no edge of those types can pass over an vertex. 
 The only other edges in $E(G_\b)$ are of the form $(j_l +1,j_{l+1}+2)$ for $j_l\in \widehat J_\b$. 
 But if $l\neq m$, then $(j_l +1,j_{l+1}+2)$ doesn't pass over $i$. 

 \qed 

 \begin{proposition} \label{prop:MinInflow}
     Let $G\in \mathscr{F}_{n,3}$ and let $i$ be an interior vertex. 
     Then any integral flow $f$ on $G$ with netflow $\w_n$ sends at most $i-2$ units of flow into vertex $i$. 
 \end{proposition}

 \begin{proof}
     Suppose $f$ sends all available flow through the spine. Then vertices $2,\ldots,i-1$ each contribute one unit of flow into vertex $i$. This accounts for all possible outflow before vertex $i$.
 \end{proof}

 \begin{proposition} \label{prop:OverpassDeterminesInflow}
    Let $G\in \mathscr{F}_{n,3}$ and let $i$ be an interior vertex. 
    Let $(a,d)$ be the unique edge passing over $i$ as in Lemma~\ref{lemma:UniqueOverpass}. 
    Then any integral flow $f$ on $G$ with netflow $\w_n$ has $i-2-f(a,d)$ units of flow entering vertex $i$.  
 \end{proposition}

 \begin{proof}
      Since $(a,d)$ is the unique edge passing over $i$, all flow coming out of vertices $2,\ldots,i-1$ not sent on edge $(a,d)$ must pass through vertex $i$. 
 Since the netflow vector is $\w_n$, each of these vertices contributes exactly one unit of flow. 
 \end{proof}

 \begin{proposition} \label{prop:ConsecutiveCrosses}
   
     Let $G\in\mathscr{F}_{n,3}$ with crossed edge pair $(a,c)$, $(b,d)$. 
     Then $c=b+1$. 
     That is, the middle two vertices of an crossed pair of edges must be consecutive. 
 \end{proposition}

 \begin{proof}
     Assume for the sake of contradiction that $G\in\mathscr{F}_{n,3}$ has interchange pairs $(a,c)$, $(b,d)$ where $c>b+1$. 
     Then if we look at vertex $b+1$ (which is interior), we have that $(a,c)$ satisfies $a<b+1<c$ and $(b,d)$ satisfies $b<b+1<d$. 
     But this means there are two distinct edges passing over vertex $b+1$, in violation of Lemma~\ref{lemma:UniqueOverpass}. 
 \end{proof}

 \begin{lemma} \label{lemma:0interchangeable}
     Let $G_\b\in\mathscr{F}_{n,3}$ with $\b=b_1\cdots b_{n-2}$. Suppose $b_i=0$. 
     Then $(i+1,i+2)\in E(G_\b)$, and is nested within exactly one other edge in $E(G_\b)$.

 \end{lemma}

 \begin{proof}
     Since $b_i=0$, it follows from the bijective construction in Theorem~\ref{thm:bijection} that $(i+1,i+2)\in E(G_\b)$. Then we can apply Lemma~\ref{lemma:UniqueOverpass} to vertex $i+1$ to get a unique edge $(j_l+1,j_{l+1}+2)$ satisfying $j_l+1<i+1<j_{l+1}+2$. 

 Since $\b_i=0$, both in-stubs for $i+2$ are filled by consecutive edges (Type (i) and Type (ii)). 
 Thus $j_{l+1}+2\neq i+2$, so we have $j_l+1<i+1<i+2<j_{l+1}+2$, and the edges $(i+1,i+2)$ and $(j_l+1,j_{l+1}+2)$ are nested. 
 \end{proof} 

 \begin{lemma} \label{lemma:interchange=0to1}
     Let $\b\in(\Z/2\Z)^{n-2}$ with $b_i=0$. Define $\b'$ to be the binary string obtained by changing $\b$ to have a $1$ at index $i$. 
     Define $G'$ to be the DAG obtained by interchanging the edges $(i+1,i+2)$ and $(j_l+1,j_{l+1}+2)$ in $G$ (as in Lemma~\ref{lemma:0interchangeable}). 
     Then $G_{\b'}=G'$. 

     In short, changing a $0$ to a $1$ is ``the same" as interchanging the edge associated with the $0$. 
 \end{lemma}

 \begin{proof}
     Since $V(G')=V(G_{\b'})=[n+1]$, we need only show that $E(G')=E(G_{\b'})$. By Lemma~\ref{lemma:0interchangeable} and the definition of interchange, we have that 
 \begin{align*}
     E(G') &= E(G)/\{(i+1,i+2),(j_a+1,j_{a+1}+2)\} 
     \cup \{(i+1,j_{l+1}+2),(j_l+1,i+2)\}.
 \end{align*}

 Now consider $\b'$. 
 By assumption, the list of indices where $\b'$ has $1$s is
 \[ 
 J_{\b'} = J_\b\cup\{i\} = \{\cdots j_l < i< j_{l+1}<\cdots\}. 
 \]
 It follows from the bijective construction that $(i+1,j_{l+1}+2),(j_l+1,i+2)\in E(G_{\b'})$. 
 Additionally, since $b'_i\neq 0$, we have $(i+1,i+2)\not\in E(G_{\b'})$.
 Finally since $j_l$ and $j_{l+1}$ are not the indices of consecutive $1$s in $\b'$, $(j_l+1,j_{l+1}+2)\not\in E(G_{\b'})$.
 Since no other indices of $\b$ were changed to get $\b'$, we have that   
 \begin{align*}
     E(G_{\b'}) &= E(G)/\{(i+1,i+2),(j_l+1,j_{l+1}+2)\} 
     \cup \{(i+1,j_{l+1}+2),(j_l+1,i+2)\} \\
     &= E(G').
 \end{align*}
 \end{proof}  

 \begin{corollary} \label{cor:3coverage}
     Every DAG in $\mathscr{F}_{n,3}$ is obtainable by performing interchange operations on the DAG $G$ defined by
     \begin{align*}
         V(G) &= [n+1] \\
         E(G) &= \{(i,i+1)^2 \ |\ i\in[n]\} \cup \{(1,n+1)\}.
     \end{align*}
   
     That is, the DAG $G$ in correspondence with the binary string $0\cdots0$. 
 \end{corollary}

 \begin{proof}
     Let $G_\b\in\mathscr{F}_{n,3}$, where $\b\in(\Z/2\Z)^{n-2}$ is the binary string corresponding to $G_\b$. 
 Then perform interchanges to $G_{0\cdots0}$, as in Lemma~\ref{lemma:interchange=0to1}, at every index where $\b$ has a 1. 

 Note that we can choose any order to perform the necessary interchanges. 
 This is ensured by the bijection with binary strings and previous lemma. 
 \end{proof} 

 \begin{corollary} \label{cor:3poset}
     The interchange operation induces a partial order on $\mathscr{F}_{n,3}$ that agrees with the Boolean partial order on binary strings of length $n-2$. 
 \end{corollary}

 \begin{example} \label{ex:n=5HasseInterchange}
     See Figure~\ref{fig:n=5HasseInterchange} for the Hasse diagram demonstrating poset equivalence between $\mathscr{F}_{5,3}$ and $(\Z/2\Z)^{3}$. 
     Edge interchange pairs are listed next to the edges the correspond with (assume moving upwards through the Hasse diagram). Note that all $2^3$ DAGs in $\mathscr{F}_{5,3}$ are achieved via interchanges of $G_{000}$.
 \end{example}

 \begin{figure}
     \begin{center}
     \begin{tikzpicture}
   
     \begin{scope}[xshift=0, yshift=305, scale=0.7]
   
	\vertex[fill](a1) at (1,0) {};
	\vertex[fill](a2) at (2,0) {};
	\vertex[fill](a3) at (3,0) {};
	\vertex[fill](a4) at (4,0) {};
	\vertex[fill](a5) at (5,0) {};
	\vertex[fill](a6) at (6,0) {};
	\draw[-stealth, thick] (a1) to[out=0,in=180] (a2);
	 \draw[-stealth, thick] (a1) to[out=60,in=120] (a2);
	 \draw[-stealth, thick] (a1) to[out=60,in=120] (a3);
	 \draw[-stealth, thick] (a2) to[out=0,in=180] (a3);
	\draw[-stealth, thick] (a2) to[out=60,in=120] (a4);
	 \draw[-stealth, thick] (a3) to[out=0,in=180] (a4);
	\draw[-stealth, thick] (a3) to[out=60,in=120] (a5);
	 \draw[-stealth, thick] (a4) to[out=0,in=180] (a5);
	\draw[-stealth, thick] (a4) to[out=60,in=120] (a6);
	 \draw[-stealth, thick] (a5) to[out=0,in=180] (a6);
	\draw[-stealth, thick] (a5) to[out=60,in=120] (a6);
	 \node[] at (2.5,-0.5) {\textcolor{orange}{$1$}};
	 \node[] at (3.5,-0.5) {\textcolor{orange}{$1$}};
	 \node[] at (4.5,-0.5) {\textcolor{orange}{$1$}};
   
     \end{scope}
   
     \begin{scope}[xshift=-129.9, yshift=200, scale=0.7]
   
	\vertex[fill](a1) at (1,0) {};
	\vertex[fill](a2) at (2,0) {};
	\vertex[fill](a3) at (3,0) {};
	\vertex[fill](a4) at (4,0) {};
	\vertex[fill](a5) at (5,0) {};
	\vertex[fill](a6) at (6,0) {};
	\draw[-stealth, thick] (a1) to[out=0,in=180] (a2);
	 \draw[-stealth, thick] (a1) to[out=60,in=120] (a2);
	 \draw[-stealth, thick] (a1) to[out=60,in=120] (a4);
	 \draw[-stealth, thick] (a2) to[out=0,in=180] (a3);
	\draw[-stealth, thick] (a2) to[out=60,in=120] (a3);
	 \draw[-stealth, thick] (a3) to[out=0,in=180] (a4);
	\draw[-stealth, thick] (a3) to[out=60,in=120] (a5);
	 \draw[-stealth, thick] (a4) to[out=0,in=180] (a5);
	\draw[-stealth, thick] (a4) to[out=60,in=120] (a6);
	 \draw[-stealth, thick] (a5) to[out=0,in=180] (a6);
	\draw[-stealth, thick] (a5) to[out=60,in=120] (a6);
	 \node[] at (2.5,-0.5) {\textcolor{orange}{$0$}};
	 \node[] at (3.5,-0.5) {\textcolor{orange}{$1$}};
	 \node[] at (4.5,-0.5) {\textcolor{orange}{$1$}};
   
     \end{scope}
   
     \begin{scope}[xshift=0, yshift=200, scale=0.7]
   
	\vertex[fill](a1) at (1,0) {};
	\vertex[fill](a2) at (2,0) {};
	\vertex[fill](a3) at (3,0) {};
	\vertex[fill](a4) at (4,0) {};
	\vertex[fill](a5) at (5,0) {};
	\vertex[fill](a6) at (6,0) {};
	\draw[-stealth, thick] (a1) to[out=0,in=180] (a2);
	 \draw[-stealth, thick] (a1) to[out=60,in=120] (a2);
	 \draw[-stealth, thick] (a1) to[out=60,in=120] (a3);
	 \draw[-stealth, thick] (a2) to[out=0,in=180] (a3);
	\draw[-stealth, thick] (a2) to[out=60,in=120] (a5);
	 \draw[-stealth, thick] (a3) to[out=0,in=180] (a4);
	\draw[-stealth, thick] (a3) to[out=60,in=120] (a4);
	 \draw[-stealth, thick] (a4) to[out=0,in=180] (a5);
	\draw[-stealth, thick] (a4) to[out=60,in=120] (a6);
	 \draw[-stealth, thick] (a5) to[out=0,in=180] (a6);
	\draw[-stealth, thick] (a5) to[out=60,in=120] (a6);
	 \node[] at (2.5,-0.5) {\textcolor{orange}{$1$}};
	 \node[] at (3.5,-0.5) {\textcolor{orange}{$0$}};
	 \node[] at (4.5,-0.5) {\textcolor{orange}{$1$}};
   
     \end{scope}
   
     \begin{scope}[xshift=129.9, yshift=200, scale=0.7]
   
	\vertex[fill](a1) at (1,0) {};
	\vertex[fill](a2) at (2,0) {};
	\vertex[fill](a3) at (3,0) {};
	\vertex[fill](a4) at (4,0) {};
	\vertex[fill](a5) at (5,0) {};
	\vertex[fill](a6) at (6,0) {};
	\draw[-stealth, thick] (a1) to[out=0,in=180] (a2);
	 \draw[-stealth, thick] (a1) to[out=60,in=120] (a2);
	 \draw[-stealth, thick] (a1) to[out=60,in=120] (a3);
	 \draw[-stealth, thick] (a2) to[out=0,in=180] (a3);
	\draw[-stealth, thick] (a2) to[out=60,in=120] (a4);
	 \draw[-stealth, thick] (a3) to[out=0,in=180] (a4);
	\draw[-stealth, thick] (a3) to[out=60,in=120] (a6);
	 \draw[-stealth, thick] (a4) to[out=0,in=180] (a5);
	\draw[-stealth, thick] (a4) to[out=60,in=120] (a5);
	 \draw[-stealth, thick] (a5) to[out=0,in=180] (a6);
	\draw[-stealth, thick] (a5) to[out=60,in=120] (a6);
	 \node[] at (2.5,-0.5) {\textcolor{orange}{$1$}};
	 \node[] at (3.5,-0.5) {\textcolor{orange}{$1$}};
	 \node[] at (4.5,-0.5) {\textcolor{orange}{$0$}};
   
     \end{scope}
   
     \begin{scope}[xshift=-129.9, yshift=75, scale=0.7]
   
	\vertex[fill](a1) at (1,0) {};
	\vertex[fill](a2) at (2,0) {};
	\vertex[fill](a3) at (3,0) {};
	\vertex[fill](a4) at (4,0) {};
	\vertex[fill](a5) at (5,0) {};
	\vertex[fill](a6) at (6,0) {};
	\draw[-stealth, thick] (a1) to[out=0,in=180] (a2);
	 \draw[-stealth, thick] (a1) to[out=60,in=120] (a2);
	 \draw[-stealth, thick] (a1) to[out=60,in=120] (a5);
	 \draw[-stealth, thick] (a2) to[out=0,in=180] (a3);
	\draw[-stealth, thick] (a2) to[out=60,in=120] (a3);
	 \draw[-stealth, thick] (a3) to[out=0,in=180] (a4);
	\draw[-stealth, thick] (a3) to[out=60,in=120] (a4);
	 \draw[-stealth, thick] (a4) to[out=0,in=180] (a5);
	\draw[-stealth, thick] (a4) to[out=60,in=120] (a6);
	 \draw[-stealth, thick] (a5) to[out=0,in=180] (a6);
	\draw[-stealth, thick] (a5) to[out=60,in=120] (a6);
	 \node[] at (2.5,-0.5) {\textcolor{orange}{$0$}};
	 \node[] at (3.5,-0.5) {\textcolor{orange}{$0$}};
	 \node[] at (4.5,-0.5) {\textcolor{orange}{$1$}};
   
     \end{scope}
   
     \begin{scope}[xshift=0, yshift=75, scale=0.7]
   
	\vertex[fill](a1) at (1,0) {};
	\vertex[fill](a2) at (2,0) {};
	\vertex[fill](a3) at (3,0) {};
	\vertex[fill](a4) at (4,0) {};
	\vertex[fill](a5) at (5,0) {};
	\vertex[fill](a6) at (6,0) {};
	\draw[-stealth, thick] (a1) to[out=0,in=180] (a2);
	 \draw[-stealth, thick] (a1) to[out=60,in=120] (a2);
	 \draw[-stealth, thick] (a1) to[out=60,in=120] (a4);
	 \draw[-stealth, thick] (a2) to[out=0,in=180] (a3);
	\draw[-stealth, thick] (a2) to[out=60,in=120] (a3);
	 \draw[-stealth, thick] (a3) to[out=0,in=180] (a4);
	\draw[-stealth, thick] (a3) to[out=60,in=120] (a6);
	 \draw[-stealth, thick] (a4) to[out=0,in=180] (a5);
	\draw[-stealth, thick] (a4) to[out=60,in=120] (a5);
	 \draw[-stealth, thick] (a5) to[out=0,in=180] (a6);
	\draw[-stealth, thick] (a5) to[out=60,in=120] (a6);
	 \node[] at (2.5,-0.5) {\textcolor{orange}{$0$}};
	 \node[] at (3.5,-0.5) {\textcolor{orange}{$1$}};
	 \node[] at (4.5,-0.5) {\textcolor{orange}{$0$}};
   
     \end{scope}
   
     \begin{scope}[xshift=129.9, yshift=75, scale=0.7]
   
	\vertex[fill](a1) at (1,0) {};
	\vertex[fill](a2) at (2,0) {};
	\vertex[fill](a3) at (3,0) {};
	\vertex[fill](a4) at (4,0) {};
	\vertex[fill](a5) at (5,0) {};
	\vertex[fill](a6) at (6,0) {};
	\draw[-stealth, thick] (a1) to[out=0,in=180] (a2);
	 \draw[-stealth, thick] (a1) to[out=60,in=120] (a2);
	 \draw[-stealth, thick] (a1) to[out=60,in=120] (a3);
	 \draw[-stealth, thick] (a2) to[out=0,in=180] (a3);
	\draw[-stealth, thick] (a2) to[out=60,in=120] (a6);
	 \draw[-stealth, thick] (a3) to[out=0,in=180] (a4);
	\draw[-stealth, thick] (a3) to[out=60,in=120] (a4);
	 \draw[-stealth, thick] (a4) to[out=0,in=180] (a5);
	\draw[-stealth, thick] (a4) to[out=60,in=120] (a5);
	 \draw[-stealth, thick] (a5) to[out=0,in=180] (a6);
	\draw[-stealth, thick] (a5) to[out=60,in=120] (a6);
	 \node[] at (2.5,-0.5) {\textcolor{orange}{$1$}};
	 \node[] at (3.5,-0.5) {\textcolor{orange}{$0$}};
	 \node[] at (4.5,-0.5) {\textcolor{orange}{$0$}};
   
     \end{scope}
   
     \begin{scope}[xshift=0, yshift=-30, scale=0.7]
   
	\vertex[fill](a1) at (1,0) {};
	\vertex[fill](a2) at (2,0) {};
	\vertex[fill](a3) at (3,0) {};
	\vertex[fill](a4) at (4,0) {};
	\vertex[fill](a5) at (5,0) {};
	\vertex[fill](a6) at (6,0) {};
	\draw[-stealth, thick] (a1) to[out=0,in=180] (a2);
	 \draw[-stealth, thick] (a1) to[out=60,in=120] (a2);
	 \draw[-stealth, thick] (a1) to[out=60,in=120] (a6);
	 \draw[-stealth, thick] (a2) to[out=0,in=180] (a3);
	\draw[-stealth, thick] (a2) to[out=60,in=120] (a3);
	 \draw[-stealth, thick] (a3) to[out=0,in=180] (a4);
	\draw[-stealth, thick] (a3) to[out=60,in=120] (a4);
	 \draw[-stealth, thick] (a4) to[out=0,in=180] (a5);
	\draw[-stealth, thick] (a4) to[out=60,in=120] (a5);
	 \draw[-stealth, thick] (a5) to[out=0,in=180] (a6);
	\draw[-stealth, thick] (a5) to[out=60,in=120] (a6);
	 \node[] at (2.5,-0.5) {\textcolor{orange}{$0$}};
	 \node[] at (3.5,-0.5) {\textcolor{orange}{$0$}};
	 \node[] at (4.5,-0.5) {\textcolor{orange}{$0$}};
   
     \end{scope}
   
     \draw[thick] (2.5,0) to (-2,2);
     \draw[thick] (2.5,0) to (2.5,2);
     \draw[thick] (2.5,0) to (7,2);
   
     \draw[thick] (2.5,10) to (-2,8);
     \draw[thick] (2.5,10) to (2.5,8);
     \draw[thick] (2.5,10) to (7,8);
   
     \draw[thick] (-2,4) to (-2,6);
     \draw[thick] (-2,4) to (2.5,6);
   
     \draw[thick] (7,4) to (2.5,6);
     \draw[thick] (7,4) to (7,6);
   
     \draw[thick] (2.5,4) to (-2,6);
     \draw[thick] (2.5,4) to (7,6);
   
     \node[] at (0.6,1.5) {$(1,6),(4,5)$};
     \node[] at (3.6,1.5) {$(1,6),(3,4)$};
     \node[] at (7.6,1.5) {$(1,6),(2,3)$};
   
     \node[] at (0.6,8.5) {$(1,4),(2,3)$};
     \node[] at (3.6,8.5) {$(2,5),(3,4)$};
     \node[] at (7.6,8.5) {$(3,6),(4,5)$};
   
     \node[] at (-3.1,5) {$(1,5),(3,4)$};
     \node[] at (4.5,5.8) {$(2,6),(4,5)$};
   
     \node[] at (0.5,5.8) {$(1,5),(2,3)$};
     \node[] at (4.5,4.2) {$(1,4),(2,3)$};
   
     \node[] at (0.5,4.2) {$(3,6),(4,5)$};
     \node[] at (8.1,5) {$(2,6),(3,4)$};
   
     \end{tikzpicture}
     \end{center}
     \caption{$\mathscr{F}_{5,3}$ Hasse diagram with edges labelled by corresponding interchanges. DAG vertex labels suppressed for clarity, but each DAG is labelled $1$ to $6$ from left to right.}
     \label{fig:n=5HasseInterchange}
     \end{figure}
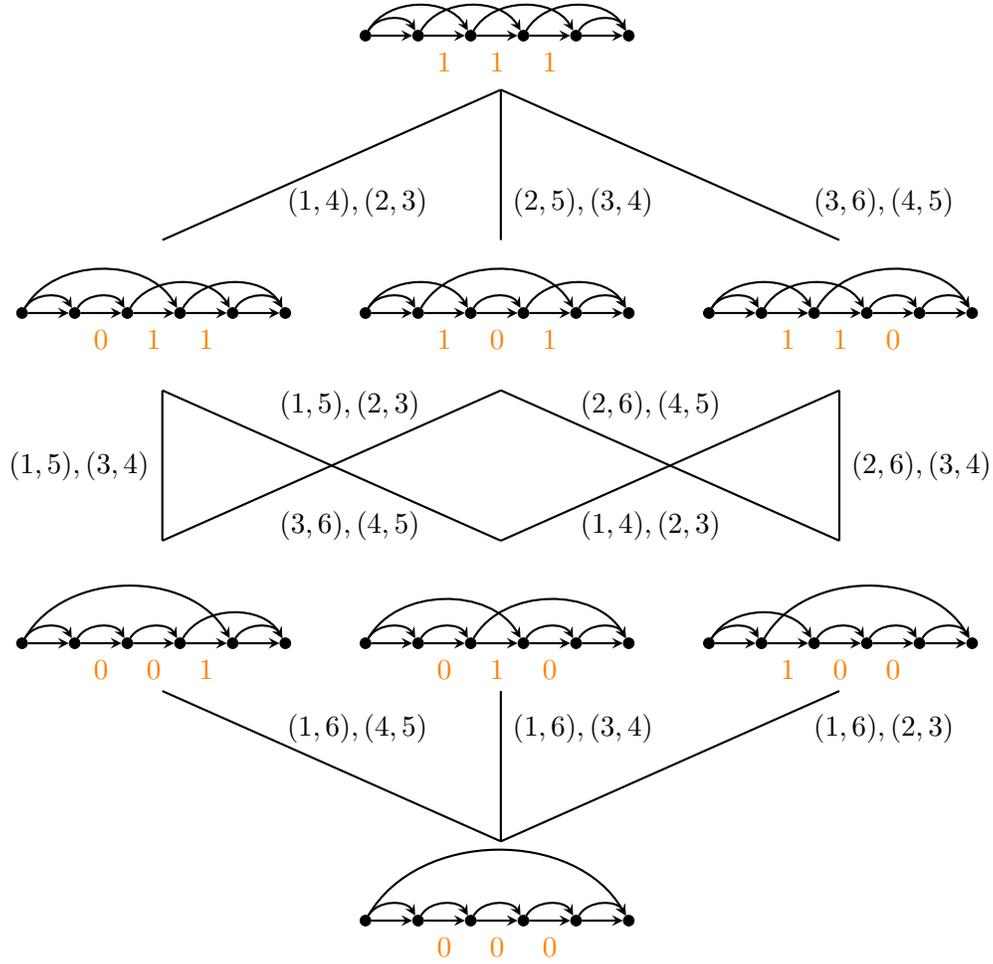

\section{Volume Inequalities} \label{sec:volume}

Our goal in this section is to prove Theorem~\ref{thm:MainThm}, which states that the values of the volumes of flow polytopes for DAGs under the Boolean partial order obtained in Corollary~\ref{cor:3coverage} respects the dual partial order.
To illustrate this, consider Example~\ref{ex:n=5HasseVolume}, which demonstrates that the function $\vol:\mathscr{F}_{n,3}\ra \Z_{\geq0}$ is order reversing with respect to our interchange order, as we prove in Theorem~\ref{thm:MainThm}. 
To prove that this is true in general, we construct a procedure to find the value of $K_G(\w_n)$, and hence find the volume of $\F_1(G)$ (Proposition~\ref{prop:VolumeSpecialization}). 
Further, this procedure can be used to compute $K_G(\a)$ for any netflow vector $\a$, and any DAG $G$ (not necessarily full).

\begin{example} \label{ex:n=5HasseVolume}
    In Figure~\ref{fig:n=5HasseVolume}, the DAGs in $\mathscr{F}_{5,3}$ are arranged as in Figure~\ref{fig:n=5HasseInterchange}. 
    Additionally, labels are given for each DAG $G$ that provide the volume of its flow polytope $\F_1(G)$.
  \end{example}

\begin{figure}
    \begin{center}
    \begin{tikzpicture}
    
    \begin{scope}[xshift=0, yshift=305, scale=0.7]
    
    	\vertex[fill](a1) at (1,0) {};
    	\vertex[fill](a2) at (2,0) {};
    	\vertex[fill](a3) at (3,0) {};
    	\vertex[fill](a4) at (4,0) {};
    	\vertex[fill](a5) at (5,0) {};
    	\vertex[fill](a6) at (6,0) {};
    	\draw[-stealth, thick] (a1) to[out=0,in=180] (a2);
        \draw[-stealth, thick] (a1) to[out=60,in=120] (a2);
        \draw[-stealth, thick] (a1) to[out=60,in=120] (a3);
        \draw[-stealth, thick] (a2) to[out=0,in=180] (a3);
    	\draw[-stealth, thick] (a2) to[out=60,in=120] (a4);
        \draw[-stealth, thick] (a3) to[out=0,in=180] (a4);
    	\draw[-stealth, thick] (a3) to[out=60,in=120] (a5);
        \draw[-stealth, thick] (a4) to[out=0,in=180] (a5);
    	\draw[-stealth, thick] (a4) to[out=60,in=120] (a6);
        \draw[-stealth, thick] (a5) to[out=0,in=180] (a6);
    	\draw[-stealth, thick] (a5) to[out=60,in=120] (a6);
        \node[] at (2.5,-0.5) {\textcolor{orange}{$1$}};
        \node[] at (3.5,-0.5) {\textcolor{orange}{$1$}};
        \node[] at (4.5,-0.5) {\textcolor{orange}{$1$}};
    
    \end{scope}
    
    \begin{scope}[xshift=-129.9, yshift=200, scale=0.7]
    
    	\vertex[fill](a1) at (1,0) {};
    	\vertex[fill](a2) at (2,0) {};
    	\vertex[fill](a3) at (3,0) {};
    	\vertex[fill](a4) at (4,0) {};
    	\vertex[fill](a5) at (5,0) {};
    	\vertex[fill](a6) at (6,0) {};
    	\draw[-stealth, thick] (a1) to[out=0,in=180] (a2);
        \draw[-stealth, thick] (a1) to[out=60,in=120] (a2);
        \draw[-stealth, thick] (a1) to[out=60,in=120] (a4);
        \draw[-stealth, thick] (a2) to[out=0,in=180] (a3);
    	\draw[-stealth, thick] (a2) to[out=60,in=120] (a3);
        \draw[-stealth, thick] (a3) to[out=0,in=180] (a4);
    	\draw[-stealth, thick] (a3) to[out=60,in=120] (a5);
        \draw[-stealth, thick] (a4) to[out=0,in=180] (a5);
    	\draw[-stealth, thick] (a4) to[out=60,in=120] (a6);
        \draw[-stealth, thick] (a5) to[out=0,in=180] (a6);
    	\draw[-stealth, thick] (a5) to[out=60,in=120] (a6);
        \node[] at (2.5,-0.5) {\textcolor{orange}{$0$}};
        \node[] at (3.5,-0.5) {\textcolor{orange}{$1$}};
        \node[] at (4.5,-0.5) {\textcolor{orange}{$1$}};
    
    \end{scope}
    
    \begin{scope}[xshift=0, yshift=200, scale=0.7]
    
    	\vertex[fill](a1) at (1,0) {};
    	\vertex[fill](a2) at (2,0) {};
    	\vertex[fill](a3) at (3,0) {};
    	\vertex[fill](a4) at (4,0) {};
    	\vertex[fill](a5) at (5,0) {};
    	\vertex[fill](a6) at (6,0) {};
    	\draw[-stealth, thick] (a1) to[out=0,in=180] (a2);
        \draw[-stealth, thick] (a1) to[out=60,in=120] (a2);
        \draw[-stealth, thick] (a1) to[out=60,in=120] (a3);
        \draw[-stealth, thick] (a2) to[out=0,in=180] (a3);
    	\draw[-stealth, thick] (a2) to[out=60,in=120] (a5);
        \draw[-stealth, thick] (a3) to[out=0,in=180] (a4);
    	\draw[-stealth, thick] (a3) to[out=60,in=120] (a4);
        \draw[-stealth, thick] (a4) to[out=0,in=180] (a5);
    	\draw[-stealth, thick] (a4) to[out=60,in=120] (a6);
        \draw[-stealth, thick] (a5) to[out=0,in=180] (a6);
    	\draw[-stealth, thick] (a5) to[out=60,in=120] (a6);
        \node[] at (2.5,-0.5) {\textcolor{orange}{$1$}};
        \node[] at (3.5,-0.5) {\textcolor{orange}{$0$}};
        \node[] at (4.5,-0.5) {\textcolor{orange}{$1$}};
    
    \end{scope}
    
    \begin{scope}[xshift=129.9, yshift=200, scale=0.7]
    
    	\vertex[fill](a1) at (1,0) {};
    	\vertex[fill](a2) at (2,0) {};
    	\vertex[fill](a3) at (3,0) {};
    	\vertex[fill](a4) at (4,0) {};
    	\vertex[fill](a5) at (5,0) {};
    	\vertex[fill](a6) at (6,0) {};
    	\draw[-stealth, thick] (a1) to[out=0,in=180] (a2);
        \draw[-stealth, thick] (a1) to[out=60,in=120] (a2);
        \draw[-stealth, thick] (a1) to[out=60,in=120] (a3);
        \draw[-stealth, thick] (a2) to[out=0,in=180] (a3);
    	\draw[-stealth, thick] (a2) to[out=60,in=120] (a4);
        \draw[-stealth, thick] (a3) to[out=0,in=180] (a4);
    	\draw[-stealth, thick] (a3) to[out=60,in=120] (a6);
        \draw[-stealth, thick] (a4) to[out=0,in=180] (a5);
    	\draw[-stealth, thick] (a4) to[out=60,in=120] (a5);
        \draw[-stealth, thick] (a5) to[out=0,in=180] (a6);
    	\draw[-stealth, thick] (a5) to[out=60,in=120] (a6);
        \node[] at (2.5,-0.5) {\textcolor{orange}{$1$}};
        \node[] at (3.5,-0.5) {\textcolor{orange}{$1$}};
        \node[] at (4.5,-0.5) {\textcolor{orange}{$0$}};
    
    \end{scope}
    
    \begin{scope}[xshift=-129.9, yshift=75, scale=0.7]
    
    	\vertex[fill](a1) at (1,0) {};
    	\vertex[fill](a2) at (2,0) {};
    	\vertex[fill](a3) at (3,0) {};
    	\vertex[fill](a4) at (4,0) {};
    	\vertex[fill](a5) at (5,0) {};
    	\vertex[fill](a6) at (6,0) {};
    	\draw[-stealth, thick] (a1) to[out=0,in=180] (a2);
        \draw[-stealth, thick] (a1) to[out=60,in=120] (a2);
        \draw[-stealth, thick] (a1) to[out=60,in=120] (a5);
        \draw[-stealth, thick] (a2) to[out=0,in=180] (a3);
    	\draw[-stealth, thick] (a2) to[out=60,in=120] (a3);
        \draw[-stealth, thick] (a3) to[out=0,in=180] (a4);
    	\draw[-stealth, thick] (a3) to[out=60,in=120] (a4);
        \draw[-stealth, thick] (a4) to[out=0,in=180] (a5);
    	\draw[-stealth, thick] (a4) to[out=60,in=120] (a6);
        \draw[-stealth, thick] (a5) to[out=0,in=180] (a6);
    	\draw[-stealth, thick] (a5) to[out=60,in=120] (a6);
        \node[] at (2.5,-0.5) {\textcolor{orange}{$0$}};
        \node[] at (3.5,-0.5) {\textcolor{orange}{$0$}};
        \node[] at (4.5,-0.5) {\textcolor{orange}{$1$}};
    
    \end{scope}
    
    \begin{scope}[xshift=0, yshift=75, scale=0.7]
    
    	\vertex[fill](a1) at (1,0) {};
    	\vertex[fill](a2) at (2,0) {};
    	\vertex[fill](a3) at (3,0) {};
    	\vertex[fill](a4) at (4,0) {};
    	\vertex[fill](a5) at (5,0) {};
    	\vertex[fill](a6) at (6,0) {};
    	\draw[-stealth, thick] (a1) to[out=0,in=180] (a2);
        \draw[-stealth, thick] (a1) to[out=60,in=120] (a2);
        \draw[-stealth, thick] (a1) to[out=60,in=120] (a4);
        \draw[-stealth, thick] (a2) to[out=0,in=180] (a3);
    	\draw[-stealth, thick] (a2) to[out=60,in=120] (a3);
        \draw[-stealth, thick] (a3) to[out=0,in=180] (a4);
    	\draw[-stealth, thick] (a3) to[out=60,in=120] (a6);
        \draw[-stealth, thick] (a4) to[out=0,in=180] (a5);
    	\draw[-stealth, thick] (a4) to[out=60,in=120] (a5);
        \draw[-stealth, thick] (a5) to[out=0,in=180] (a6);
    	\draw[-stealth, thick] (a5) to[out=60,in=120] (a6);
        \node[] at (2.5,-0.5) {\textcolor{orange}{$0$}};
        \node[] at (3.5,-0.5) {\textcolor{orange}{$1$}};
        \node[] at (4.5,-0.5) {\textcolor{orange}{$0$}};
    
    \end{scope}
    
    \begin{scope}[xshift=129.9, yshift=75, scale=0.7]
    
    	\vertex[fill](a1) at (1,0) {};
    	\vertex[fill](a2) at (2,0) {};
    	\vertex[fill](a3) at (3,0) {};
    	\vertex[fill](a4) at (4,0) {};
    	\vertex[fill](a5) at (5,0) {};
    	\vertex[fill](a6) at (6,0) {};
    	\draw[-stealth, thick] (a1) to[out=0,in=180] (a2);
        \draw[-stealth, thick] (a1) to[out=60,in=120] (a2);
        \draw[-stealth, thick] (a1) to[out=60,in=120] (a3);
        \draw[-stealth, thick] (a2) to[out=0,in=180] (a3);
    	\draw[-stealth, thick] (a2) to[out=60,in=120] (a6);
        \draw[-stealth, thick] (a3) to[out=0,in=180] (a4);
    	\draw[-stealth, thick] (a3) to[out=60,in=120] (a4);
        \draw[-stealth, thick] (a4) to[out=0,in=180] (a5);
    	\draw[-stealth, thick] (a4) to[out=60,in=120] (a5);
        \draw[-stealth, thick] (a5) to[out=0,in=180] (a6);
    	\draw[-stealth, thick] (a5) to[out=60,in=120] (a6);
        \node[] at (2.5,-0.5) {\textcolor{orange}{$1$}};
        \node[] at (3.5,-0.5) {\textcolor{orange}{$0$}};
        \node[] at (4.5,-0.5) {\textcolor{orange}{$0$}};
    
    \end{scope}
    
    \begin{scope}[xshift=0, yshift=-30, scale=0.7]
    
    	\vertex[fill](a1) at (1,0) {};
    	\vertex[fill](a2) at (2,0) {};
    	\vertex[fill](a3) at (3,0) {};
    	\vertex[fill](a4) at (4,0) {};
    	\vertex[fill](a5) at (5,0) {};
    	\vertex[fill](a6) at (6,0) {};
    	\draw[-stealth, thick] (a1) to[out=0,in=180] (a2);
        \draw[-stealth, thick] (a1) to[out=60,in=120] (a2);
        \draw[-stealth, thick] (a1) to[out=60,in=120] (a6);
        \draw[-stealth, thick] (a2) to[out=0,in=180] (a3);
    	\draw[-stealth, thick] (a2) to[out=60,in=120] (a3);
        \draw[-stealth, thick] (a3) to[out=0,in=180] (a4);
    	\draw[-stealth, thick] (a3) to[out=60,in=120] (a4);
        \draw[-stealth, thick] (a4) to[out=0,in=180] (a5);
    	\draw[-stealth, thick] (a4) to[out=60,in=120] (a5);
        \draw[-stealth, thick] (a5) to[out=0,in=180] (a6);
    	\draw[-stealth, thick] (a5) to[out=60,in=120] (a6);
        \node[] at (2.5,-0.5) {\textcolor{orange}{$0$}};
        \node[] at (3.5,-0.5) {\textcolor{orange}{$0$}};
        \node[] at (4.5,-0.5) {\textcolor{orange}{$0$}};
    
    \end{scope}
    
    \draw[thick] (2.5,0) to (-2,2);
    \draw[thick] (2.5,0) to (2.5,2);
    \draw[thick] (2.5,0) to (7,2);
    
    \draw[thick] (2.5,10) to (-2,8);
    \draw[thick] (2.5,10) to (2.5,8);
    \draw[thick] (2.5,10) to (7,8);
    
    \draw[thick] (-2,4) to (-2,6);
    \draw[thick] (-2,4) to (2.5,6);
    
    \draw[thick] (7,4) to (2.5,6);
    \draw[thick] (7,4) to (7,6);
    
    \draw[thick] (2.5,4) to (-2,6);
    \draw[thick] (2.5,4) to (7,6);
    
    \node[] at (4.7,-1) {\textcolor{blue}{$120$}};
    
    \node[] at (-4.4,2.7) {\textcolor{blue}{$84$}};
    \node[] at (4.6,2.7) {\textcolor{blue}{$76$}};
    \node[] at (9.4,2.7) {\textcolor{blue}{$84$}};
    
    \node[] at (-4.4,7.1) {\textcolor{blue}{$70$}};
    \node[] at (4.6,7.1) {\textcolor{blue}{$66$}};
    \node[] at (9.4,7.1) {\textcolor{blue}{$70$}};
    
    \node[] at (4.7,10.8) {\textcolor{blue}{$61$}};
    
    \end{tikzpicture}
    \end{center}
    \caption{$\mathscr{F}_{5,3}$ Hasse diagram with associated flow polytope volumes.}
    \label{fig:n=5HasseVolume}
    \end{figure}
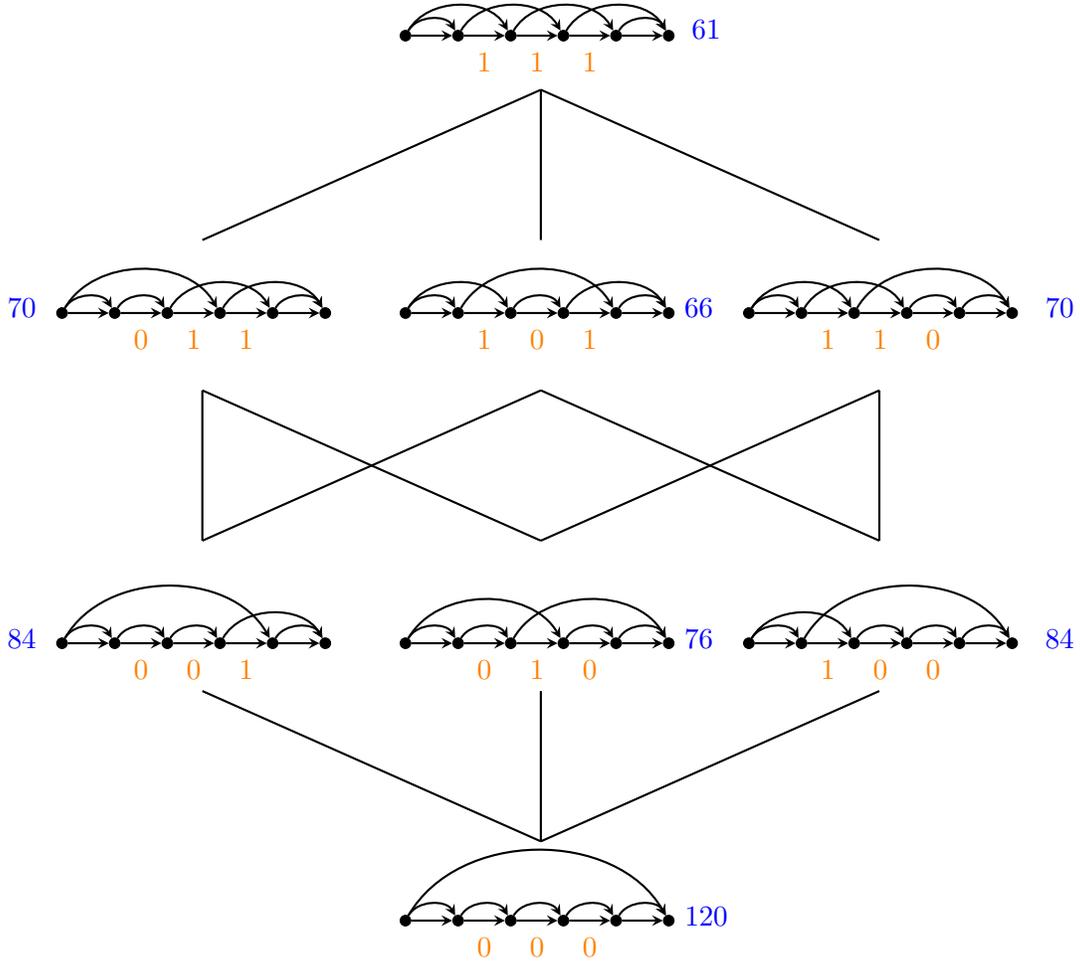

We begin with some notation and definitions regarding partially-determined flow values for a DAG.

\begin{definition}
    Let $G$ be a DAG on vertex set $[n+1]$ with netflow vector $\a=(a_1,\ldots,a_n,-\sum_i a_i)$. 
    Let $L\subseteq E(G)$ and $M\subseteq [n+1]$. 
    Define a \textit{partial flow $f_{G,L}:L\ra \R_{\geq0}$} to be an assignment of flow values to the subset of edges $L$. 
    We say a \emph{partial flow} $f_{G,L}$ is \textit{valid on $M$} if the assigned flow values respect the netflow values $a_i$ for all $i\in M$. 
\end{definition}

    For convenience, we introduce the following notation. For $G$ a DAG with vertices $[n+1]$, we set  
    \[ 
    L_k:=\{(i,j)\in E(G) \ |\ i\leq k\}. 
    \]

\begin{definition}
    Let $G$ be a DAG with vertices $[n+1]$ and $\a=(a_1,\ldots,a_n,-\sum_i a_i)$ be a netflow vector. 
    We define a tree $T_G(\a)$, called the \emph{flow decomposition tree}, 
    as follows:
    \begin{itemize}
        \item The nodes of $T_G(\a)$ are of the form $(G,i,f_{G,L_i})$, i.e., they consist of the DAG $G$ with a partial flow $f_{G,L_i}$ valid on $[i]$, for each $0\leq i\leq n$. 
        Note that when $i=0$, we consider the node $(G,0,f_{G,L_0})$ to be $G$ with no flow values assigned.
        \item The edges of $T_G(\a)$ are between $(G,i,f_{G,L_i})$ and $(G,i+1,g_{G,L_{i+1}})$ whenever
        \[ 
        f_{G,L_i} = g_{G,L_{i+1}}|_{L_i}\, . 
        \] 
    \end{itemize}
    
\end{definition}

We will use the word ``node'' when referring to flow decomposition trees and the word ``vertex'' when referring to DAGs to help distinguish the different objects. 
Note that the data of a node $(G,i,f_{G,L_i})$ in $T_G(\a)$ includes the DAG $G$, a vertex $i$, and a partial flow on $L_i$ valid for $[i]$. 
    However, we can represent the node $(G,i,f_{G,L_i})$ by just its partial flow $f_{G,L_i}$, since the DAG is already part of the partial flow's data, and we choose to make the set of vertices on which the partial flow is valid to match the subscript of $L_i$. 
    We will use this abbreviation convention unless specifying the full node-data is necessary. 

\begin{definition} \label{def:LeafCount}
    Let $f_{G,L_i}$ be a partial flow on $G$ valid on $[n]$ that determines a node in $T_G(\a)$. 
    Denote the number of branches at the node $f_{G,L_{i}}$ by $b(f_{G,L_i})$. 
    Similarly, denote the number of leaves at the node $f_{G,L_{i}}$ by $l(f_{G,L_i})$.
\end{definition}

Observe that $T_G(\a)$ is a rooted tree whose leaves are (complete) flows on $G$. 
    Thus 
    \[ 
    K_G(\a) = \# \text{ of leaves in } T_G(\a) = l(G,0,f_{G,L_0}). 
    \]     
 Further, we can generate $T_G(\a)$ inductively using the following process.
    \begin{enumerate}
        \item Begin with the node $f_{G,L_0}$, which is the root. 
        \item Compute all possible partial flows on $L_1$ valid for vertex $1$, and create a node for each. 
        These nodes form the second level of $T_G(\a)$.
        \item Continue this process until termination. 
        For a node of the form $f_{G,L_i}$, its children are precisely the nodes of the form $g_{G,L_{i+1}}$ where $g_{G,L_{i+1}}|_{L_i} = f_{G,L_i}$ (just as partial flows, not as nodes). 
        In other words, given a partial flow on $L_i$ valid on $[i]$, assign flows to edges out of vertex $i+1$ so that netflow at vertex $i+1=a_{i+1}$
        \item A node is a leaf if $i=n+1$ (flows have been assigned to all edges). 
    \end{enumerate}
    We say a node $f_{G,L_{i}}$ in $T_G(\a)$ is \emph{at (or on) level} $i+1$. We denote the set of level $i+1$ nodes in $T_G(\a)$ by $V_{i+1}(T_G(\a))$, or $V_{i+1}(G)$ if the netflow vector is understood, or just $V_{i+1}$ if both the DAG and netflow vector are clear. 
    The branches of a level $i$ node, which are level $i+1$ nodes, correspond with the cases of partial flows on $L_i$ valid on $[i]$. 
    Further, looking at all the level $i+1$ nodes of $T_G(\a)$, we see all possible partial flows on $L_i$ valid on $[i]$.
    Further, the following are equivalent:
    \begin{itemize}
        \item A partial flow with outflow determined for vertices $[i]$, respecting the netflows on $[i]$;
        \item A partial flow $f_{G,L_i}$ valid on vertices $[i]$; and
        \item A level $i+1$ node in $T_G(\a)$.
    \end{itemize}

\begin{example}
    See Figure~\ref{fig:K_4fd-tree} below for $T_{K_4}(1,1,1,-3)$. 
    Each node of the flow decomposition tree is a copy of $K_4$ (vertices labelled $1$ to $4$, but suppressed) with a partial flow. 
    In the partial flows, unassigned and flow $0$ edges are dashed, edges with flow $1$ are solid, and flows greater than $1$ are labelled. 
\end{example}

\begin{figure}
\begin{center}
\begin{tikzpicture}[scale=0.95]


    \begin{scope}[xshift=175, yshift=240, scale=0.7]
	\vertex[fill](a1) at (1,0) {};
	\vertex[fill](a2) at (2,0) {};
	\vertex[fill](a3) at (3,0) {};
	\vertex[fill](a4) at (4,0) {};
	\draw[-stealth, dashed] (a1) to[out=0,in=180] (a2);
    \draw[-stealth, dashed] (a1) to[out=60,in=120] (a3);
    \draw[-stealth, dashed] (a1) to[out=60,in=120] (a4);
    \draw[-stealth, dashed] (a2) to[out=0,in=180] (a3);
	\draw[-stealth, dashed] (a2) to[out=60,in=120] (a4);
    \draw[-stealth, dashed] (a3) to[out=0,in=180] (a4);

    \end{scope}

    \begin{scope}[xshift=35, yshift=160, scale=0.7]
	\vertex[fill](a1) at (1,0) {};
	\vertex[fill](a2) at (2,0) {};
	\vertex[fill](a3) at (3,0) {};
	\vertex[fill](a4) at (4,0) {};
	\draw[-stealth, dashed] (a1) to[out=0,in=180] (a2);
    \draw[-stealth, dashed] (a1) to[out=60,in=120] (a3);
    \draw[-stealth, thick] (a1) to[out=60,in=120] (a4);
    \draw[-stealth, dashed] (a2) to[out=0,in=180] (a3);
	\draw[-stealth, dashed] (a2) to[out=60,in=120] (a4);
    \draw[-stealth, dashed] (a3) to[out=0,in=180] (a4);

    \end{scope}

    \begin{scope}[xshift=175, yshift=160, scale=0.7]
	\vertex[fill](a1) at (1,0) {};
	\vertex[fill](a2) at (2,0) {};
	\vertex[fill](a3) at (3,0) {};
	\vertex[fill](a4) at (4,0) {};
	\draw[-stealth, dashed] (a1) to[out=0,in=180] (a2);
    \draw[-stealth, thick] (a1) to[out=60,in=120] (a3);
    \draw[-stealth, dashed] (a1) to[out=60,in=120] (a4);
    \draw[-stealth, dashed] (a2) to[out=0,in=180] (a3);
	\draw[-stealth, dashed] (a2) to[out=60,in=120] (a4);
    \draw[-stealth, dashed] (a3) to[out=0,in=180] (a4);

    \end{scope}

    \begin{scope}[xshift=350, yshift=160, scale=0.7]
	\vertex[fill](a1) at (1,0) {};
	\vertex[fill](a2) at (2,0) {};
	\vertex[fill](a3) at (3,0) {};
	\vertex[fill](a4) at (4,0) {};
	\draw[-stealth, thick] (a1) to[out=0,in=180] (a2);
    \draw[-stealth, dashed] (a1) to[out=60,in=120] (a3);
    \draw[-stealth, dashed] (a1) to[out=60,in=120] (a4);
    \draw[-stealth, dashed] (a2) to[out=0,in=180] (a3);
	\draw[-stealth, dashed] (a2) to[out=60,in=120] (a4);
    \draw[-stealth, dashed] (a3) to[out=0,in=180] (a4);

    \end{scope}

    \begin{scope}[xshift=0, yshift=80, scale=0.7]
	\vertex[fill](a1) at (1,0) {};
	\vertex[fill](a2) at (2,0) {};
	\vertex[fill](a3) at (3,0) {};
	\vertex[fill](a4) at (4,0) {};
	\draw[-stealth, dashed] (a1) to[out=0,in=180] (a2);
    \draw[-stealth, dashed] (a1) to[out=60,in=120] (a3);
    \draw[-stealth, thick] (a1) to[out=60,in=120] (a4);
    \draw[-stealth, dashed] (a2) to[out=0,in=180] (a3);
	\draw[-stealth, thick] (a2) to[out=60,in=120] (a4);
    \draw[-stealth, dashed] (a3) to[out=0,in=180] (a4);

    \end{scope}

    \begin{scope}[xshift=70, yshift=80, scale=0.7]
	\vertex[fill](a1) at (1,0) {};
	\vertex[fill](a2) at (2,0) {};
	\vertex[fill](a3) at (3,0) {};
	\vertex[fill](a4) at (4,0) {};
	\draw[-stealth, dashed] (a1) to[out=0,in=180] (a2);
    \draw[-stealth, dashed] (a1) to[out=60,in=120] (a3);
    \draw[-stealth, thick] (a1) to[out=60,in=120] (a4);
    \draw[-stealth, thick] (a2) to[out=0,in=180] (a3);
	\draw[-stealth, dashed] (a2) to[out=60,in=120] (a4);
    \draw[-stealth, dashed] (a3) to[out=0,in=180] (a4);

    \end{scope}

    \begin{scope}[xshift=140, yshift=80, scale=0.7]
	\vertex[fill](a1) at (1,0) {};
	\vertex[fill](a2) at (2,0) {};
	\vertex[fill](a3) at (3,0) {};
	\vertex[fill](a4) at (4,0) {};
	\draw[-stealth, dashed] (a1) to[out=0,in=180] (a2);
    \draw[-stealth, thick] (a1) to[out=60,in=120] (a3);
    \draw[-stealth, dashed] (a1) to[out=60,in=120] (a4);
    \draw[-stealth, dashed] (a2) to[out=0,in=180] (a3);
	\draw[-stealth, thick] (a2) to[out=60,in=120] (a4);
    \draw[-stealth, dashed] (a3) to[out=0,in=180] (a4);

    \end{scope}

    \begin{scope}[xshift=210, yshift=80, scale=0.7]
	\vertex[fill](a1) at (1,0) {};
	\vertex[fill](a2) at (2,0) {};
	\vertex[fill](a3) at (3,0) {};
	\vertex[fill](a4) at (4,0) {};
	\draw[-stealth, dashed] (a1) to[out=0,in=180] (a2);
    \draw[-stealth, thick] (a1) to[out=60,in=120] (a3);
    \draw[-stealth, dashed] (a1) to[out=60,in=120] (a4);
    \draw[-stealth, thick] (a2) to[out=0,in=180] (a3);
	\draw[-stealth, dashed] (a2) to[out=60,in=120] (a4);
    \draw[-stealth, dashed] (a3) to[out=0,in=180] (a4);

    \end{scope}

    \begin{scope}[xshift=280, yshift=80, scale=0.7]
	\vertex[fill](a1) at (1,0) {};
	\vertex[fill](a2) at (2,0) {};
	\vertex[fill](a3) at (3,0) {};
	\vertex[fill](a4) at (4,0) {};
	\draw[-stealth, thick] (a1) to[out=0,in=180] (a2);
    \draw[-stealth, dashed] (a1) to[out=60,in=120] (a3);
    \draw[-stealth, dashed] (a1) to[out=60,in=120] (a4);
    \draw[-stealth, dashed] (a2) to[out=0,in=180] (a3);
	\draw[-stealth, thick] (a2) to[out=60,in=120] (a4);
    \draw[-stealth, dashed] (a3) to[out=0,in=180] (a4);
    \node[] at (2.1,0.4) {\tiny 2};

    \end{scope}

    \begin{scope}[xshift=350, yshift=80, scale=0.7]
	\vertex[fill](a1) at (1,0) {};
	\vertex[fill](a2) at (2,0) {};
	\vertex[fill](a3) at (3,0) {};
	\vertex[fill](a4) at (4,0) {};
	\draw[-stealth, thick] (a1) to[out=0,in=180] (a2);
    \draw[-stealth, dashed] (a1) to[out=60,in=120] (a3);
    \draw[-stealth, dashed] (a1) to[out=60,in=120] (a4);
    \draw[-stealth, thick] (a2) to[out=0,in=180] (a3);
	\draw[-stealth, thick] (a2) to[out=60,in=120] (a4);
    \draw[-stealth, dashed] (a3) to[out=0,in=180] (a4);

    \end{scope}

    \begin{scope}[xshift=420, yshift=80, scale=0.7]
	\vertex[fill](a1) at (1,0) {};
	\vertex[fill](a2) at (2,0) {};
	\vertex[fill](a3) at (3,0) {};
	\vertex[fill](a4) at (4,0) {};
	\draw[-stealth, thick] (a1) to[out=0,in=180] (a2);
    \draw[-stealth, dashed] (a1) to[out=60,in=120] (a3);
    \draw[-stealth, dashed] (a1) to[out=60,in=120] (a4);
    \draw[-stealth, thick] (a2) to[out=0,in=180] (a3);
	\draw[-stealth, dashed] (a2) to[out=60,in=120] (a4);
    \draw[-stealth, dashed] (a3) to[out=0,in=180] (a4);
    \node[] at (2.4,-0.2) {\tiny 2};

    \end{scope}

    \begin{scope}[xshift=0, yshift=0, scale=0.7]
	\vertex[fill](a1) at (1,0) {};
	\vertex[fill](a2) at (2,0) {};
	\vertex[fill](a3) at (3,0) {};
	\vertex[fill](a4) at (4,0) {};
	\draw[-stealth, dashed] (a1) to[out=0,in=180] (a2);
    \draw[-stealth, dashed] (a1) to[out=60,in=120] (a3);
    \draw[-stealth, thick] (a1) to[out=60,in=120] (a4);
    \draw[-stealth, dashed] (a2) to[out=0,in=180] (a3);
	\draw[-stealth, thick] (a2) to[out=60,in=120] (a4);
    \draw[-stealth, thick] (a3) to[out=0,in=180] (a4);

    \end{scope}

    \begin{scope}[xshift=70, yshift=0, scale=0.7]
	\vertex[fill](a1) at (1,0) {};
	\vertex[fill](a2) at (2,0) {};
	\vertex[fill](a3) at (3,0) {};
	\vertex[fill](a4) at (4,0) {};
	\draw[-stealth, dashed] (a1) to[out=0,in=180] (a2);
    \draw[-stealth, dashed] (a1) to[out=60,in=120] (a3);
    \draw[-stealth, thick] (a1) to[out=60,in=120] (a4);
    \draw[-stealth, thick] (a2) to[out=0,in=180] (a3);
	\draw[-stealth, dashed] (a2) to[out=60,in=120] (a4);
    \draw[-stealth, thick] (a3) to[out=0,in=180] (a4);
    \node[] at (3.4,-0.2) {\tiny 2};

    \end{scope}

    \begin{scope}[xshift=140, yshift=0, scale=0.7]
	\vertex[fill](a1) at (1,0) {};
	\vertex[fill](a2) at (2,0) {};
	\vertex[fill](a3) at (3,0) {};
	\vertex[fill](a4) at (4,0) {};
	\draw[-stealth, dashed] (a1) to[out=0,in=180] (a2);
    \draw[-stealth, thick] (a1) to[out=60,in=120] (a3);
    \draw[-stealth, dashed] (a1) to[out=60,in=120] (a4);
    \draw[-stealth, dashed] (a2) to[out=0,in=180] (a3);
	\draw[-stealth, thick] (a2) to[out=60,in=120] (a4);
    \draw[-stealth, thick] (a3) to[out=0,in=180] (a4);
    \node[] at (3.4,-0.2) {\tiny 2};

    \end{scope}

    \begin{scope}[xshift=210, yshift=0, scale=0.7]
	\vertex[fill](a1) at (1,0) {};
	\vertex[fill](a2) at (2,0) {};
	\vertex[fill](a3) at (3,0) {};
	\vertex[fill](a4) at (4,0) {};
	\draw[-stealth, dashed] (a1) to[out=0,in=180] (a2);
    \draw[-stealth, thick] (a1) to[out=60,in=120] (a3);
    \draw[-stealth, dashed] (a1) to[out=60,in=120] (a4);
    \draw[-stealth, thick] (a2) to[out=0,in=180] (a3);
	\draw[-stealth, dashed] (a2) to[out=60,in=120] (a4);
    \draw[-stealth, thick] (a3) to[out=0,in=180] (a4);
    \node[] at (3.4,-0.2) {\tiny 3};

    \end{scope}

    \begin{scope}[xshift=280, yshift=0, scale=0.7]
	\vertex[fill](a1) at (1,0) {};
	\vertex[fill](a2) at (2,0) {};
	\vertex[fill](a3) at (3,0) {};
	\vertex[fill](a4) at (4,0) {};
	\draw[-stealth, thick] (a1) to[out=0,in=180] (a2);
    \draw[-stealth, dashed] (a1) to[out=60,in=120] (a3);
    \draw[-stealth, dashed] (a1) to[out=60,in=120] (a4);
    \draw[-stealth, dashed] (a2) to[out=0,in=180] (a3);
	\draw[-stealth, thick] (a2) to[out=60,in=120] (a4);
    \draw[-stealth, thick] (a3) to[out=0,in=180] (a4);
    \node[] at (2.1,0.4) {\tiny 2};

    \end{scope}

    \begin{scope}[xshift=350, yshift=0, scale=0.7]
	\vertex[fill](a1) at (1,0) {};
	\vertex[fill](a2) at (2,0) {};
	\vertex[fill](a3) at (3,0) {};
	\vertex[fill](a4) at (4,0) {};
	\draw[-stealth, thick] (a1) to[out=0,in=180] (a2);
    \draw[-stealth, dashed] (a1) to[out=60,in=120] (a3);
    \draw[-stealth, dashed] (a1) to[out=60,in=120] (a4);
    \draw[-stealth, thick] (a2) to[out=0,in=180] (a3);
	\draw[-stealth, thick] (a2) to[out=60,in=120] (a4);
    \draw[-stealth, thick] (a3) to[out=0,in=180] (a4);
    \node[] at (3.4,-0.2) {\tiny 2};

    \end{scope}

    \begin{scope}[xshift=420, yshift=0, scale=0.7]
	\vertex[fill](a1) at (1,0) {};
	\vertex[fill](a2) at (2,0) {};
	\vertex[fill](a3) at (3,0) {};
	\vertex[fill](a4) at (4,0) {};
	\draw[-stealth, thick] (a1) to[out=0,in=180] (a2);
    \draw[-stealth, dashed] (a1) to[out=60,in=120] (a3);
    \draw[-stealth, dashed] (a1) to[out=60,in=120] (a4);
    \draw[-stealth, thick] (a2) to[out=0,in=180] (a3);
	\draw[-stealth, dashed] (a2) to[out=60,in=120] (a4);
    \draw[-stealth, thick] (a3) to[out=0,in=180] (a4);
    \node[] at (2.4,-0.2) {\tiny 2};
    \node[] at (3.4,-0.2) {\tiny 3};

    \end{scope}


    \draw[thick] (1.75,1) to (1.75,2.4);
    \draw[thick] (4.2,1) to (4.2,2.4);
    \draw[thick] (6.65,1) to (6.65,2.4);
    \draw[thick] (9.1,1) to (9.1,2.4);
    \draw[thick] (11.55,1) to (11.55,2.4);
    \draw[thick] (14,1) to (14,2.4);
    \draw[thick] (16.45,1) to (16.45,2.4);

    \draw[thick] (1.75,3.8) to (2.975,5.2);
    \draw[thick] (4.2,3.8) to (2.975,5.2);
    \draw[thick] (6.65,3.8) to (7.875,5.2);
    \draw[thick] (9.1,3.8) to (7.875,5.2);
    \draw[thick] (11.55,3.8) to (14,5.2);
    \draw[thick] (14,3.8) to (14,5.2);
    \draw[thick] (16.45,3.8) to (14,5.2);

    \draw[thick] (2.975,6.6) to (7.875,8);
    \draw[thick] (7.875,6.6) to (7.875,8);
    \draw[thick] (14,6.6) to (7.875,8);

\end{tikzpicture}
\end{center}
\caption{The flow decomposition tree $T_{K_4}(1,1,1,-3)$.}
\label{fig:K_4fd-tree}
\end{figure}
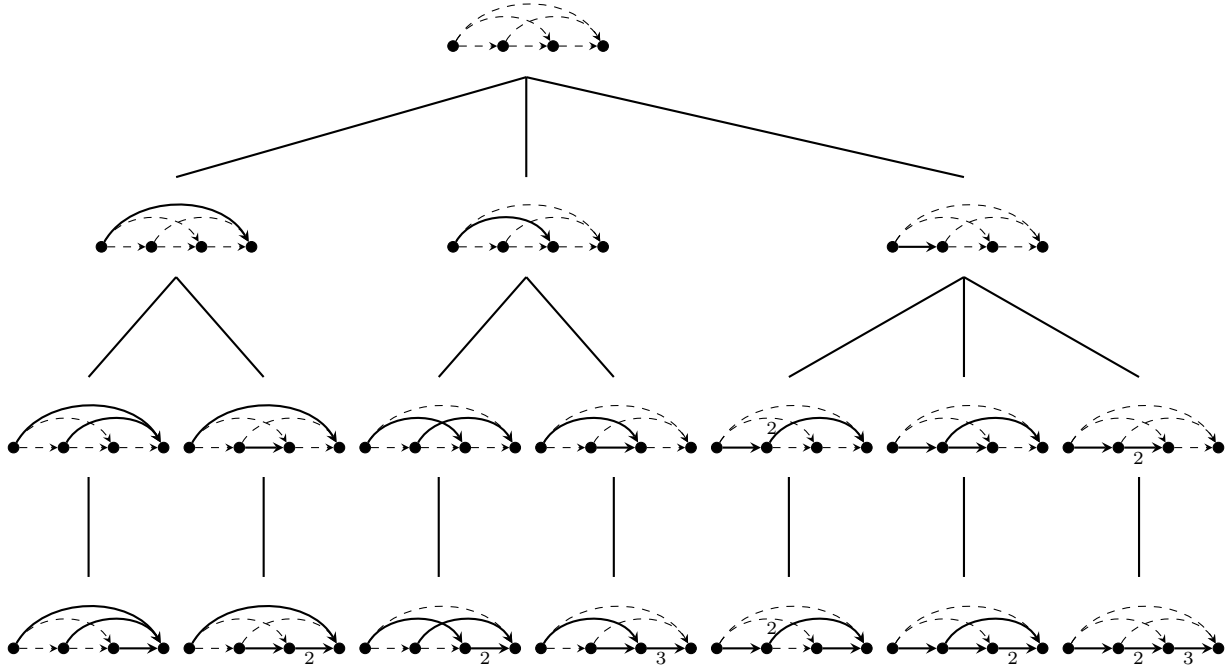

In the following proposition, we show that the number of branches at a given node in the flow decomposition tree is determined by the partial flow, netflow vector, and out-degree sequence.

\begin{proposition} \label{prop:BranchCountGeneral}
    Let $f_{G,L_{i-1}}$ be a level $i$ node in $T_G(\a)$ for some $1\leq i\leq n$. 
    Suppose the partial flow $f_{G,L_{i-1}}$ has $t$ units of inflow entering vertex $i$. 
    Then the number of branches of the node is
    \begin{align*}
       b(f_{G,L_{i-1}}) &= \binom{t+a_i+\outd_i-1}{\outd_i-1}.
    \end{align*}   
\end{proposition}

\begin{proof}
    To extend $f_{G,L_{i-1}}$ to a $[i]$-valid partial flow on $L_{i}$, we need to respect the netflow at vertex $i$. 
Thus we need the outflow at $i$ to be $t+a_i$ units. 
This outgoing flow is distributed over the outgoing edges from $i$. 
Hence the number of branches is the number of weak compositions of $t+a_i$ into $\outd_i$ parts. 
This quantity is counted by 
$\displaystyle \binom{t+a_i+\outd_i-1}{\outd_i-1}. $
\end{proof}

For $G\in \mathscr{F}_{n,3}$, we have the following special case. 

\begin{corollary} \label{cor:BranchCountSpeial}
    Let $G\in \mathscr{F}_{n,3}$ and netflow vector $\w_n$. 
    Then the number of branches of a node $f_{G,L_{i-1}}$ will be 
    \begin{align*}
        b(f_{G,L_{i-1}}) &= \left\{
        \begin{array}{ll}
            1, & i=1 \\
            2, & i=2 \\
            t+2, &  i\geq3
        \end{array}
        \right. .
    \end{align*}
\end{corollary}

\begin{proof}
    By assumption $G\in \mathscr{F}_{n,3}$, 
$\outd_1=3$ and $\outd_i=2$ for interior vertices $i$. 
Further, assuming netflow vector to be $\w_n$ gives $a_1=0$ and $a_i=1$ for interior vertices $i$. 
Thus
\begin{align*}
    b(f_{G,L_{i-1}})
    &= \left\{
        \begin{array}{ll}
            \binom{0+0+3-1}{3-1}, & i=1 \\
            \binom{0+1+2-1}{2-1}, & i=2 \\
            \binom{t+1+2-1}{2-1}, &  i\geq3
        \end{array}
        \right. 
        = \left\{
        \begin{array}{ll}
            1, & i=1 \\
            2, & i=2 \\
            t+2, &  i\geq3
        \end{array}
        \right. .
\end{align*}
\end{proof}

The following lemma establishes a relationship between level $i$ nodes in flow decomposition trees for DAGs related by an edge interchange.

\begin{lemma} \label{lemma:Levelb+1Corr}
    Let $G\in\mathscr{F}_{n,3}$. 
    Suppose $G'$ is obtained by interchanging edges $(a,d)$ and $(b,b+1)$ in $G$. 
    Then for every $1\leq i\leq b+1$, the nodes of level $i$ in $T_G(\w_n)$ are in bijection with the nodes of level $i$ in $T_{G'}(\w_n)$.
    This bijection is denoted $\phi$.
\end{lemma}

\begin{proof}
    We construct a map $\phi$ between the level $i$ nodes of $T_{G'}(\w_n)$ and $T_{G}(\w_n)$ for $1\leq i\leq b+1$. 
    Given a level $i$ node $f_{G',L_{i-1}}$ in $T_{G'}(\w_n)$, we map it to the following node in $T_G(\w_n)$: 
    \[ 
    \phi(f_{G',L_{i-1}})(j,k):=f_{G,L_{i-1}}(j,k) = \left\{ 
    \begin{array}{ll}
        f_{G',L_{i-1}}(a,b+1), &  (j,k)=(a,d) \\
        f_{G',L_{i-1}}(b,d), & (j,k)=(b,b+1) \\
        f_{G',L_{i-1}}(j,k), & \text{else}
    \end{array}\right. 
    \]
    
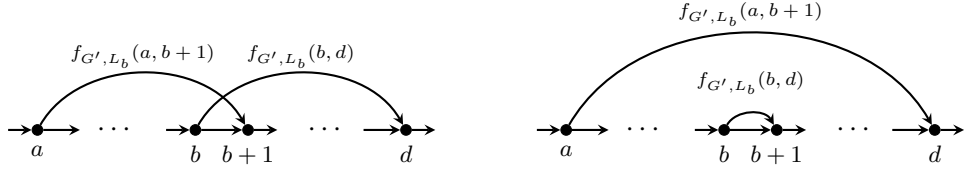
\begin{figure}
\begin{center}
\begin{tikzpicture}

\begin{scope}[xshift=0, yshift=0, scale=0.7]

	\vertex[fill,label=below:\footnotesize{$a$}](a1) at (1,0) {};
	\node(a2) at (2,0) {};
	\vertex[fill,label=below:\footnotesize{$b$}](a3) at (4,0) {};
	\vertex[fill,label=below:\footnotesize{$b+1$}](a4) at (5,0) {};
	\node(a5) at (7,0) {};
	\vertex[fill,label=below:\footnotesize{$d$}](a6) at (8,0) {};
        \node(b2) at (0.25,0) {};
        \node(b4) at (2.75,0) {};
        \node(b6) at (3.25,0) {};
        \node(b8) at (5.75,0) {};
        \node(b10) at (6.25,0) {};
        \node(b12) at (8.75,0) {};
        \draw[-stealth, thick] (b2) to[out=0,in=180] (a1);
        \draw[-stealth, thick] (a1) to[out=0,in=180] (a2);
        \draw[-stealth, thick] (a1) to[out=60,in=120] (a4);
        \draw[-stealth, thick] (b6) to[out=0,in=180] (a3);
        \draw[-stealth, thick] (a3) to[out=0,in=180] (a4);
	\draw[-stealth, thick] (a3) to[out=60,in=120] (a6);
        \draw[-stealth, thick] (a4) to[out=0,in=180] (b8);
        \draw[-stealth, thick] (a5) to[out=0,in=180] (a6);
        \draw[-stealth, thick] (a6) to[out=0,in=180] (b12);
        \node[] at (2.5,0) {$\cdots$};
        \node[] at (6.5,0) {$\cdots$};
        \node[] at (3,1.4) {\tiny $f_{G',L_b}(a,b+1)$};
        \node[] at (6,1.4) {\tiny $f_{G',L_b}(b,d)$};

\end{scope}

\begin{scope}[xshift=200, yshift=0, scale=0.7]

	\vertex[fill,label=below:\footnotesize{$a$}](a1) at (1,0) {};
	\node(a2) at (2,0) {};
	\vertex[fill,label=below:\footnotesize{$b$}](a3) at (4,0) {};
	\vertex[fill,label=below:\footnotesize{$b+1$}](a4) at (5,0) {};
	\node(a5) at (7,0) {};
	\vertex[fill,label=below:\footnotesize{$d$}](a6) at (8,0) {};
        \node(b2) at (0.25,0) {};
        \node(b4) at (2.75,0) {};
        \node(b6) at (3.25,0) {};
        \node(b8) at (5.75,0) {};
        \node(b10) at (6.25,0) {};
        \node(b12) at (8.75,0) {};
        \draw[-stealth, thick] (b2) to[out=0,in=180] (a1);
        \draw[-stealth, thick] (a1) to[out=0,in=180] (a2);
        \draw[-stealth, thick] (a1) to[out=60,in=120] (a6);
        \draw[-stealth, thick] (b6) to[out=0,in=180] (a3);
        \draw[-stealth, thick] (a3) to[out=0,in=180] (a4);
	\draw[-stealth, thick] (a3) to[out=60,in=120] (a4);
        \draw[-stealth, thick] (a4) to[out=0,in=180] (b8);
        \draw[-stealth, thick] (a5) to[out=0,in=180] (a6);
        \draw[-stealth, thick] (a6) to[out=0,in=180] (b12);
        \node[] at (2.5,0) {$\cdots$};
        \node[] at (6.5,0) {$\cdots$};
        \node[] at (4.5,2.2) {\tiny $f_{G',L_b}(a,b+1)$};
        \node[] at (4.5,0.9) {\tiny $f_{G',L_b}(b,d)$};

\end{scope}

\end{tikzpicture}
\end{center}
\caption{A level $b+1$ node $f_{G',L_b}$ and $\phi(f_{G',L_b})$, respectively.}
\label{fig:Levelb+1Corr}
\end{figure}
        
    This defines a partial flow on $L_{i-1}$ in $G$. 
    In order for it to correspond to a node in $T_G(\w_n)$, the partial flow must be valid on on $[i-1]$. 
    Since $G$ and $G'$ have edges in common, we need only verify validity for the vertices where edges have changed. 
    The vertices where edges have changed are $a$, $b$, $b+1$ and $d$. 
    However, since $1\leq i\leq b+1$, the only vertices in $[i-1]$ are $a$ and $b$. 
    But for both $a$ and $b$, this map doesn't change the edges going in to $a$ or $b$, and preserves the outflow from $a$ and $b$. 
    Thus, netflow at $a$ and $b$ is preserved, so $f_{G,L_{i-1}}$ is valid on $a$ and $b$. 

   Given two distinct partial flows in $T_{G'}(\w_n)$, it follows directly from the construction of $\phi$ that their images will have distinct flows. Thus $\phi$ is injective. 

   Given a level $i$ node in $T_G(\w_n)$, call it $f_{G,L_{i-1}}$, consider the following partial flow $f$ on $G'$ defined by
   \[ 
    f(j,k) := \left\{ 
    \begin{array}{ll}
        f_{G,L_{i-1}}(a,d), &  (j,k)=(a,b+1) \\
        f_{G,L_{i-1}}(b,b+1), & (j,k)=(b,d) \\
        f_{G,L_{i-1}}(j,k), & \text{else}
    \end{array}\right. .
    \]
    Next we verify that $f$ is a level $i$ node in $T_{G'}(\w_n)$. 
    Observe that it is equivalent to check that $f$ is partial flow on $L_{i-1}$ valid on vertices $[i-1]$ (all in $G$).
    If we look at the subgraph of $G'$ induced by the vertices $[i-1]\setminus\{a,b\}$, $G'$ agrees with $G$ and $f$ agrees with $f_{G,L_{i-1}}$. 
    Further, $f$ is defined in a way that preserves the netflow at vertices $a$ and $b$. 

    Finally, we show that $\phi$ maps $f$ to $f_{G,L_{i-1}}$. For every edge $(j,k)\in L_{i-1}$ in $G$, applying our construction for $\phi$ followed by our assumptions about $f$ gives
    \begin{align*}
        \phi(f)(j,k) &= \left\{ 
    \begin{array}{ll}
        f(a,b+1), &  (j,k)=(a,d) \\
        f(b,d), & (j,k)=(b,b+1) \\
        f(j,k), & \text{else}
    \end{array}\right. \\
    &= \left\{ 
    \begin{array}{ll}
        f_{G,L_{i-1}}(a,d), &  (j,k)=(a,d) \\
        f_{G,L_{i-1}}(b,b+1), & (j,k)=(b,b+1) \\
        f_{G,L_{i-1}}(j,k), & \text{else}
    \end{array}\right. \\
    &= f_{G,L_{i-1}}(j,k)
    \end{align*} \, .
\end{proof} 

\begin{example} \label{ex:Levelb+1Corr}
    In Figure~\ref{fig:fd-tree11} and Figure~\ref{fig:fd-tree10} we have the flow decomposition trees $T_{G_{11}}(\w_4)$ and $T_{G_{10}}(\w_4)$, respectively.
    In order to fit the tree on the page, the bottom-level nodes have been stacked vertically below their ancestors. 
    In this example, $G_{11}$ can be obtained by performing an interchange operation on $G_{10}$ with edges $(2,5)$ and $(3,4)$.
    Thus, we have $a=2$, $b=3$, and $d=5$. 
    Observe that the for each level up to $b+1=4$, the nodes of $T_{G_{11}}(\w_4)$ and $T_{G_{10}}(\w_4)$ are in bijection. 
    The flow decomposition trees have been arranged in a way that the bijective pairs lie in the same relative position in their respective trees. 
    It is also worth noting that for the netflow vector $\w_n$, the outflow at vertex $1$ is $0$, so the root always has a single branch. 
\end{example}

\begin{figure}
\begin{center}

\end{center}
\caption{$T_{G_{11}}(\w_4)$.}
\label{fig:fd-tree11}
\end{figure}

\begin{figure}
\begin{center}
%
\end{center}
\caption{$T_{G_{10}}(\w_4)$.}
\label{fig:fd-tree10}
\end{figure}

Now that we have established basic properties of flow decomposition trees, we seek to organize these DAGs in a way that ensures $T_{G'}(\w_n)$ has fewer leaves than $T_G(\w_n)$, where $G\in\mathscr{F}_{n,3}$ has nested edges that are interchanged to obtain $G'$. 
To do this, we focus on the level $b+1$ nodes. 
Throughout the remainder of this section, assume $G$ and $G'$ are related by the edge interchange structure just described.

\begin{lemma} \label{lemma:BadNodes}
  Assume $G\in\mathscr{F}_{n,3}$ has nested edges that are interchanged to obtain $G'$. 
    Let $f_{G',L_b}$ be a level $b+1$ node in $T_{G'}(\w_n)$, with corresponding level $b+1$ node $\phi(f_{G',L_b})$ in $T_{G}(\w_n)$.
    Then 
    \[
    l(f_{G',L_b}) > l(\phi(f_{G',L_b}))
    \]
    if and only if $f_{G',L_b}(a,b+1)>f_{G',L_b}(b,d)$.
\end{lemma}

\begin{proof}
    Let $f_{G',L_b}$ be a level $b+1$ node in $T_{G'}(\w_n)$.
    Note that a leaf of $f_{G',L_b}$ is a (complete) flow $f$ on $G'$ that agrees with $f_{G',L_b}$. 
    Similarly, a leaf in $T_{G}(\w_n)$ of the level $b+1$ node $\phi(f_{G',L_b})$ is a (complete) flow $g$ on $G$ that agrees with $\phi(f_{G',L_b})$. 
    
    For the backward implication of our if and only if, suppose $f_{G',L_b}(a,b+1)>f_{G',L_b}(b,d)$. We build an injection from the leaves of $\phi(f_{G',L_b})$ into the leaves of $f_{G',L_b}$, then show that the injection is not surjective.
    Given a leaf $g$ of the level $b+1$ node $\phi(f_{G',L_b})$ in $T_G(\w_n)$, we map it to the leaf $f$ of $T_{G'}(\w_n)$, where $f$ is defined as 
    \[
        f(i,j) = \left\{ \begin{array}{ll}
            g(a,d), & (i,j)=(a,b+1) \\
            g(b,b+1), & (i,j)=(b,d) \\
            g(e_i)+g(a,d)-g(b,b+1), & (i,j)=e_i, \ b+1\leq i\leq d-1 \\
            g(i,j), & \text{else}
        \end{array}
        \right. .
    \]
    This map preserves the flows on the non-spine edges leaving $a$ and $b$, which forces more flow into vertex $b+1$ in $G$, which we adjust for by increasing the flow on the spine between vertices $b+1$ and $d$. 
    
    Next, we verify that $f$ is a valid flow on $G'$. That is, the flows on $E(G')$ prescribed by $f$ agree with the netflow vector $\w_n$. For vertices $i\neq b+1,\ldots, d$, we have 
    \begin{align*}
        (f\ \text{outflow at}\ i) - (f\ \text{inflow at}\ i) 
        &= (g\ \text{outflow at}\ i) - (g\ \text{inflow at}\ i) 
        = w_i .
    \end{align*}
    For the vertex $b+1$,
    \begin{align*}
        (f\ \text{outflow at}\ b+1) - (f\ \text{inflow at}\ b+1) 
        =& [f(e_{b+1}) + f(b+1,b+2)] - [f(e_b)+f(a,b+1)] \\
        =& [g(e_{b+1}) + g(a,d)-g(b,b+1) + g(b+1,b+2)] \\
        &- [g(e_b) + g(a,d)] \\
        =& [g(e_{b+1}) + g(b+1,b+2)] - [g(e_b) + g(b,b+1)] \\
        =& (g\ \text{outflow at}\ b+1) - (g\ \text{inflow at}\ b+1) \\
        =& w_{b+1} .
    \end{align*}
    If there are any vertices $b+2\leq i \leq d-1$, note that all four incident edges are consecutive edges, else we would violate the unique overpass condition of Lemma~\ref{lemma:UniqueOverpass}. 
    So we have 
    \begin{align*}
        (f\ \text{outflow at}\ i) - (f\ \text{inflow at}\ i) 
        =& [f(e_{i}) + f(i,i+1)] - [f(e_{i-1})+f(i-1,i)] \\
        =& [g(e_{i}) + g(a,d)-g(b,b+1) + g(i,i+1)] \\
        &- [g(e_{i-1}) + g(a,d)-g(b,b+1) + g(i-1,i)] \\
        =& [g(e_{i}) + g(i,i+1)] - [g(e_{i-1}) + g(i-1,i)] \\
        =& (g\ \text{outflow at}\ i) - (g\ \text{inflow at}\ i) \\
        =& w_{i} .
    \end{align*}
    The last vertex to check is $d$, which has two outgoing edges $e_d$ (the spine) and $(d,j)$ for some $j$. Our netflow is then 
    \begin{align*}
        (f\ \text{outflow at}\ d) - (f\ \text{inflow at}\ d) 
        =& [f(e_{d}) + f(d,j)] - [f(e_{d-1})+f(d-1,d)] \\
        =& [g(e_{d}) + g(d,j)] \\
        &- [g(e_{d-1}) + g(a,d)-g(b,b+1) + g(d-1,d)] \\
        =& [g(e_{d}) + g(d,j)] - [g(e_{d-1}) + g(a,d)] \\
        =& (g\ \text{outflow at}\ d) - (g\ \text{inflow at}\ d) \\
        =& w_{d} .
    \end{align*}
    Thus $f$ is indeed a valid flow on $G'$. Now we show that $f$ agrees with $f_{G',L_b}$. For every edge $(i,j)\in L_b$, we apply the construction of $f$, the fact that $g$ agrees with $\phi(f_{G',L_b})$, and the definition of $\phi$, respectively, to get that 
    \begin{align*}
        f(i,j) &= \left\{ \begin{array}{ll}
            g(a,d), & (i,j)=(a,b+1) \\
            g(b,b+1), & (i,j)=(b,d) \\
            g(i,j), & \text{else}
        \end{array}
        \right. \\
        &= \left\{ \begin{array}{ll}
            \phi(f_{G',L_b})(a,d), & (i,j)=(a,b+1) \\
            \phi(f_{G',L_b})(b,b+1), & (i,j)=(b,d) \\
            \phi(f_{G',L_b})(i,j), & \text{else}
        \end{array}
        \right. \\
        &= \left\{ \begin{array}{ll}
            f_{G',L_b}(a,b+1), & (i,j)=(a,b+1) \\
            f_{G',L_b}(b,d), & (i,j)=(b,d) \\
            f_{G',L_b}(i,j), & \text{else}
        \end{array}
        \right. \\
        &= f_{G',L_b}(i,j) .
    \end{align*}

    At this point, we have only verified that the map $g\mapsto f$ is well-defined. 
    To prove it is an injective map, let $g_1$ and $g_2$ be (flows on $G$ that give) leaves in $T_G(\w_n)$. 
    Suppose that both $g_1$ and $g_2$ map to a leaf $f$ in $T_{G'}(\w_n)$. Then 
    \begin{align*}
        &\left\{ \begin{array}{ll}
            g_1(a,d), & (i,j)=(a,b+1) \\
            g_1(b,b+1), & (i,j)=(b,d) \\
            g_1(e_i)+g_1(a,d)-g_1(b,b+1), & (i,j)=e_i, \ b+1\leq i\leq d-1 \\
            g_1(i,j), & \text{else}
        \end{array}
        \right. \\
        =&
        \left\{ \begin{array}{ll}
            g_2(a,d), & (i,j)=(a,b+1) \\
            g_2(b,b+1), & (i,j)=(b,d) \\
            g_2(e_i)+g_2(a,d)-g_2(b,b+1), & (i,j)=e_i, \ b+1\leq i\leq d-1 \\
            g_2(i,j), & \text{else}
        \end{array}
        \right.  .
    \end{align*}
    We can immediately conclude that $g_1(a,d)=g_2(a,d)$ and $g_1(b,b+1)=g_2(b,b+1)$, and that $g_1(i,j)=g_2(i,j)$ for every edge that isn't a spine edge between $b+1\leq i\leq d-1$.
    Then it follows that $g_1(e_i)=g_2(e_i)$ for every spine edge $e_i$ such that $b+1\leq i\leq d-1$. 
    Hence $g_1=g_2$, so our map is indeed injective. 

    To prove that the inequality $l(f_{G',L_b}) \gneq l(\phi(f_{G',L_b}))$ is strict, we describe a leaf of $f_{G',L_b}$ that cannot be in the image of our injection. 
    Consider the flow $h$ on $G'$, where 
    \[
        h(i,j) = \left\{ \begin{array}{ll}
            f_{G',L_b}(i,j), & (i,j)\in L_b \\
            i-1-f_{G',L_b}(b,d), & (i,j)=(i,i+1), \ b+1\leq i\leq d-2 \\
            d-1-f_{G',L_b}(b,d), & (i,j)=e_{d-1} \\
            i-1, & (i,j)=e_i, \ d\leq i\leq n  \\
            0, & \text{else}
        \end{array}
        \right.  .
    \]
    To verify that $h$ is a valid flow on $G'$, we need only check the netflows at vertices $b+1,\ldots,n$ since $h$ agrees with $f_{G',L_b}$ on other vertices. 
    At vertex $b+1$, 
    \begin{align*}
        (h\ \text{outflow at}\ b+1) - (h\ \text{inflow at}\ b+1) 
        =& [h(b+1,b+2) + h(e_{b+1})] - [h(a,b+1) + h(e_b) ] \\
        =& [b-f_{G',L_b}(b,d) + 0 ] \\
        &- [f_{G',L_b}(a,b+1) + f_{G',L_b}(e_b)] \\
        =& (f_{G',L_b}\ \text{outflow at}\ b+1) - (f_{G',L_b}\ \text{inflow at}\ b+1) 
    \end{align*}
    by Proposition~\ref{prop:OverpassDeterminesInflow}. 

    If there are any vertices $i$ for $b+2\leq i\leq d-2$,
    \begin{align*}
        (h\ \text{outflow at}\ i) - (h\ \text{inflow at}\ i) 
        =& [h(i,i+1) + h(e_{i})] - [h(i-1,i) + h(e_{i-1})] \\
        =& [i-1-f_{G',L_b}(b,d) + 0] - [i-2-f_{G',L_b}(b,d) + 0] \\
        =& 1 \\
        =& w_i .
    \end{align*}
    
   At vertex $d-1$, there is some non-spine edge $(d-1,j)$. 
   Then our netflow is 
   \begin{align*}
        (h\ \text{outflow at}\ d) - (h\ \text{inflow at}\ d) 
        =& [h(d-1,j) + h(e_{d-1})] - [h(d-2,d-1) + h(e_{d-2})] \\
        =& [0 + d-2-f_{G',L_b}(b,d)] \\
        &- [0+ d-3- f_{G',L_b}(b,d) ] \\
        =& 1 \\
        =& w_{d-1} . 
    \end{align*}
            
    At vertex $d$, there is some non-spine edge $(d,j)$. 
    Then our netflow is
    \begin{align*}
        (h\ \text{outflow at}\ d) - (h\ \text{inflow at}\ d) 
        =& [h(d,j) + h(e_d)] - [h(b,d) + h(e_{d-1})] \\
        =& [d-1 + 0] \\
        &- [f_{G',L_b}(b,d) + d-2-f_{G',L_b}(b,d)] \\
        =& 1 \\
        =& w_d . 
    \end{align*}
    
    At a vertex $i$ for $d+1\leq i \leq n$, there are incident non-spine edges $(k,i)$ and $(i,j)$ for some $k\geq d-1$ and $j\geq 1$. 
    Then our netflow is 
    \begin{align*}
    	(h\ \text{outflow at}\ i) - (h\ \text{inflow at}\ i) 
        =& [h(i,j) + h(e_i)] - [h(k,i) + h(e_{e-1})] \\
        =& [0 + i-1] - [0 + i-2] \\
        =& 1 \\
        =& w_i . 
    \end{align*}
    
    At vertex $n+1$, there are two incoming non-spine edges $(n,n+1)$ and $(k,n+1)$ for some $k\geq d-1$. 
    Then our netflow is 
    \begin{align*}
    	(h\ \text{outflow at}\ n+1) - (h\ \text{inflow at}\ n+1) 
        =& 0- [h(k,n+1) + h(n,n+1) + h(e_{n})] \\
        =& -(0+0+n-1) \\
        =& -(n-1) \\
        =& w_{n+1} . 
    \end{align*}

    We know the leaf associated to $h$ cannot be in the image of the injection since the flow along the spine between vertices $b+1$ and $d$ is $0$, and the injection puts a nonzero amount of flow on the spine through the rule $g(e_i)+g(a,d)-g(b,b+1)$ for $e_i$ where $b+1\leq i\leq d-1$. 

    We argue the forward direction of our if and only if by contrapositive.
    Suppose $f_{G',L_b}(a,b+1)\leq f_{G',L_b}(b,d)$.
    We build an injection from the leaves of $f_{G',L_b}$ into the leaves of $\phi(f_{G',L_b})$. 
    This map and argument is nearly identical to the backward direction. 
    This time, our injection goes from the leaves of $f_{G',L_b}$ into the leaves of $\phi(f_{G',L_b})$. 
    Given a leaf $f$ of the level $b+1$ node $f_{G',L_b}$ in $T_G(\w_n)$, we map it to the leaf $g$ of $T_{G}(\w_n)$, where $g$ is defined as 
    \[
        g(i,j) = \left\{ \begin{array}{ll}
            f(a,b+1), & (i,j)=(a,d) \\
            f(b,d), & (i,j)=(b,b+1) \\
            f(e_i)+f(b,d)-f(a,b+1), & (i,j)=e_i, \ b+1\leq i\leq d-1 \\
            f(i,j), & \text{else}
        \end{array}
        \right. .
    \]
    Once again, we verify the netflow equality condition is satisfied, check that $g$ restricts to $\phi(f_{G',L_b})$, and then verify that the map $f\mapsto g$ is injective. 
    Just like in the backwards direction, the only vertices on which the netflows will change are $b+1,\ldots, d$. 
    At vertex $b+1$, we know the non-spine outgoing edge is consecutive, as $(a,d)$ is already the unique overpass edge to $b+2$ (as in Lemma~\ref{lemma:UniqueOverpass}). 
    So we have 
    \begin{align*}
    	(g\ \text{outflow at}\ b+1) - (g\ \text{inflow at}\ b+1) 
	=& [g(b+1,b+2) + g(e_{b+1})] - [g(b,b+1) + g(e_b)] \\
	=& [ f(b+1,b+2) + f(e_{b+1})+f(b,d)+f(a,b+1)] \\
	&- [ f(b,d) + f(e_b) ] \\
	=& [ f(b+1,b+2) + f(e_{b+1})] - [ f(a,b+1) + f(e_b) ] \\
	=& (f\ \text{outflow at}\ b+1) - (f\ \text{inflow at}\ b+1) \\
	=& w_{b+1} .
    \end{align*}
    If there are any vertices $i$ such that $b+2\leq i\leq d-1$, note that all four incident edges are consecutive edges, else we would violate the unique overpass condition of Lemma~\ref{lemma:UniqueOverpass}. 
    So we have 
    \begin{align*}
    	(g\ \text{outflow at}\ i) - (g\ \text{inflow at}\ i) 
	=& [g(i,i+1) + g(e_{i})] - [g(i-1,i) + g(e_{i-1})] \\
	=& [ f(i,i+1) + f(e_i) + f(b,d) - f(a,b+1) ] \\
	&- [ f(i-1,i) + f(e_{i-1}) + f(b,d) - f(a,b+1) ] \\
	=& [ f(i,i+1) + f(e_i) ] - [ f(i-1,i) + f(e_{i-1}) ] \\
	=& (f\ \text{outflow at}\ i) - (f\ \text{inflow at}\ i) \\
	=& w_{i} .
    \end{align*}
    At vertex $d$, we have a non-spine outgoing edge $(d,j)$ for some $j$. 
    Then our netflow is 
    \begin{align*}
    	(g\ \text{outflow at}\ d) - (g\ \text{inflow at}\ d) 
	=& [g(d,j) + g(e_{d})] - [g(a,d) + g(e_{d-1})] \\
	=& [ f(d,j) + f(e_d) ] \\
	&- [ f(a,b+1) + f(e_{d-1}) + f(b,d) - f(a,b+1) ] \\
	=& [ f(d,j) + f(e_d) ] - [ f(e_{d-1}) + f(b,d) ] \\
	=& (f\ \text{outflow at}\ d) - (f\ \text{inflow at}\ d) \\
	=& w_{d} .
    \end{align*}
    Thus $g$ is indeed a valid flow on $G$. 
    Next we verify that $g$ agrees with $\phi(f_{G',L_b})$ on $L_b$. 
    Given an edge in $G$, $(i,j)\in L_b$, since $f$ is a leaf of $f_{G',L_b}$ 
    \begin{align*}
    	g(i,j) &= \left\{ \begin{array}{ll}
            f(a,b+1), & (i,j)=(a,d) \\
            f(b,d), & (i,j)=(b,b+1) \\
            f(i,j), & \text{else}
        \end{array}
        \right. \\
        &=  \left\{ \begin{array}{ll}
            f_{G',L_b}(a,b+1), & (i,j)=(a,d) \\
            f_{G',L_b}(b,d), & (i,j)=(b,b+1) \\
            f_{G',L_b}(i,j), & \text{else}
        \end{array}
        \right. \\ 
        &= \phi(f_{G',L_b})(i,j) .
    \end{align*}
    
    Finally, we prove that this mapping is injective.
    Let $f_1$ and $f_2$ be (flows on $G'$ that give) leaves in $T_G'(\w_n)$. 
    Suppose that both $f_1$ and $f_2$ map to a leaf $g$ in $T_{G}(\w_n)$. Then 
    \begin{align*}
        &\left\{ \begin{array}{ll}
            f_1(a,b+1), & (i,j)=(a,d) \\
            f_1(b,d), & (i,j)=(b,b+1) \\
            f_1(e_i)+f_1(b,d)-f_1(b,b+1), & (i,j)=e_i, \ b+1\leq i\leq d-1 \\
            f_1(i,j), & \text{else}
        \end{array}
        \right. \\
        =&
        \left\{ \begin{array}{ll}
            f_2(a,b+1), & (i,j)=(a,d) \\
            f_2(b,d), & (i,j)=(b,b+1) \\
            f_2(e_i)+f_2(b,d)-f_2(a,b+1), & (i,j)=e_i, \ b+1\leq i\leq d-1 \\
            f_2(i,j), & \text{else}
        \end{array}
        \right.  .
    \end{align*}
    We can immediately conclude that $f_1(a,b+1)=f_2(a,b+1)$ and $f_1(b,d)=f_2(b,d)$, and that $f_1(i,j)=f_2(i,j)$ for every edge that isn't a spine edge between $b+1\leq i\leq d-1$.
    Then it follows that $f_1(e_i)=f_2(e_i)$ for every spine edge $e_i$ such that $b+1\leq i\leq d-1$. 
    Hence $f_1=f_2$, so our map is indeed injective. 
\end{proof}

Having established a relationship between the number of leaves below level $b+1$ notes in $G$ and $G'$, our next goal is to relate the total number of leaves in each of the flow decomposition trees.

\begin{definition} \label{def:BadNodes}
  Assume $G\in\mathscr{F}_{n,3}$ has nested edges that are interchanged to obtain $G'$. 
  We call a level $b+1$ node $f_{G',L_b}$ in $T_{G'}(\w_n)$ that satisfies the condition of Lemma~\ref{lemma:BadNodes} a \emph{bad} node.
  We denote the set of bad level $b+1$ nodes of $T_{G'}(\w_n)$ by $V^>_{b+1}(G')$.
  Otherwise, we say $f_{G',L_b}$ is \emph{good}.
  We denote the set of good level $b+1$ nodes of $T_{G'}(\w_n)$ by $V^\leq_{b+1}(G')$.
\end{definition}

We can use the $\phi$ bijection to apply the good and bad classifications to the level $b+1$ nodes of $T_G(\w_n)$.
Using this, we get that a node $f_{G,L_b}$ in $T_G(\w_n)$ is bad if and only if $f_{G,L_b}(a,d)>f_{G,L_b}(b,b+1)$.

\begin{example} \label{ex:BadNode}
    If we look at the five level $b+1=4$ nodes in Figure~\ref{fig:fd-tree11}, we see one bad node (the second node from the left). 
    The said node has $f_{G_{11},L_3}(2,4)=1>f_{G_{11},L_3}(3,5)=0$. 
    When looking at the corresponding level $b+1=4$ nodes in Figure~\ref{fig:fd-tree10}, we see that every node has more leaves than its corresponding level $b+1$ node in $T_{G_{11}}$, except for the second node from the left corresponding to the aforementioned bad level-$4$ node in Figure~\ref{fig:fd-tree11}.  
\end{example}

At this point, we have partitioned the level $b+1$ nodes in $T_{G'}(\w_n)$ into nodes that have too few leaves when compared to $T_G(\w_n)$ (bad) and nodes that have the same or more leaves when compared to $T_G(\w_n)$ (good). 
Next, we control the number of bad nodes relative to the good nodes.

\begin{lemma} \label{lemma:CoordSwap}
  Assume $G\in\mathscr{F}_{n,3}$ has nested edges that are interchanged to obtain $G'$. 
    For each bad node $f_{G',L_b}$ in $T_{G'}(\w_n)$, there exists a unique good node $g_{G',L_b}$ such that for every edge $(i,j)\in L_b$,
    \[ 
    g_{G',L_b}(i,j) = \left\{
    \begin{array}{ll}
        f_{G',L_b}(b,d), & (i,j)=(a,b+1) \\
        f_{G',L_b}(a,b+1), & (i,j)=(b,d) \\
        f_{G',L_b}(e_i)+f_{G',L_b}(a,b+1)-f_{G',L_b}(b,d), & (i,j)= e_i,\  a\leq i\leq b-1 \\
        f_{G',L_b}(i,j), & \text{else}
    \end{array}
    \right. .
    \]
    This yields an injection $\psi$ from the bad nodes of $T_{G'}(\w_n)$ into the good nodes of $T_{G'}(\w_n)$, where $\psi(f_{G',L_b}):=g_{G',L_b}$
\end{lemma}

\begin{proof}
    Let $f_{G',L_b}$ be a bad node in $T_{G'}(\w_n)$. 
    Define $\psi(f_{G',L_b}):=g_{G',L_b}$, where   
    \[ 
    g_{G',L_b}(i,j) = \left\{
    \begin{array}{ll}
        f_{G',L_b}(b,d), & (i,j)=(a,b+1) \\
        f_{G',L_b}(a,b+1), & (i,j)=(b,d) \\
        f_{G',L_b}(e_i)+f_{G',L_b}(a,b+1)-f_{G',L_b}(b,d), & (i,j)= e_i,\  a\leq i\leq b-1 \\
        f_{G',L_b}(i,j), & \text{else}
    \end{array}
    \right. .
    \]
    To show $g_{G',L_b}$ is a valid level $b+1$ node, we need only check the netflow at vertices in $[b]$ where the in- or outflow has changed: $a,\ldots,b$.
    At vertex $a$, there is an incoming edge $(i,a)$ for some $i$. 
    Then our netflow at $a$ is
    \begin{align*}
    	(g_{G',L_b}\ \text{outflow at}\ a) - (g_{G',L_b}\ \text{inflow at}\ a) 
	=& [g_{G',L_b}(a,b+1) + g_{G',L_b}(e_{a})] - [g_{G',L_b}(i,a) + g_{G',L_b}(e_{a-1})] \\
	=& [ f_{G',L_b}(b,d) + f_{G',L_b}(e_a)+f_{G',L_b}(a,b+1)-f_{G',L_b}(b,d) ] \\
	&- [  f_{G',L_b}(i,a) + f_{G',L_b}(e_{a-1}) ] \\
	=& [ f_{G',L_b}(a,b+1) + + f_{G',L_b}(e_a) ] - [ f_{G',L_b}(i,a) + f_{G',L_b}(e_{a-1}) ] \\
	=& (f_{G',L_b}\ \text{outflow at}\ a) - (f_{G',L_b}\ \text{inflow at}\ a) \\
	=& w_{a} .
    \end{align*}
    
    If there are any vertices $i$ such that $a+1\leq i\leq b-1$, all incident edges must be consecutive, else we would violate Lemma~\ref{lemma:UniqueOverpass}. 
    Then our netflow at $i$ is
    \begin{align*}
    	(g_{G',L_b}\ \text{outflow at}\ i) - (g_{G',L_b}\ \text{inflow at}\ i) 
	=& [g_{G',L_b}(i,i+1) + g_{G',L_b}(e_{i})] - [g_{G',L_b}(i-1,i) + g_{G',L_b}(e_{i-1})] \\
	=& [ f_{G',L_b}(i,i+1) + f_{G',L_b}(e_i)+f_{G',L_b}(a,b+1)-f_{G',L_b}(b,d) ] \\
	&- [  f_{G',L_b}(i-1,i) + f_{G',L_b}(e_{i-1})+f_{G',L_b}(a,b+1)-f_{G',L_b}(b,d) ] \\
	=& [ f_{G',L_b}(i,i+1) + f_{G',L_b}(e_i) ] - [ f_{G',L_b}(i-1,i) + f_{G',L_b}(e_{i-1})] \\
	=& (f_{G',L_b}\ \text{outflow at}\ i) - (f_{G',L_b}\ \text{inflow at}\ i) \\
	=& w_{i} .
    \end{align*}
    
    At vertex $b$, we have some incoming edge $(i,b)$.
    Our netflow is 
    \begin{align*}
    	(g_{G',L_b}\ \text{outflow at}\ b) - (g_{G',L_b}\ \text{inflow at}\ b) 
	=& [g_{G',L_b}(b,d) + g_{G',L_b}(e_{b})] - [g_{G',L_b}(i,b) + g_{G',L_b}(e_{b-1})] \\
	=& [ f_{G',L_b}(a,b+1) + f_{G',L_b}(e_b) ] \\
	&- [  f_{G',L_b}(i,b) + f_{G',L_b}(e_{b-1}) + f_{G',L_b}(a,b+1)-f_{G',L_b}(b,d) ] \\
	=& [ f_{G',L_b}(b,d) + f_{G',L_b}(e_b) ] - [ f_{G',L_b}(i,b) + f_{G',L_b}(e_{b-1}) ] \\
	=& (f_{G',L_b}\ \text{outflow at}\ b) - (f_{G',L_b}\ \text{inflow at}\ b) \\
	=& w_{b} .
    \end{align*}
    Thus $g_{G',L_b}$ is valid on $[b]$. 
    
    To verify that $g_{G',L_b}$ is a good node as in Lemma~\ref{lemma:BadNodes}, observe that 
    \begin{align*}
    	g_{G',L_b}(a,b+1) &= f_{G',L_b}(b,d) \leq f_{G',L_b}(a,b+1) = g_{G',L_b}(b,d). 
    \end{align*}
    
    Finally, we check that $\psi$ is injective. 
    Let $f_1$ and $f_2$ be two bad level \(b+1\) nodes in \(T_{G'}(\w_n)\). 
    Suppose $\psi(f_1) = \psi(f_2)$. 
    Then we can immediately conclude that $f_1(a,b+1)=f_2(a,b+1)$, $f_1(b,d)=f_2(b,d)$, and $f_1(i,j)=f_2(i,j)$ for any edge in $L_b$ that isn't $(a,b+1)$, $(b,d)$ or a spine edge between $a$ and $b-1$. 
    This in turn implies that $f_1(i,j)=f_2(i,j)$ on the spine edges between $a$ and $b-1$.
    Thus $f_1(i,j) = f_2(i,j)$ for all $(i,j)\in L_b$.
\end{proof}

\begin{example} \label{ex:CoordSwap}
    Considering the single bad node $f_{G_{11},L_3}$ in Figure~\ref{fig:fd-tree11}  (the second-from-left level 4 node), we see that there is a unique level $b+1=4$ node $g_{G_{11},L_3}$ in $T_{G_{11}}(\w_4)$ with 
    \begin{align*}
        g_{G_{11},L_3}(2,4) &= 0 = f_{G_{11},L_3}(3,5), \\
        g_{G_{11},L_3}(3,5) &= 1 = f_{G_{11},L_3}(2,4), \\
        g_{G_{11},L_3}(e_2) &= 1 = f_{G_{11},L_3}(e_2)+f_{G_{11},L_3}(3,5)-f_{G_{11},L_3}(2,4),
    \end{align*}
    and the same flows on all other edges. 
\end{example}


\begin{definition}
    We use the level $b+1$ node bijection $\psi$ for $T_{G'}(\w_n)$ to define a version of $\psi$ for the bad nodes of $T_G(\w_n)$. 
    More specifically, define 
    $\psi(f_{G,L_b}):=\phi(\psi(f_{G',L_b}))$ 
    where $\phi(f_{G',L_b})=f_{G,L_b}$. 
\end{definition}

\begin{lemma} \label{lemma:commute}
  Assume $G\in\mathscr{F}_{n,3}$ has nested edges that are interchanged to obtain $G'$. 
    The construction of $\psi$ for level $b+1$ nodes in $T_G(\w_n)$ is an injection. 
\end{lemma}

\begin{proof}
    This follows from the fact that $\phi$ is a bijection and $\psi$ as defined for $T_{G'}(\w_n)$ is an injection.
\end{proof}

Thus, every bad level $b+1$ node in $T_{G'}(\w_n)$ corresponds with a unique good node in $T_{G'}(\w_n)$, and each of these nodes has a bijective corresponding node in $T_G(\w_n)$.
We next prove our most crucial intermediate result, pertaining to the number of leaves of these two pairs of nodes.

\begin{lemma} \label{lemma:KeyLemma}
  Assume $G\in\mathscr{F}_{n,3}$ has nested edges that are interchanged to obtain $G'$. 
    Given a bad level $b+1$ node $f_{G',L_b}$ in $T_{G'}(\w_n)$,
    \[
    l(f_{G',L_b})=l(\phi(\psi(f_{G',L_b}))) \quad\text{and}\quad l(\psi(f_{G',L_b}))=l(\phi(f_{G',L_b})) .
    \]
    Further, the four nodes comprise two pairs of nodes, one pair in $T_{G'}(\w_n)$ and another in $T_{G}(\w_n)$, and each pair has the same total number of leaves. 
    That is, 
    \[
    l(f_{G',L_b}) + l(\psi(f_{G',L_b})) = l(\phi(f_{G',L_b})) + l(\phi(\psi(f_{G',L_b}))) .
    \]
\end{lemma}

\begin{proof}
    Let $f_{G',L_b}$ be a bad node in $T_{G'}(\w_n)$. 
    Our proof is illuminated by Figures~\ref{fig:KeyLemma1} and~\ref{fig:KeyLemma2}.
    
    Note that in Figures~\ref{fig:KeyLemma1} and~\ref{fig:KeyLemma2} the solid lines represent edges whose flows have been decided by the partial flow and dotted lines represent edges whose flows will be decided later in the flow decomposition tree. 
    (This is different than what dotted edges represented in Figures~\ref{fig:fd-tree11} and~\ref{fig:fd-tree10}.)

    In Figure~\ref{fig:KeyLemma1}, we see that the nodes $f_{G',L_b}$ and $\phi(\psi(f_{G',L_b}))$ will have the same number of leaves because they have the same flows in to vertices $b+1$ (by Proposition~\ref{prop:OverpassDeterminesInflow}) and $d$, and identical edges remaining on which to direct those flows.
    The same argument applies in Figure~\ref{fig:KeyLemma2} to $\phi(f_{G',L_b})$ and $\psi(f_{G',L_b})$. 

    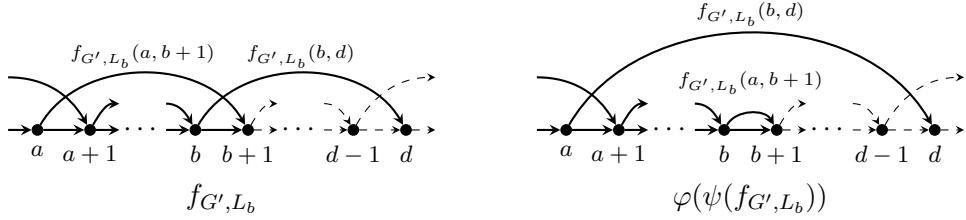
\begin{figure}
    \begin{center}
    \begin{tikzpicture}
    
    \begin{scope}[xshift=0, yshift=0, scale=0.7]
    
            \node[] at (4.5,-1.3) {$f_{G',L_b}$};
    	\vertex[fill,label=below:\footnotesize{$a$}](a1) at (1,0) {};
    	\vertex[fill,label=below:\footnotesize{$a+1$}](a2) at (2,0) {};
    	\vertex[fill,label=below:\footnotesize{$b$}](a3) at (4,0) {};
    	\vertex[fill,label=below:\footnotesize{$b+1$}](a4) at (5,0) {};
    	\vertex[fill,label=below:\footnotesize{$d-1$}](a5) at (7,0) {};
    	\vertex[fill,label=below:\footnotesize{$d$}](a6) at (8,0) {};
            \node(b1) at (0.25,1) {};
            \node(b2) at (0.25,0) {};
            \node(b3) at (2.75,0.5) {};
            \node(b4) at (2.75,0) {};
            \node(b5) at (3.25,0.5) {};
            \node(b6) at (3.25,0) {};
            \node(b7) at (5.75,0.5) {};
            \node(b8) at (5.75,0) {};
            \node(b9) at (6.25,0.5) {};
            \node(b10) at (6.25,0) {};
            \node(b11) at (8.75,1) {};
            \node(b12) at (8.75,0) {};
    	\draw[-stealth, thick] (b1) to[out=0,in=120] (a2);
            \draw[-stealth, thick] (b2) to[out=0,in=180] (a1);
            \draw[-stealth, thick] (a1) to[out=0,in=180] (a2);
            \draw[-stealth, thick] (a1) to[out=60,in=120] (a4);
            \draw[-stealth, thick] (a2) to[out=60,in=180] (b3);
            \draw[-stealth, thick] (a2) to[out=0,in=180] (b4);
            \draw[-stealth, thick] (b5) to[out=0,in=120] (a3);
            \draw[-stealth, thick] (b6) to[out=0,in=180] (a3);
            \draw[-stealth, thick] (a3) to[out=0,in=180] (a4);
    	\draw[-stealth, thick] (a3) to[out=60,in=120] (a6);
            \draw[-stealth, dashed] (a4) to[out=60,in=180] (b7);
            \draw[-stealth, dashed] (a4) to[out=0,in=180] (b8);
            \draw[-stealth, dashed] (b9) to[out=0,in=120] (a5);
            \draw[-stealth, dashed] (b10) to[out=0,in=180] (a5);
            \draw[-stealth, dashed] (a5) to[out=0,in=180] (a6);
            \draw[-stealth, dashed] (a5) to[out=60,in=180] (b11);
            \draw[-stealth, dashed] (a6) to[out=0,in=180] (b12);
            \node[] at (3,0) {$\cdots$};
            \node[] at (6,0) {$\cdots$};
            \node[] at (3,1.4) {\tiny $f_{G',L_b}(a,b+1)$};
            \node[] at (6,1.4) {\tiny $f_{G',L_b}(b,d)$};

    \end{scope}
    
    \begin{scope}[xshift=200, yshift=0, scale=0.7]
    
            \node[] at (4.5,-1.3) {$\phi(\psi(f_{G',L_b}))$};
    	\vertex[fill,label=below:\footnotesize{$a$}](a1) at (1,0) {};
    	\vertex[fill,label=below:\footnotesize{$a+1$}](a2) at (2,0) {};
    	\vertex[fill,label=below:\footnotesize{$b$}](a3) at (4,0) {};
    	\vertex[fill,label=below:\footnotesize{$b+1$}](a4) at (5,0) {};
    	\vertex[fill,label=below:\footnotesize{$d-1$}](a5) at (7,0) {};
    	\vertex[fill,label=below:\footnotesize{$d$}](a6) at (8,0) {};
            \node(b1) at (0.25,1) {};
            \node(b2) at (0.25,0) {};
            \node(b3) at (2.75,0.5) {};
            \node(b4) at (2.75,0) {};
            \node(b5) at (3.25,0.5) {};
            \node(b6) at (3.25,0) {};
            \node(b7) at (5.75,0.5) {};
            \node(b8) at (5.75,0) {};
            \node(b9) at (6.25,0.5) {};
            \node(b10) at (6.25,0) {};
            \node(b11) at (8.75,1) {};
            \node(b12) at (8.75,0) {};
    	\draw[-stealth, thick] (b1) to[out=0,in=120] (a2);
            \draw[-stealth, thick] (b2) to[out=0,in=180] (a1);
            \draw[-stealth, thick] (a1) to[out=0,in=180] (a2);
            \draw[-stealth, thick] (a1) to[out=60,in=120] (a6);
            \draw[-stealth, thick] (a2) to[out=60,in=180] (b3);
            \draw[-stealth, thick] (a2) to[out=0,in=180] (b4);
            \draw[-stealth, thick] (b5) to[out=0,in=120] (a3);
            \draw[-stealth, thick] (b6) to[out=0,in=180] (a3);
            \draw[-stealth, thick] (a3) to[out=0,in=180] (a4);
    	\draw[-stealth, thick] (a3) to[out=60,in=120] (a4);
            \draw[-stealth, dashed] (a4) to[out=60,in=180] (b7);
            \draw[-stealth, dashed] (a4) to[out=0,in=180] (b8);
            \draw[-stealth, dashed] (b9) to[out=0,in=120] (a5);
            \draw[-stealth, dashed] (b10) to[out=0,in=180] (a5);
            \draw[-stealth, dashed] (a5) to[out=0,in=180] (a6);
            \draw[-stealth, dashed] (a5) to[out=60,in=180] (b11);
            \draw[-stealth, dashed] (a6) to[out=0,in=180] (b12);
            \node[] at (3,0) {$\cdots$};
            \node[] at (6,0) {$\cdots$};
            \node[] at (4.5,2.2) {\tiny $f_{G',L_b}(b,d)$};
            \node[] at (4.5,0.9) {\tiny $f_{G',L_b}(a,b+1)$};

    \end{scope}
    
    \end{tikzpicture}
    \end{center}
    \caption{Comparing nodes $f_{G',L_b}$ and $\phi(\psi(f_{G',L_b}))$.}
    \label{fig:KeyLemma1}
    \end{figure}

    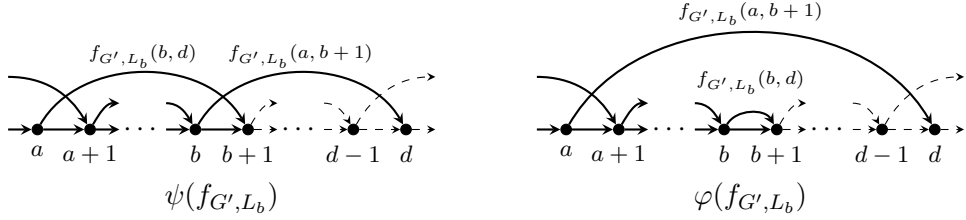
\begin{figure}
    \begin{center}
    \begin{tikzpicture}

    	\begin{scope}[xshift=0, yshift=0, scale=0.7]
    
            \node[] at (4.5,-1.3) {$\psi(f_{G',L_b})$};
    	\vertex[fill,label=below:\footnotesize{$a$}](a1) at (1,0) {};
    	\vertex[fill,label=below:\footnotesize{$a+1$}](a2) at (2,0) {};
    	\vertex[fill,label=below:\footnotesize{$b$}](a3) at (4,0) {};
    	\vertex[fill,label=below:\footnotesize{$b+1$}](a4) at (5,0) {};
    	\vertex[fill,label=below:\footnotesize{$d-1$}](a5) at (7,0) {};
    	\vertex[fill,label=below:\footnotesize{$d$}](a6) at (8,0) {};
            \node(b1) at (0.25,1) {};
            \node(b2) at (0.25,0) {};
            \node(b3) at (2.75,0.5) {};
            \node(b4) at (2.75,0) {};
            \node(b5) at (3.25,0.5) {};
            \node(b6) at (3.25,0) {};
            \node(b7) at (5.75,0.5) {};
            \node(b8) at (5.75,0) {};
            \node(b9) at (6.25,0.5) {};
            \node(b10) at (6.25,0) {};
            \node(b11) at (8.75,1) {};
            \node(b12) at (8.75,0) {};
    	\draw[-stealth, thick] (b1) to[out=0,in=120] (a2);
            \draw[-stealth, thick] (b2) to[out=0,in=180] (a1);
            \draw[-stealth, thick] (a1) to[out=0,in=180] (a2);
            \draw[-stealth, thick] (a1) to[out=60,in=120] (a4);
            \draw[-stealth, thick] (a2) to[out=60,in=180] (b3);
            \draw[-stealth, thick] (a2) to[out=0,in=180] (b4);
            \draw[-stealth, thick] (b5) to[out=0,in=120] (a3);
            \draw[-stealth, thick] (b6) to[out=0,in=180] (a3);
            \draw[-stealth, thick] (a3) to[out=0,in=180] (a4);
    	\draw[-stealth, thick] (a3) to[out=60,in=120] (a6);
            \draw[-stealth, dashed] (a4) to[out=60,in=180] (b7);
            \draw[-stealth, dashed] (a4) to[out=0,in=180] (b8);
            \draw[-stealth, dashed] (b9) to[out=0,in=120] (a5);
            \draw[-stealth, dashed] (b10) to[out=0,in=180] (a5);
            \draw[-stealth, dashed] (a5) to[out=0,in=180] (a6);
            \draw[-stealth, dashed] (a5) to[out=60,in=180] (b11);
            \draw[-stealth, dashed] (a6) to[out=0,in=180] (b12);
            \node[] at (3,0) {$\cdots$};
            \node[] at (6,0) {$\cdots$};
            \node[] at (3,1.4) {\tiny $f_{G',L_b}(b,d)$};
            \node[] at (6,1.4) {\tiny $f_{G',L_b}(a,b+1)$};

    \end{scope}
    
    \begin{scope}[xshift=200, yshift=0, scale=0.7]
    
            \node[] at (4.5,-1.3) {$\phi(f_{G',L_b})$}; 
    	\vertex[fill,label=below:\footnotesize{$a$}](a1) at (1,0) {};
    	\vertex[fill,label=below:\footnotesize{$a+1$}](a2) at (2,0) {};
    	\vertex[fill,label=below:\footnotesize{$b$}](a3) at (4,0) {};
    	\vertex[fill,label=below:\footnotesize{$b+1$}](a4) at (5,0) {};
    	\vertex[fill,label=below:\footnotesize{$d-1$}](a5) at (7,0) {};
    	\vertex[fill,label=below:\footnotesize{$d$}](a6) at (8,0) {};
            \node(b1) at (0.25,1) {};
            \node(b2) at (0.25,0) {};
            \node(b3) at (2.75,0.5) {};
            \node(b4) at (2.75,0) {};
            \node(b5) at (3.25,0.5) {};
            \node(b6) at (3.25,0) {};
            \node(b7) at (5.75,0.5) {};
            \node(b8) at (5.75,0) {};
            \node(b9) at (6.25,0.5) {};
            \node(b10) at (6.25,0) {};
            \node(b11) at (8.75,1) {};
            \node(b12) at (8.75,0) {};
    	\draw[-stealth, thick] (b1) to[out=0,in=120] (a2);
            \draw[-stealth, thick] (b2) to[out=0,in=180] (a1);
            \draw[-stealth, thick] (a1) to[out=0,in=180] (a2);
            \draw[-stealth, thick] (a1) to[out=60,in=120] (a6);
            \draw[-stealth, thick] (a2) to[out=60,in=180] (b3);
            \draw[-stealth, thick] (a2) to[out=0,in=180] (b4);
            \draw[-stealth, thick] (b5) to[out=0,in=120] (a3);
            \draw[-stealth, thick] (b6) to[out=0,in=180] (a3);
            \draw[-stealth, thick] (a3) to[out=0,in=180] (a4);
    	\draw[-stealth, thick] (a3) to[out=60,in=120] (a4);
            \draw[-stealth, dashed] (a4) to[out=60,in=180] (b7);
            \draw[-stealth, dashed] (a4) to[out=0,in=180] (b8);
            \draw[-stealth, dashed] (b9) to[out=0,in=120] (a5);
            \draw[-stealth, dashed] (b10) to[out=0,in=180] (a5);
            \draw[-stealth, dashed] (a5) to[out=0,in=180] (a6);
            \draw[-stealth, dashed] (a5) to[out=60,in=180] (b11);
            \draw[-stealth, dashed] (a6) to[out=0,in=180] (b12);
            \node[] at (3,0) {$\cdots$};
            \node[] at (6,0) {$\cdots$};
            \node[] at (4.5,2.2) {\tiny $f_{G',L_b}(a,b+1)$};
            \node[] at (4.5,0.9) {\tiny $f_{G',L_b}(b,d)$};

    \end{scope}

    \end{tikzpicture}
    \end{center}
    \caption{Comparing nodes $\psi(f_{G',L_b})$ and $\phi(f_{G',L_b})$.}
    \label{fig:KeyLemma2}
    \end{figure}
    \end{proof}

\begin{example} \label{ex:KeyLemma}
    Looking again at the bad node $f_{G_{11},L_3}$ in Figure~\ref{fig:fd-tree11}, we found its unique good node pair $g_{G_{11},L_3}$ in Example~\ref{ex:CoordSwap}. 
    If we then find the level $b+1$ bijective correspondents for these two nodes in $T_{G_{10}}(\w_4)$, $f_{G_{10},L_3}$ and $g_{G_{10},L_3}$ respectively, we see that 
    \[
    l(f_{G_{11},L_3})=4=l(g_{G_{10},L_3}) \quad\text{and}\quad l(g_{G_{11},L_3})=3=l(f_{G_{10},L_3}) .
    \]
\end{example}

We are now ready to state our main result of this section.

\begin{theorem} \label{thm:MainThm}
    Let $G\in\mathscr{F}_{n,3}$ and suppose $G'$ is obtained by interchanging a pair of nested edges in $G$. 
    Then 
    \[
    \vol \F_1(G') \leq \vol \F_1(G). 
    \]
\end{theorem}

\begin{proof}
    We have established that for any DAG $G\in\mathscr{F}_{n,3}$,
    \[ 
    \vol \F_1(G) = K_G(\w_n) = \# \text{ of leaves in } T_G(\w_n).
    \]
    Given $G$ and $G'$ as in the theorem statement, we compare the number of leaves in their respective flow decomposition trees by summing over the leaves of the level $b+1$ nodes. Starting in $T_{G'}(\w_n)$,
    \begin{align*}
        \# \text{ of leaves in } T_{G'}(\w_n) &= \sum_{f_{G',L_b}\in V_{b+1}(G')} l(f_{G',L_b}) \\
        &= \sum_{f_{G',L_b}\in V^>_{b+1}(G')} [ l(f_{G',L_b}) + l(\psi(f_{G',L_b})) ] + \sum_{f_{G',L_b}\in V^\leq_{b+1}(G')\setminus \psi(V^>_{b+1}(G'))} l(f_{G',L_b}) \\
        &= \sum_{f_{G',L_b}\in V^>_{b+1}(G')} [ l(\phi(f_{G',L_b})) + l(\phi(\psi(f_{G',L_b}))) ] + \sum_{f_{G',L_b}\in V^\leq_{b+1}(G')\setminus \psi(V^>_{b+1}(G'))} l(f_{G',L_b}) \\
        &\leq \sum_{f_{G',L_b}\in V^>_{b+1}(G')} [ l(\phi(f_{G',L_b})) + l(\phi(\psi(f_{G',L_b}))) ] + \sum_{f_{G',L_b}\in V^\leq_{b+1}(G')\setminus \psi(V^>_{b+1}(G'))} l(\phi(f_{G',L_b})) \\
        &= \sum_{f_{G,L_b}\in V_{b+1}(G)} l(f_{G,L_b}) \\
        &= \# \text{ of leaves in } T_{G}(\w_n).
    \end{align*}
\end{proof}

\begin{corollary} \label{cor:OrderRevering}
    The function $\vol:\mathscr{F}_{n,3}\ra \Z_{\geq0}$, which maps a DAG to the volume of its flow polytope, is order-reversing with respect the the Boolean (equivalently, interchange) partial order on $\mathscr{F}_{n,3}$.
\end{corollary}

\section{Linear Extensions} \label{sec:LinearExtensions}

In this section we establish a connection between flow polytope volumes for DAGs in $\mathscr{F}_{n,3}$ and volumes of a related family of order polytopes; order polytopes are lattice polytopes defined by finite posets~\cite{stanley-posets}.
M\'{e}sz\'{a}ros, Morales, and Striker~\cite{meszaros-morales-striker} established a relationship between flow polytopes and order polytopes, and we use their equivalences to translate our results from Section~\ref{sec:volume} about flow polytopes to results about order polytopes. 
In particular, in Corollary~\ref{cor:MainThmLinearExt} we restate Theorem~\ref{thm:MainThm} as a statement about posets and their linear extensions. 
We briefly review the background and relevant results of~\cite{meszaros-morales-striker}.

\begin{definition}\label{def:orderpolytope}
Given a finite poset $P$, define the \emph{order polytope} to be 
\[
\mathcal{O}(P):= \{x\in [0,1]^{|P|}:x_i\leq x_j \text{ if }i\leq_P j\} \, .
\]
\end{definition}

\begin{remark} Note that the definitions of strongly planar given below are more restrictive than the definitions given in the original M\'{e}sz\'{a}ros, Morales, and Striker paper~\cite{meszaros-morales-striker}, as the original definitions are not sufficient to allow Theorem~\ref{thm:flowandorderequivalence} to hold in general.
  However, Definitions~\ref{def:stronglyplanardag} and~\ref{def:stronglyplanarposet} are sufficient both to establish Theorem~\ref{thm:flowandorderequivalence} and to encompass all DAGs of interest.
  \end{remark}

\begin{definition} \label{def:stronglyplanardag}
A DAG $G$ is \emph{strongly planar} if $G$ is a planar graph with a planar realization such that the following two conditions hold. First, for every directed edge $(i,j)$ of $G$, the $x$-coordinate of $i$ is strictly less than the $x$-coordinate of $j$.
Second, each edge $(i,j)$ is embedded in the plane as the graph of a piecewise differentiable function of $x$.
\end{definition}

\begin{definition} \label{def:stronglyplanarposet}
A poset $P$ is \emph{strongly planar} if the Hasse diagram of $P$ is a planar graph with a planar realization such that the following two conditions hold. First, for every directed edge $(i,j)$ in the Hasse diagram, the $y$-coordinate of $i$ is strictly less than the $y$-coordinate of $j$.
Second, each edge $(i,j)$ is embedded in the plane as the graph of a piecewise differentiable function of $y$.
\end{definition}

The following definition was originally given by M\'esz\'aros, Morales, and Striker~\cite[Section 3.3]{meszaros-morales-striker}.

\begin{definition}
    \label{def:dualgraph}
    Given a strongly planar DAG $G$, we define the \emph{truncated dual graph} $G^*$ to be a strongly planar realization of the dual graph of $G$ with the vertex corresponding to the infinite region deleted.
    The graph $G^*$ is thus the Hasse diagram of a strongly planar poset that we denote by $P_G$.
    Similarly, given a strongly planar poset $P$, let $H$ be the Hasse diagram of $P\cup \{\hat{0},\hat{1}\}$.
    We define $G_P$ to be the strongly planar DAG obtained as the strongly planar dual of $H$.
\end{definition}

\begin{example} \label{ex:n=4Duals}
	Refer to Figure~\ref{fig:n=4Duals} for the truncated duals of the DAGs in $\mathscr{F}_{4,3}$.
\end{example}

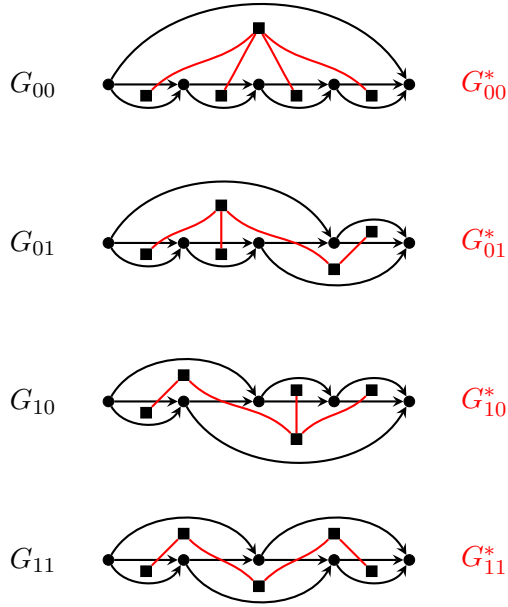
\begin{figure}
\begin{center}
\begin{tikzpicture}

\begin{scope}[scale=1, xshift=0, yshift=180]

	\vertex[fill](a1) at (1,0) {};
	\vertex[fill](a2) at (2,0) {};
	\vertex[fill](a3) at (3,0) {};
	\vertex[fill](a4) at (4,0) {};
	\vertex[fill](a5) at (5,0) {};
	\draw[-stealth, thick] (a1) to[out=0,in=180] (a2);
        \draw[-stealth, thick] (a1) to[out=-60,in=-120] (a2);
        \draw[-stealth, thick] (a1) to[out=60,in=120] (a5);
        \draw[-stealth, thick] (a2) to[out=0,in=180] (a3);
	\draw[-stealth, thick] (a2) to[out=-60,in=-120] (a3);
        \draw[-stealth, thick] (a3) to[out=0,in=180] (a4);
	\draw[-stealth, thick] (a3) to[out=-60,in=-120] (a4);
        \draw[-stealth, thick] (a4) to[out=0,in=180] (a5);
	\draw[-stealth, thick] (a4) to[out=-60,in=-120] (a5);
        \node[] at (0,0) {$G_{00}$};
        \vertex[rectangle, fill](b1) at (1.5,-0.15) {};
        \vertex[rectangle, fill](b2) at (2.5,-0.15) {};
        \vertex[rectangle, fill](b3) at (3.5,-0.15) {};
        \vertex[rectangle, fill](b4) at (4.5,-0.15) {};
        \vertex[rectangle, fill](b5) at (3,0.75) {};
        \draw[red, thick] (b1) to[out=45,in=-135] (b5);
        \draw[red, thick] (b2) to[out=60,in=-120] (b5);
        \draw[red, thick] (b3) to[out=120,in=-60] (b5);
        \draw[red, thick] (b4) to[out=135,in=-45] (b5);
        \node[] at (6,0) {\textcolor{red}{$G_{00}^*$}};
	
\end{scope}

\begin{scope}[scale=1, xshift=0, yshift=120]

	\vertex[fill](a1) at (1,0) {};
	\vertex[fill](a2) at (2,0) {};
	\vertex[fill](a3) at (3,0) {};
	\vertex[fill](a4) at (4,0) {};
	\vertex[fill](a5) at (5,0) {};
	\draw[-stealth, thick] (a1) to[out=0,in=180] (a2);
        \draw[-stealth, thick] (a1) to[out=-60,in=-120] (a2);
        \draw[-stealth, thick] (a1) to[out=60,in=120] (a4);
        \draw[-stealth, thick] (a2) to[out=0,in=180] (a3);
	\draw[-stealth, thick] (a2) to[out=-60,in=-120] (a3);
        \draw[-stealth, thick] (a3) to[out=0,in=180] (a4);
	\draw[-stealth, thick] (a3) to[out=-60,in=-120] (a5);
        \draw[-stealth, thick] (a4) to[out=0,in=180] (a5);
	\draw[-stealth, thick] (a4) to[out=60,in=120] (a5);
        \node[] at (0,0) {$G_{01}$};
        \vertex[rectangle, fill](b1) at (1.5,-0.15) {};
        \vertex[rectangle, fill](b2) at (2.5,-0.15) {};
        \vertex[rectangle, fill](b3) at (4,-0.35) {};
        \vertex[rectangle, fill](b4) at (4.5,0.15) {};
        \vertex[rectangle, fill](b5) at (2.5,0.5) {};
        \draw[red, thick] (b1) to[out=45,in=-135] (b5);
        \draw[red, thick] (b2) to[out=90,in=-90] (b5);
        \draw[red, thick] (b3) to[out=135,in=-45] (b5);
        \draw[red, thick] (b3) to[out=45,in=-135] (b4);
        \node[] at (6,0) {\textcolor{red}{$G_{01}^*$}};
	
\end{scope}

\begin{scope}[scale=1, xshift=0, yshift=60]

	\vertex[fill](a1) at (1,0) {};
	\vertex[fill](a2) at (2,0) {};
	\vertex[fill](a3) at (3,0) {};
	\vertex[fill](a4) at (4,0) {};
	\vertex[fill](a5) at (5,0) {};
	\draw[-stealth, thick] (a1) to[out=0,in=180] (a2);
        \draw[-stealth, thick] (a1) to[out=-60,in=-120] (a2);
        \draw[-stealth, thick] (a1) to[out=60,in=120] (a3);
        \draw[-stealth, thick] (a2) to[out=0,in=180] (a3);
	\draw[-stealth, thick] (a2) to[out=-60,in=-120] (a5);
        \draw[-stealth, thick] (a3) to[out=0,in=180] (a4);
	\draw[-stealth, thick] (a3) to[out=60,in=120] (a4);
        \draw[-stealth, thick] (a4) to[out=0,in=180] (a5);
	\draw[-stealth, thick] (a4) to[out=60,in=120] (a5);
        \node[] at (0,0) {$G_{10}$};
        \vertex[rectangle, fill](b1) at (1.5,-0.15) {};
        \vertex[rectangle, fill](b2) at (2,0.35) {};
        \vertex[rectangle, fill](b3) at (3.5,-0.5) {};
        \vertex[rectangle, fill](b4) at (3.5,0.15) {};
        \vertex[rectangle, fill](b5) at (4.5,0.15) {};
        \draw[red, thick] (b1) to[out=45,in=-135] (b2);
        \draw[red, thick] (b2) to[out=-45,in=135] (b3);
        \draw[red, thick] (b3) to[out=90,in=-90] (b4);
        \draw[red, thick] (b3) to[out=45,in=-135] (b5);
        \node[] at (6,0) {\textcolor{red}{$G_{10}^*$}};
 
\end{scope}

\begin{scope}[scale=1, xshift=0, yshift=0]

	\vertex[fill](a1) at (1,0) {};
	\vertex[fill](a2) at (2,0) {};
	\vertex[fill](a3) at (3,0) {};
	\vertex[fill](a4) at (4,0) {};
	\vertex[fill](a5) at (5,0) {};
	\draw[-stealth, thick] (a1) to[out=0,in=180] (a2);
        \draw[-stealth, thick] (a1) to[out=-60,in=-120] (a2);
        \draw[-stealth, thick] (a1) to[out=60,in=120] (a3);
        \draw[-stealth, thick] (a2) to[out=0,in=180] (a3);
	\draw[-stealth, thick] (a2) to[out=-60,in=-120] (a4);
        \draw[-stealth, thick] (a3) to[out=0,in=180] (a4);
	\draw[-stealth, thick] (a3) to[out=60,in=120] (a5);
        \draw[-stealth, thick] (a4) to[out=0,in=180] (a5);
	\draw[-stealth, thick] (a4) to[out=-60,in=-120] (a5);
        \node[] at (0,0) {$G_{11}$};
        \vertex[rectangle, fill](b1) at (1.5,-0.15) {};
        \vertex[rectangle, fill](b2) at (2,0.35) {};
        \vertex[rectangle, fill](b3) at (3,-0.35) {};
        \vertex[rectangle, fill](b4) at (4,0.35) {};
        \vertex[rectangle, fill](b5) at (4.5,-0.15) {};
        \draw[red, thick] (b1) to[out=45,in=-135] (b2);
        \draw[red, thick] (b2) to[out=-45,in=135] (b3);
        \draw[red, thick] (b3) to[out=45,in=-135] (b4);
        \draw[red, thick] (b4) to[out=-45,in=135] (b5);
        \node[] at (6,0) {\textcolor{red}{$G_{11}^*$}};
	
\end{scope}

\end{tikzpicture}
\end{center}
\caption{Truncated duals $G^*$ for $G\in\mathscr{F}_{4,3}$.}
\label{fig:n=4Duals}
\end{figure}

M\'esz\'aros, Morales, and Striker~\cite{meszaros-morales-striker} proved the following beautiful equivalence between flow and order polytopes in the strongly planar case.

\begin{theorem}[M\'esz\'aros, Morales, and Striker~\cite{meszaros-morales-striker}]\label{thm:flowandorderequivalence}
	Given a strongly planar DAG $G$, the flow polytope $\mathcal{F}_1(G)$ is integrally equivalent to the order polytope of the strongly planar poset dual to $G$.
	Conversely, given a strongly planar poset $P$, the order polytope $\mathcal{O}(P)$ is integrally equivalent to the flow polytope of the strongly planar DAG $G_P$.
\end{theorem}

As a corollary, we see that the volume of a flow polytope for a strongly planar DAG is equal to the volume of the order polytope for the strongly planar dual, and vice versa.
Further, we have the following result due to Stanley~\cite{stanley-posets}.

\begin{theorem}[Stanley~\cite{stanley-posets}]\label{thm:VolumeCountsLinearExt}
	For a poset $P$, the number of linear extensions of $P$ is counted by the volume of the order polytope $\mathcal{O}(P)$. That is
	\[
		e(P) = \vol \mathcal{O}(P) . 
	\]
\end{theorem}

We show below that DAGs in $\mathscr{F}_{n,3}$ are strongly planar, so we can apply the correspondence given by Theorem~\ref{thm:flowandorderequivalence} and Theorem~\ref{thm:VolumeCountsLinearExt} to our volume results from Section~\ref{sec:volume}. 
Our flow polytope volume inequalities from Theorem~\ref{thm:MainThm} thus extend to linear extension inequalities for the posets arising as truncated duals.
We begin by proving an observation about drawing DAGs in $\mathscr{F}_{n,3}$ with non-planar embeddings.

\begin{proposition} \label{prop:1cross}
    Let $G\in\mathscr{F}_{n,3}$. 
    If we draw the vertices in a row in so that their labels increase from left to right and draw non-spine edges as upward-facing semicircles (directed from smaller vertex labels to higher), $G$ has an edge crossing between vertices $k+1$ and $k+2$ if and only if its corresponding binary string $\b=b_1\cdots b_{n-2}$ has a $1$ in position $k$.
\end{proposition}

\begin{proof}
	Suppose $b_k = 0$. 
	Then by the bijection $\lambda$ in Theorem~\ref{thm:bijection}, the DAG $G$ has edges $e_{k+1}$ and $(k+1,k+2)$. 
	Further, Lemma~\ref{lemma:UniqueOverpass} ensures that only one other edge can pass over vertices $k+1$ and $k+2$. 
	But since the consecutive edges $e_{k+1}$ and $(k+1,k+2)$ fill all out-going stubs of $k+1$ and all in-going stubs of $k+2$, the unique edges passing over $k+1$ and $k+2$ are one in the same. 
	Drawing these three edges as upper semicircular arcs does not create any crossings. 
	
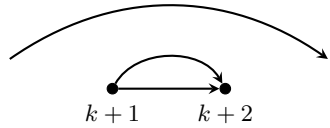
\begin{center}
\begin{tikzpicture}
\begin{scope}[xshift=0, yshift=0, scale=1]

	\node(a1) at (0.5,0.3) {};
	\vertex[fill,label=below:\footnotesize{$k+1$}](a2) at (2,0) {};
	\vertex[fill,label=below:\footnotesize{$k+2$}](a3) at (3.5,0) {};
	\node(a4) at (5,0.3) {};
        \draw[-stealth, thick] (a2) to[out=0,in=180] (a3);
        \draw[-stealth, thick] (a2) to[out=60,in=120] (a3);
        \draw[-stealth, thick] (a1) to[out=35,in=145] (a4);

\end{scope}
\end{tikzpicture}
\end{center}
	
	Conversely, suppose $b_k = 1$. 
	Denote the indices of the adjacent $1$s of $\b$ with $j$ and $l$, where $j:=0$ if $k$ is the first $1$ in $\b$ and $l:=n-1$ if $k$ is the last $1$ in $\b$. 
	So we have $j<k<l$. 
	Then the bijection $\lambda$ in Theorem~\ref{thm:bijection} gives edges $(j+1,k+2)$, $(k+1,l+2)$. 
	Note that $j<k<l$ ensures that neither $(j+1,k+2)$ nor $(k+1,l+2)$ can be a consecutive edge. 
	These two edges cross between $k+1$ and $k+2$ if drawn as upper semicircular arcs.

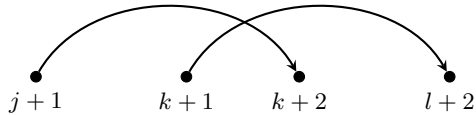
\begin{center}
\begin{tikzpicture}
\begin{scope}[xshift=0, yshift=0, scale=1]

	\vertex[fill,label=below:\footnotesize{$j+1$}](a1) at (0,0) {};
	\vertex[fill,label=below:\footnotesize{$k+1$}](a2) at (2,0) {};
	\vertex[fill,label=below:\footnotesize{$k+2$}](a3) at (3.5,0) {};
	\vertex[fill,label=below:\footnotesize{$l+2$}](a4) at (5.5,0) {};
        \draw[-stealth, thick] (a1) to[out=60,in=120] (a3);
        \draw[-stealth, thick] (a2) to[out=60,in=120] (a4);

\end{scope}
\end{tikzpicture}
\end{center}
\end{proof}

\begin{proposition} \label{prop:planar}
    Let $G\in\mathscr{F}_{n,3}$. 
    Then $G$ is strongly planar.
\end{proposition}

\begin{proof}
    Let $G\in\mathscr{F}_{n,3}$. 
    Then $G$ can be drawn with the vertices in a horizontal row with increasing labels,
    and the edges so that each spine edge $e_i$ is realized as a horizontal line from $i$ to $i+1$ and the non-spine edges are realized as upward-facing semi-circular arcs. 
    This drawing convention ensures that every directed edge $(i,j)$ of $G$ has the $x$-coordinate of $i$ strictly less than the $x$-coordinate of $j$,
    and that each edge is embedded in the plane as the graph of a piecewise differentiable function of $x$.
    Now we seek to transform some of the edges in a way that preserves these properties, but gives a planar embedding of $G$, in order to satisfy Definition~\ref{def:stronglyplanardag}.
    
    We claim we can decompose $E(G)$ into three paths from vertex $1$ to vertex $n+1$. 
    This is useful because a path in DAG embedded so that every edge $(i,j)$ has the $x$-coordinate of $i$ strictly less than the $x$-coordinate of $j$ cannot intersect itself. 
    The first path is the spine: $(e_1,\ldots, e_n)$. 
    The next path begins with the edge $(1,2)$, and the final path begins with the other edge leaving vertex $1$, $(1,j)$ say. 
    For convenience, we label the non-spine path beginning with $(1,2)$ as $A$ and the non-spine path beginning with $(1,2)$ as $B$. 
        To prove that paths $A$ and $B$ do indeed end at vertex $n+1$, consider the subgraph of $G$ obtained by deleting the spine. 
    Every interior vertex in the subgraph has in- and out-degree $1$. 
    Hence, each interior vertex $i$ can be in at most one path, and the path can continue past vertex $i$. 
        
    Now, if we reflect each edge in path $A$ below the spine, path $A$ does not intersect the spine or path $B$, since path $B$ lies entirely above the spine. 
    For the same reason, path $B$ does not intersect the spine.
    Hence we have a planar embedding of $G$. 
    Further, the reflection of the edges in path $A$ preserves the properties necessary for this embedding to be strongly planar. 
\end{proof}

\begin{example} \label{ex:planar1}
    Figure~\ref{fig:n=4Duals} provides a strongly planar embedding for each $G\in\mathscr{F}_{4,3}$.
\end{example}

\begin{example} \label{ex:planar2}
    Figure~\ref{fig:planar} provides a strongly planar embedding of $G_{1011}$.
\end{example}

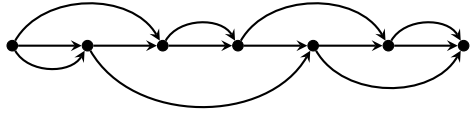
\begin{figure}
\begin{center}
\begin{tikzpicture}

    \begin{scope}[xshift=0, yshift=0, scale=1]
    
    	\vertex[fill](a1) at (1,0) {};
    	\vertex[fill](a2) at (2,0) {};
    	\vertex[fill](a3) at (3,0) {};
    	\vertex[fill](a4) at (4,0) {};
    	\vertex[fill](a5) at (5,0) {};
            \vertex[fill](a6) at (6,0) {};
            \vertex[fill](a7) at (7,0) {};
    	\draw[-stealth, thick] (a1) to[out=0,in=180] (a2);
            \draw[-stealth, thick] (a1) to[out=-60,in=-120] (a2);
            \draw[-stealth, thick] (a1) to[out=60,in=120] (a3);
            \draw[-stealth, thick] (a2) to[out=0,in=180] (a3);
    	\draw[-stealth, thick] (a2) to[out=-60,in=-120] (a5);
            \draw[-stealth, thick] (a3) to[out=0,in=180] (a4);
    	\draw[-stealth, thick] (a3) to[out=60,in=120] (a4);
            \draw[-stealth, thick] (a4) to[out=0,in=180] (a5);
    	\draw[-stealth, thick] (a4) to[out=60,in=120] (a6);
    	\draw[-stealth, thick] (a5) to[out=0,in=180] (a6);
    	\draw[-stealth, thick] (a5) to[out=-60,in=-120] (a7);
            \draw[-stealth, thick] (a6) to[out=0,in=180] (a7);
    	\draw[-stealth, thick] (a6) to[out=60,in=120] (a7);
    \end{scope}

\end{tikzpicture}
\end{center}
\caption{A strongly planar embedding for $G_{1011}$.}
\label{fig:planar}
\end{figure}

Our next goal is to identify the posets that are realizable as truncated duals of DAGs in $\mathscr{F}_{n,3}$.

\begin{proposition} \label{prop:PropertiesOfDual}
    Let $G\in\mathscr{F}_{n,3}$.
    Then the truncated dual $G^*$ 
    is a bipartite non-crossing tree on $n+1$ vertices. 
\end{proposition}

\begin{proof}
    Let $G\in\mathscr{F}_{n,3}$. 
    Let $\widehat G$ denote the proper dual of $G$, and let $G^*$ denote the truncated dual. 
    Recall that $G^*$ is the subgraph of $\widehat G$ obtained by deleting the vertex of $\hat G$ associated with the infinite face and all its incident edges. 
    By definition, $\widehat G$ is planar, and thus $G^*$ is non-crossing. 
    The planar embedding of $G$ described in Proposition~\ref{prop:planar} partitions the vertices of $G^*$ into those between the spine and path $A$, and those between the spine and path $B$.
    Hence $G^*$ is bipartite. 
    
    We use the Euler characteristic formula to verify the vertex count for $G^*$, and show that it has the requisite number of edges to be a tree. 
    By Proposition~\ref{prop:edgecount}, we know $G$ has $n+1$ vertices and $2n+1$ edges. 
    Then by the Euler characteristic formula, we have that $\widehat G$ has 
    \[ 
    	2-|V(G)|+|E(G)| = 2-(n+1) +(2n+1) = n+2
    \]
    vertices.
    Hence $G^*$ has $n+1$ vertices. 
    To count the edges in $G^*$, note that the number of edges in $G^*$ is equal to the number of edges in $\widehat G$ that are not incident to the vertex associated to the infinite face. 
    Using the planar embedding described in Proposition~\ref{prop:planar}, the edges in $\widehat G$ that are incident to the vertex associated to the infinite face correspond to the edges in $G$ that comprise paths $A$ and $B$.
    Since the spine and paths $A$ and $B$ partition $E(G)$, we have 
    \begin{align*}    
    	|E(G^*)| &= |E(G)| - (\# \text{ edges in path } A) - (\# \text{ edges in path } B) \\
	&= (\# \text{ edges in spine of } G) \\
	&= n .
    \end{align*}
    Thus $G^*$ is a graph with $n+1$ vertices and $n$ edges, so it is a tree.
\end{proof}

	Note that the drawing conventions described in the proof of Proposition~\ref{prop:planar} yield a bipartite non-crossing tree whose distinguished sets are separated by a horizontal line (the spine). 
	When we interpret one of these truncated duals as a Hasse diagram, we get a graded poset with two ranks (rank $0$ elements in the distinguished set below the spine and rank $1$ elements in the distinguished set above the spine).
	With this in mind, we use the terms ``bipartite non-crossing trees'' and ``rank $1$ posets" interchangeably throughout the remainder of the section, though in general not all rank 1 posets are bipartite non-crossing trees.

Next, we characterize the bipartite non-crossing trees on $n+1$ vertices that are realizable as duals of DAGs from $\mathscr{F}_{n,3}$, using the drawing convention described in the proof of Proposition~\ref{prop:planar}.  
We use observations due to M\'esz\'aros and Morales~\cite{meszaros-morales}, in whose work bipartite non-crossing trees arise from a completely different property of DAGs.

\begin{proposition} \label{prop:BNTcount}
    There are $2^{n-2}$ bipartite non-crossing trees on $n$ vertices.
\end{proposition}

\begin{proof}
	Bipartite non-crossing trees with $l$ left vertices and $r$ right vertices are in bijection with weak compositions of $r-1$ into $l$ parts~\cite{meszaros-morales}. 
	The bijection maps a bipartite non-crossing tree with left vertices $\{x_1,\ldots,x_l\}$ and $r$ right vertices to the weak composition $(\deg x_1-1,\ldots, \deg x_l-1)$ of $r-1$.
	It is known that the numbers of such weak compositions is $\displaystyle \binom{l+r-2}{l-1}$.
	So if $n=l+r$, the total number of bipartite non-crossing trees with $n$ vertices is
	\begin{align*}
		\sum_{l=1}^{n-1} \binom{n-2}{l-1} &= \sum_{i=0}^{n-2} \binom{n-2}{i} = 2^{n-2} .
	\end{align*}     
\end{proof}

\begin{proposition} \label{prop:BNTrealizable}
    Of the $2^{n-1}$ bipartite non-crossing trees on $n+1$ vertices, $2^{n-2}$ are realizable as truncated duals of DAGs in $\mathscr{F}_{n,3}$ when using the planar embedding convention described in Proposition~\ref{prop:planar}. 
    Specifically, the bipartite non-crossing trees that are realizable as truncated duals are those whose left-most vertex of rank $0$ has degree $1$. 
\end{proposition}

\begin{proof}
	Let $P$ be a bipartite non-crossing tree on $n+1$ vertices. 
	Label its rank $0$ elements as $\{x_1,\ldots, x_l \}$ and its rank $1$ elements as $\{y_1,\ldots, y_r \}$, with subscripts increasing from left to right. 
	Then the condition we prove is necessary and sufficient for realizability is $\deg x_1=1$. 
	
	Suppose $P$ is realized as the truncated dual of a DAG $G\in\mathscr{F}_{n,3}$, where $G$ is drawn using the conventions in Proposition~\ref{prop:planar}. 
	Note that $G$ has the edges $e_1$, $e_2$, $(1,2)$, $(2,j)$ and $(1,k)$, where $j,k\geq 3$. 
	Using the edge partition from the proof of Proposition~\ref{prop:planar}, this means $(1,2)$ and $(2,j)$ are in path $A$ and $(1,k)$ is in path $B$. 
	Then our drawing conventions give us the following regions in $G$: 
	\begin{enumerate}
		\item[(i)] The region bounded by $e_1$ and $(1,2)$, below the spine;
		\item[(ii)] The region bounded by $e_2,\ldots,e_{k-1}$ and $(2,j)$, below the spine; and
		\item[(iii)] The region bounded by $e_1,\ldots,e_{j-1}$ and $(1,j)$, above the spine.
                \end{enumerate}
                
	Taking the truncated dual of $G$ gives a vertex for each region. 
	The first two regions are rank $0$, so they are labelled $x_1$ and $x_2$.
	The third region is rank $1$, so it is labelled $y_1$. 
	Since the first two regions in $G$ share the edge $e_1$, we have an edge $(x_1,y_1)$ in $P$.
	Since the second two regions in $G$ share the edge $e_2$, we have an edge $(x_2,y_1)$ in $P$.
	Then $P$ cannot have any edge of the form $(x_1,y_i)$ for $i\geq 2$, as it would cross the edge $(x_2,y_1)$ or cross into the unbounded region, which does not correspond to a vertex in the truncated dual.
	Thus $\deg x_1=1$. 
	
	Conversely, suppose $\deg x_1=1$. 
	We construct a DAG $G\in\mathscr{F}_{n,3}$ so that $G^*=P$. 
	Begin with $n+1$ vertices and spine edges $e_1,\ldots, e_n$. 
	Associate each relation in $P$ with a spine edge, going from left to right.
	This is a well-defined association since $P$ is non-crossing. 
	For each rank $0$ element $x_i$ in $P$, let $e_j$ and $e_k$ denote the spine edges associated to the left-most and right-most relations of $x_i$. 
	Then draw the edge $(j,k+1)$ in $G$ as a downward-facing semicircular arc.
	For $x_1$, draw $(1,2)$ and for the right-most rank $0$ vertex $x_l$, draw $(j,n+1)$.
	These edges will form the path $A$ as described in Proposition~\ref{prop:planar}.
	Then for each rank $1$ element $y_i$ in $P$, let $e_j$ and $e_k$ denote the spine edges associated to the left-most and right-most relations of $y_i$. 
	Then draw the edge $(j,k+1)$ in $G$ as a upward-facing semicircular arc.
	For $y_1$, draw $(1,k+1)$ and for the right-most rank $1$ vertex $y_r$, draw $(j,n+1)$.
	These edges will form the path $B$ as described in Proposition~\ref{prop:planar}.
	By construction, we have $G^*=P$. 
	
	Given a bipartite non-crossing tree on $n+1$ vertices where $\deg x_1=1$, delete the vertex $x_1$ and edge $(x_1,y_1)$, which will give us a bipartite non-crossing tree on $n$ vertices. 
	Similarly, given any bipartite non-crossing tree on $n$ vertices, we can add a single vertex $x_0$ and edge $(x_0,y_1)$ to get a bipartite non-crossing tree on $n+1$ vertices where the left-most rank $0$ vertex has degree 1. 
	Thus bipartite non-crossing trees on $n+1$ vertices that are realizable as truncated duals of DAGs in $\mathscr{F}_{n,3}$ are in bijection with bipartite non-crossing trees on $n$ vertices.
	It follows from Proposition~\ref{prop:BNTcount} that there are $2^{n-2}$ such bipartite non-crossing trees. 
\end{proof}

In Section~\ref{sec:boolean}, we described a bijection between length $n-2$ binary strings and DAGs in $\mathscr{F}_{n,3}$. 
In this section, we describe (truncated) dualization, which is a bijection between DAGs in $\mathscr{F}_{n,3}$ and bipartite non-crossing trees on $n+1$ vertices whose left-most vertex of rank $0$ has degree $1$. 
This implies the existence of a composite bijection between length $n-2$ binary strings and bipartite non-crossing trees on $n+1$ vertices whose left-most vertex of rank $0$ has degree $1$. 
We now describe this bijection, and leave it as a straightforward exercise to verify that the bijection is correct. 

\begin{proposition} \label{prop:StringPosetBijection}
	Given a length $n-2$ binary string $\b$, we can construct a corresponding bipartite non-crossing tree $T_\b$ with $n+1$ vertices whose left-most vertex of rank $0$ has degree $1$ in the following way:
	\begin{enumerate}
		\item[(i)] Begin with an initial edge $e_0$ from rank $0$ to rank $1$. 
		\item[(ii)] From the rank $1$ vertex incident to $e_0$, construct a path of edges where each edge in the path corresponds to a $1$ in $\b$. 
		\item[(iii)] Attach a terminal edge after the final $1$-corresponding edge. 
		\item[(iv)] For each $0$ in $\b$, find the $1$s that surround it and attach an edge between the corresponding path edges in the poset. 
	\end{enumerate}
	
	Conversely, given a bipartite non-crossing tree $T$ with $n+1$ vertices whose left-most vertex of rank $0$ has degree $1$, we can construct a corresponding length $n-2$ binary string $\b_T$ in the following way:
	\begin{enumerate}
		\item[(i)] Trace a path from the left-most vertex of rank $0$ as far right as possible. 
		\item[(ii)] Ignoring the first and last edge in this path, write a $1$ in $\b_T$ for any remaining edges in the path. 
		\item[(iii)] For any edges not in the path, we place a $0$ in $\b_T$ between the $1$s that correspond to the path edges. 
	\end{enumerate}
\end{proposition}

\begin{example} \label{ex:StringPosetBijection}
The binary string $010010110010$ corresponds to the rank $1$ poset in Figure~\ref{fig:StringPosetBijection}. 
Note how the $1$s correspond to the maximal path (dashed edges), without the initial and terminal edges (dotted edges), and the $0$s correspond to the non-path edges (solid edges).
\end{example}

\begin{figure}
\begin{center}
\begin{tikzpicture}

\begin{scope}[scale=1, xshift=0, yshift=0]

        \node[] at (-3,0.5) {$0\underline{1}00\underline{1}0\underline{1}\underline{1}00\underline{1}0$};
        \node[] at (-1,0.5) {$\leftrightarrow$};
        \vertex[fill](a1) at (0,0) {};
        \vertex[fill](a2) at (1,0) {};
        \vertex[fill](a3) at (2,0) {};
        \vertex[fill](a4) at (3,0) {};
        \vertex[fill](a5) at (4,0) {};
        \vertex[fill](a6) at (4.66,0) {};
        \vertex[fill](a7) at (5.33,0) {};
        \vertex[fill](a8) at (6,0) {};
        \vertex[fill](b1) at (1,1) {};
        \vertex[fill](b2) at (1.66,1) {};
        \vertex[fill](b3) at (2.33,1) {};
        \vertex[fill](b4) at (3,1) {};
        \vertex[fill](b5) at (5,1) {};
        \vertex[fill](b6) at (6,1) {};
        \vertex[fill](b7) at (7,1) {};
        \draw[dotted, very thick] (a1) to (b1);
        \draw[thick] (a2) to (b1);
        \draw[dashed, thick] (a3) to (b1);
        \draw[thick] (a3) to (b2);
        \draw[thick] (a3) to (b3);
        \draw[dashed, thick] (a3) to (b4);
        \draw[thick] (a4) to (b4);
        \draw[dashed, thick] (a5) to (b4);
        \draw[dashed, thick] (a5) to (b5);
        \draw[thick] (a6) to (b5);
        \draw[thick] (a7) to (b5);
        \draw[dashed, thick] (a8) to (b5);
        \draw[thick] (a8) to (b6);
        \draw[dotted, very thick] (a8) to (b7);
	
\end{scope}
\end{tikzpicture}
\end{center}
\caption{The corresponding binary string and bipartite non-crossing tree of Example~\ref{ex:StringPosetBijection}.}
\label{fig:StringPosetBijection}
\end{figure}
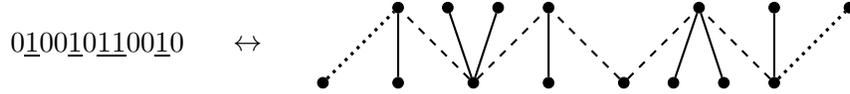

Next we describe how changing a $0$ to a $1$ affects the posets structure. 

\begin{definition} \label{def:flip}
	Let $T$ be a bipartite non-crossing tree on $n+1$ vertices whose left-most vertex of rank $0$ has degree $1$.
	Let $P$ be the path beginning at the left-most vertex of $T$ and ending at the right-most vertex. 
	For an edge $e_i$ not in the path $P$, we define a \emph{flip at edge $e_i$} to be the operation that obtains a new bipartite non-crossing tree on $n+1$ vertices whose left-most vertex of rank $0$ has degree $1$, $T'$, in the following way:
	\begin{enumerate}
		\item[(i)] Split $T$ into two connected components, the left ending with edge $e_i$ and the right beginning after edge $e_i$. 
		\item[(ii)] Take the (poset) dual of the right component.
		\item[(iii)] Reattach the right component to the left component to create a new tree $T'$ where $e_i$ is in the path beginning at the left-most vertex of $T'$ and ending at the right-most vertex. 
	\end{enumerate}
\end{definition}

\begin{example} \label{ex:flip}
	In Figure~\ref{fig:PosetInterchange}, we perform a flip at the fifth edge from the left. 
	First, we see our initial poset.
	Next, we separate it into two pieces after edge five. 
	Then we invert the right component. 
	Finally, we reattach the inverted piece so that edge five is in the path from left-most vertex to the right-most vertex.
\end{example}

The following proposition follows immediately from the bijection.

\begin{proposition} \label{prop:PosetInterchange} 
	Let $\b$ be a length $n-2$ binary string and let $T$ be the corresponding poset as described in Proposition~\ref{prop:StringPosetBijection}. 
	Then changing a $0$ in $\b$ to a $1$ is equivalent to performing a flip at the edge of $T$ that corresponds to the $0$ of $\b$. 
\end{proposition}

\begin{example} \label{ex:PosetInterchange}
In Figure~\ref{fig:PosetInterchange}, we see that the flip from Example~\ref{ex:flip} corresponds to changing $b_4$ from $0$ to $1$ in the binary string $010010110010$. 
\end{example}

\begin{figure}
\begin{center}
\begin{tikzpicture}

\begin{scope}[scale=1, xshift=0, yshift=150]

        \node[] at (-3,0.5) {$0\underline{1}00\underline{1}0\underline{11}00\underline{1}0$};
        \node[] at (-1,0.5) {$\leftrightarrow$};
        \vertex[fill](a1) at (0,0) {};
        \vertex[fill](a2) at (1,0) {};
        \vertex[fill](a3) at (2,0) {};
        \vertex[fill](a4) at (3,0) {};
        \vertex[fill](a5) at (4,0) {};
        \vertex[fill](a6) at (4.66,0) {};
        \vertex[fill](a7) at (5.33,0) {};
        \vertex[fill](a8) at (6,0) {};
        \vertex[fill](b1) at (1,1) {};
        \vertex[fill](b2) at (1.66,1) {};
        \vertex[fill](b3) at (2.33,1) {};
        \vertex[fill](b4) at (3,1) {};
        \vertex[fill](b5) at (5,1) {};
        \vertex[fill](b6) at (6,1) {};
        \vertex[fill](b7) at (7,1) {};
        \draw[dotted, thick] (a1) to (b1);
        \draw[  thick] (a2) to (b1);
        \draw[dashed, thick] (a3) to (b1);
        \draw[  thick] (a3) to (b2);
        \draw[  thick] (a3) to (b3);
        \draw[dashed, thick] (a3) to (b4);
        \draw[  thick] (a4) to (b4);
        \draw[dashed, thick] (a5) to (b4);
        \draw[dashed, thick] (a5) to (b5);
        \draw[  thick] (a6) to (b5);
        \draw[  thick] (a7) to (b5);
        \draw[dashed, thick] (a8) to (b5);
        \draw[  thick] (a8) to (b6);
        \draw[dotted, thick] (a8) to (b7);
	
\end{scope}

\begin{scope}[scale=1, xshift=0, yshift=100]

        \vertex[fill](a1) at (0,0) {};
        \vertex[fill](a2) at (1,0) {};
        \vertex[fill](a3) at (2,0) {};
        \vertex[fill](b1) at (1,1) {};
        \vertex[fill](b2) at (1.66,1) {};
        \vertex[fill](b3) at (2.33,1) {};
        \draw[  thick] (a3) to (b3);
        \draw[dotted, thick] (a1) to (b1);
        \draw[  thick] (a2) to (b1);
        \draw[dashed, thick] (a3) to (b1);
        \draw[  thick] (a3) to (b2);
	
\end{scope}
\begin{scope}[scale=1, xshift=60, yshift=100]

        \vertex[fill](a3) at (2,0) {};
        \vertex[fill](a4) at (3,0) {};
        \vertex[fill](a5) at (4,0) {};
        \vertex[fill](a6) at (4.66,0) {};
        \vertex[fill](a7) at (5.33,0) {};
        \vertex[fill](a8) at (6,0) {};
        \vertex[fill](b4) at (3,1) {};
        \vertex[fill](b5) at (5,1) {};
        \vertex[fill](b6) at (6,1) {};
        \vertex[fill](b7) at (7,1) {};
        \draw[dashed, thick] (a3) to (b4);
        \draw[  thick] (a4) to (b4);
        \draw[dashed, thick] (a5) to (b4);
        \draw[dashed, thick] (a5) to (b5);
        \draw[  thick] (a6) to (b5);
        \draw[  thick] (a7) to (b5);
        \draw[dashed, thick] (a8) to (b5);
        \draw[  thick] (a8) to (b6);
        \draw[dotted, thick] (a8) to (b7);
	
\end{scope}

\begin{scope}[scale=1, xshift=0, yshift=50]

        \vertex[fill](a1) at (0,0) {};
        \vertex[fill](a2) at (1,0) {};
        \vertex[fill](a3) at (2,0) {};
        \vertex[fill](b1) at (1,1) {};
        \vertex[fill](b2) at (1.66,1) {};
        \vertex[fill](b3) at (3,1) {};
        \draw[dotted, thick] (a1) to (b1);
        \draw[  thick] (a2) to (b1);
        \draw[dashed, thick] (a3) to (b1);
        \draw[  thick] (a3) to (b2);
        \draw[dashed, thick] (a3) to (b3);
	
\end{scope}
\begin{scope}[scale=1, xshift=60, yshift=50]

        \vertex[fill](a3) at (2,1) {};
        \vertex[fill](a4) at (3,1) {};
        \vertex[fill](a5) at (4,1) {};
        \vertex[fill](a6) at (4.66,1) {};
        \vertex[fill](a7) at (5.33,1) {};
        \vertex[fill](a8) at (6,1) {};
        \vertex[fill](b4) at (3,0) {};
        \vertex[fill](b5) at (5,0) {};
        \vertex[fill](b6) at (6,0) {};
        \vertex[fill](b7) at (7,0) {};
        \draw[dashed, thick] (a3) to (b4);
        \draw[  thick] (a4) to (b4);
        \draw[dashed, thick] (a5) to (b4);
        \draw[dashed, thick] (a5) to (b5);
        \draw[  thick] (a6) to (b5);
        \draw[  thick] (a7) to (b5);
        \draw[dashed, thick] (a8) to (b5);
        \draw[  thick] (a8) to (b6);
        \draw[dotted, thick] (a8) to (b7);
	
\end{scope}

\begin{scope}[scale=1, xshift=0, yshift=0]

        \node[] at (-3,0.5) {$0\underline{1}0\underline{11}0\underline{11}00\underline{1}0$};
        \node[] at (-1,0.5) {$\leftrightarrow$}; 
        \vertex[fill](a1) at (0,0) {};
        \vertex[fill](a2) at (1,0) {};
        \vertex[fill](a3) at (2,0) {};
        \vertex[fill](b4) at (4,0) {};
        \vertex[fill](b5) at (6,0) {};
        \vertex[fill](b6) at (7,0) {};
        \vertex[fill](b7) at (8,0) {};
        \vertex[fill](b1) at (1,1) {};
        \vertex[fill](b2) at (2,1) {};
        \vertex[fill](b3) at (3,1) {};
        \vertex[fill](a4) at (4,1) {};
        \vertex[fill](a5) at (5,1) {};
        \vertex[fill](a6) at (5.66,1) {};
        \vertex[fill](a7) at (6.33,1) {};
        \vertex[fill](a8) at (7,1) {};
        \draw[dotted, thick] (a1) to (b1);
        \draw[  thick] (a2) to (b1);
        \draw[dashed, thick] (a3) to (b1);
        \draw[  thick] (a3) to (b2);
        \draw[dashed, thick] (a3) to (b3);
        \draw[dashed, thick] (b3) to (b4);
        \draw[  thick] (a4) to (b4);
        \draw[dashed, thick] (a5) to (b4);
        \draw[dashed, thick] (a5) to (b5);
        \draw[  thick] (a6) to (b5);
        \draw[  thick] (a7) to (b5);
        \draw[dashed, thick] (a8) to (b5);
        \draw[  thick] (a8) to (b6);
        \draw[dotted, thick] (a8) to (b7);
	
\end{scope}

\end{tikzpicture}
\end{center}
\caption{The poset operation corresponding to changing $b_4$ to a $1$ in the binary string $010010110010$.}\label{fig:PosetInterchange}
\end{figure}
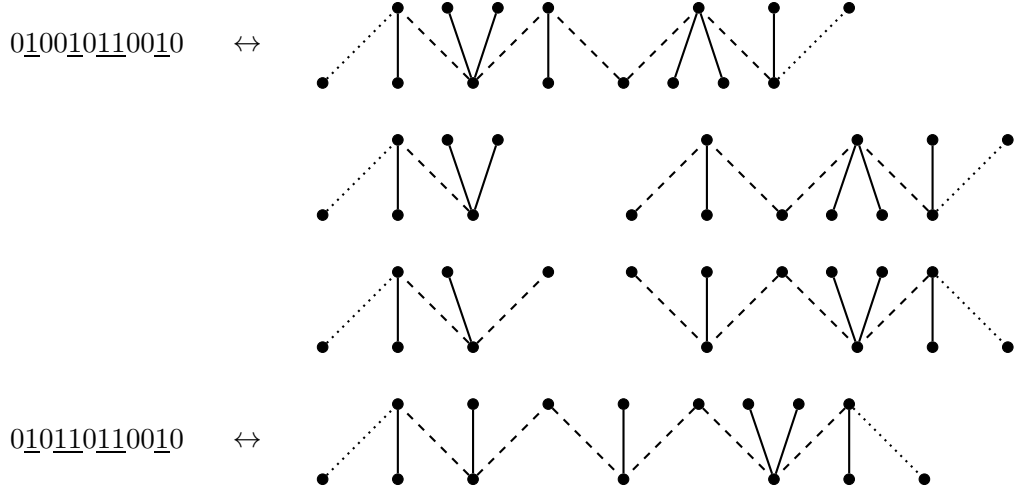

\begin{corollary} \label{cor:MainThmLinearExt}
    Let $T$ be a bipartite non-crossing tree poset on $n+1$ vertices whose left-most vertex of rank $0$ has degree $1$. 
    Let $P$ be the path beginning at the left-most vertex of $T$ and ending at the right-most vertex.
     If we perform a flip at an edge of $T$ not in $P$ to obtain a new tree $T'$, then
    \[
    	e(T')\leq e(T).
    \]
    \end{corollary}

\bibliographystyle{plain}
\bibliography{refs}

\begin{thebibliography}{10}

\bibitem{baldoni-vergne}
Welleda Baldoni and Mich\`ele Vergne.
\newblock Kostant partitions functions and flow polytopes.
\newblock {\em Transform. Groups}, 13(3-4):447--469, 2008.

\bibitem{caracolvolume}
Carolina Benedetti, Rafael~S. Gonz\'{a}lez~D'Le\'{o}n, Christopher R.~H.
  Hanusa, Pamela~E. Harris, Apoorva Khare, Alejandro~H. Morales, and Martha
  Yip.
\newblock The volume of the caracol polytope.
\newblock {\em S\'{e}m. Lothar. Combin.}, 80B:Art. 87, 12, 2018.

\bibitem{combflowmodel}
Carolina Benedetti, Rafael~S. Gonz\'{a}lez~D'Le\'{o}n, Christopher R.~H.
  Hanusa, Pamela~E. Harris, Apoorva Khare, Alejandro~H. Morales, and Martha
  Yip.
\newblock A combinatorial model for computing volumes of flow polytopes.
\newblock {\em Trans. Amer. Math. Soc.}, 372(5):3369--3404, 2019.

\bibitem{DKK}
Vladimir~I. Danilov, Alexander~V. Karzanov, and Gleb~A. Koshevoy.
\newblock Coherent fans in the space of flows in framed graphs.
\newblock In {\em 24th {I}nternational {C}onference on {F}ormal {P}ower
  {S}eries and {A}lgebraic {C}ombinatorics ({FPSAC} 2012)}, volume~AR of {\em
  Discrete Math. Theor. Comput. Sci. Proc.}, pages 481--490. Assoc. Discrete
  Math. Theor. Comput. Sci., Nancy, 2012.

\bibitem{generalpitmanstanley}
William~T. Dugan, Maura Hegarty, Alejandro~H. Morales, and Annie Raymond.
\newblock Generalized {P}itman-{S}tanley flow polytopes.
\newblock {\em S\'{e}m. Lothar. Combin.}, 89B:Art. 80, 12, 2023.

\bibitem{SIAMMCMC}
Bailey~K. Fosdick, Daniel~B. Larremore, Joel Nishimura, and Johan Ugander.
\newblock Configuring random graph models with fixed degree sequences.
\newblock {\em SIAM Review}, 60(2):315--355, 2018.

\bibitem{cyclicorderflow}
Rafael~S. Gonz\'{a}lez~D'Le\'{o}n, Christopher R.~H. Hanusa, Alejandro~H.
  Morales, and Martha Yip.
\newblock Column-convex matrices, {$G$}-cyclic orders, and flow polytopes.
\newblock {\em Discrete Comput. Geom.}, 70(4):1593--1631, 2023.

\bibitem{integerpointsflow}
Kabir Kapoor, Karola M\'{e}sz\'{a}ros, and Linus Setiabrata.
\newblock Counting integer points of flow polytopes.
\newblock {\em Discrete Comput. Geom.}, 66(2):723--736, 2021.

\bibitem{randomDAG2009}
Brian Karrer and M.~E.~J. Newman.
\newblock Random graph models for directed acyclic networks.
\newblock {\em Physical Review E}, 80(4), 2009.

\bibitem{lidskii}
B.~V. Lidski\u{\i}.
\newblock The {K}ostant function of the system of roots {$A\sb{n}$}.
\newblock {\em Funktsional. Anal. i Prilozhen.}, 18(1):76--77, 1984.

\bibitem{gtflow}
Ricky~I. Liu, Karola M\'{e}sz\'{a}ros, and Avery St.~Dizier.
\newblock Gelfand-{T}setlin polytopes: a story of flow and order polytopes.
\newblock {\em SIAM J. Discrete Math.}, 33(4):2394--2415, 2019.

\bibitem{flowdiagonalharmonics}
Ricky~Ini Liu, Alejandro~H. Morales, and Karola M\'{e}sz\'{a}ros.
\newblock Flow polytopes and the space of diagonal harmonics.
\newblock {\em Canad. J. Math.}, 71(6):1495--1521, 2019.

\bibitem{meszaros-morales}
Karola M\'{e}sz\'{a}ros and Alejandro~H. Morales.
\newblock Volumes and {E}hrhart polynomials of flow polytopes.
\newblock {\em Math. Z.}, 293(3-4):1369--1401, 2019.

\bibitem{meszaros-morales-striker}
Karola M\'{e}sz\'{a}ros, Alejandro~H. Morales, and Jessica Striker.
\newblock On flow polytopes, order polytopes, and certain faces of the
  alternating sign matrix polytope.
\newblock {\em Discrete Comput. Geom.}, 62(1):128--163, 2019.

\bibitem{morrisidentityflow}
Alejandro~H. Morales and William Shi.
\newblock Refinements and symmetries of the {M}orris identity for volumes of
  flow polytopes.
\newblock {\em C. R. Math. Acad. Sci. Paris}, 359:823--851, 2021.

\bibitem{stanley-posets}
Richard~P. Stanley.
\newblock Two poset polytopes.
\newblock {\em Discrete Comput. Geom.}, 1(1):9--23, 1986.

\bibitem{vonbell2024triangulations}
Matias von Bell, Benjamin Braun, Kaitlin Bruegge, Derek Hanely, Zachery
  Peterson, Khrystyna Serhiyenko, and Martha Yip.
\newblock Triangulations of flow polytopes, ample framings, and gentle
  algebras, 2024.
\newblock To appear in \textit{Selecta Mathematica}.

\bibitem{framedtriangulations}
Matias von Bell, Rafael~S. Gonz\'{a}lez~D'Le\'{o}n, Francisco~A.
  Mayorga~Cetina, and Martha Yip.
\newblock On framed triangulations of flow polytopes, the {$\nu$}-{T}amari
  lattice and {Y}oung's lattice.
\newblock {\em S\'{e}m. Lothar. Combin.}, 85B:Art. 42, 12, 2021.

\end{thebibliography}

\end{document}